\documentclass{article}
\usepackage{pgfplots}
\usepackage{fullpage}
\usepackage{amsmath}
\usepackage{amsthm}
\usepackage{amssymb}
\usepackage{url}
\usepackage{tikz-cd}
\usepackage{mathtools}
\usepackage{amsfonts}
\usepackage{graphicx}
\usepackage{verbatim}
\usepackage{url}
\usepackage{graphicx}
\usepackage[caption=false]{subfig}
\usepackage{float}
\usepackage[ruled]{algorithm2e}
\usepackage{tikz}
\usetikzlibrary{calc}
\usetikzlibrary{patterns}
\usepackage{tikz}
\tikzstyle{mybox} = [draw=black, fill=white,  thick,
rectangle, inner sep=10pt, inner ysep=20pt]
\tikzstyle{mybox} = [draw=black, fill=white,  thick,
rectangle, inner sep=2pt, inner ysep=2pt]
\usetikzlibrary{arrows}
\usetikzlibrary{arrows,chains,matrix,positioning,scopes}
\usetikzlibrary{arrows,shapes,positioning}
\usetikzlibrary{decorations.markings}
\tikzstyle arrowstyle=[scale=1]
\tikzstyle directed=[postaction={decorate,decoration={markings,
		mark=at position .65 with {\arrow[arrowstyle]{stealth}}}}]
\tikzstyle reverse directed=[postaction={decorate,decoration={markings,
		mark=at position .65 with {\arrowreversed[arrowstyle]{stealth};}}}]
\newcommand{\boundellipse}[3]
{(#1) ellipse (#2 and #3)
}

\newtheorem{theorem}{Theorem}
\newtheorem{lemma}{Lemma}
\newtheorem{corollary}{Corollary}
\newtheorem{proposition}{Proposition}
\theoremstyle{definition}
\newtheorem{definition}{Definition}
\newtheorem{remark}{Remark}
\newtheorem{example}{Example}

\usepackage{multirow}
\usepackage[utf8]{inputenc}

\begin{document}

\title{A Family of Iteration Functions for General Linear Systems}

\author{Bahman Kalantari\footnote{Emeritus Professor of Computer Science, Rutgers University, Piscataway, New Jersey (kalantari@cs.rutgers.edu)}}

\date{}
\maketitle

\begin{abstract}
We introduce innovative algorithms and analyze their theoretical complexity for computing exact or approximate (minimum-norm) solutions to a linear system $Ax=b$ or the {\it normal equation} $A^TAx=A^Tb$, where $A$ is an $m \times n$ real matrix of arbitrary rank. For the cases where $A$ is symmetric and positive semidefinite (PSD), we present more efficient algorithms.

First, we introduce the {\it Triangle Algorithm} (TA), a {\it convex-hull membership} algorithm designed for linear systems. Given the current iterate $b_k=Ax_k$ inside the ellipsoid $E_{A,\rho}=\{Ax: \Vert x \Vert \leq \rho\}$, TA iteratively either computes an improved approximation $b_{k+1}=Ax_{k+1}$ within $E_{A,\rho}$ or determines that $b$ lies outside. By adjusting $\rho$, TA can compute the desired approximations.

We then introduce a dynamic variant of TA, the {\it Centering Triangle Algorithm} (CTA), generating residual, $r_k=b -Ax_k$, using a remarkably simple iteration formula: $F_1(r)=r-(r^THr/r^TH^2r)Hr$, where $H=AA^T$. If $A$ is symmetric PSD, $H$ can be taken as $A$ itself. In a broader context, for each $t=1, \dots, m$, we derive $F_t(r)=r- \sum_{i=1}^t \alpha_{t,i}(r) H^i r$, where $\alpha_{t,i}$ are solutions to a $t \times t$ {\it auxiliary linear system}. Iterations of $F_t$ correspond to a Krylov subspace method with restart.

Let $\kappa^+(H)$ denote the ratio of the largest to smallest positive eigenvalues of $H$. We prove that, regardless of the invertibility of $H$, when the system $Ax=b$ is consistent, in $k=O({\kappa^+(H)}{t^{-1}} \ln \varepsilon^{-1})$ iterations of $F_t$, $\Vert r_k \Vert \leq \varepsilon$. Each iteration takes $O(tN+t^3)$ operations, where $N$ is the number of nonzero entries in $A$. When $H=AA^T$, the algorithm implicitly solves $AA^Tw=b$, all the while operating with the original variable $x$. It thus computes an approximate minimum-norm solution.

According to our complexity bound, by directly applying $F_t$ to the normal equation, one can achieve $\Vert A^TAx_k - A^Tb \Vert \leq \varepsilon$ in $O({\kappa^+(AA^T)}{t}^{-1} \ln \varepsilon^{-1})$ iterations. Therefore, a hybrid strategy is to initially apply CTA to an arbitrary linear system $Ax=b$ and, if $r_k$ does not decrease sufficiently, indicating possible inconsistency, switch CTA to the normal equation.

On the other hand, given arbitrary residual $r$, we can compute  $s$,  the degree of its  minimal polynomial  with respect to $H$  in $O(sN+s^3)$ operations.  Then $F_s(r)$  gives the minimum-norm  solution of $Ax=b$ or an exact solution of $A^TAx=A^Tb$.

The proposed algorithms are characterized by their simplicity of implementation and theoretical robustness. When applicable and advantageous, they can benefit from preconditioning techniques. We present sample computational results, comparing the performance of CTA with that of the CG and GMRES methods. These results support CTA as a highly competitive option. Practitioners experimentation with these algorithms may confirm their competitiveness among state-of-the-art linear solvers.
\end{abstract}



\newpage

\tableofcontents

\newpage

\section{Introduction} \label{sec1}

In this article, our primary focus is on computing a range of solutions for a linear system and its associated {\it normal equation}. Specifically, we aim to describe algorithms for computing exact or approximate solutions for a linear system  or its normal equation. These considerations also include exact or  approximate minimum-norm solutions. We investigate the linear equation $Ax=b$, where $A$ is an  $m \times n$ real matrix of arbitrary rank, and $b \in \mathbb{R}^m$ is a nonzero vector. As a result, the system $Ax=b$ can take on various forms, including being {\it underdetermined} ($m \leq n$), {\it overdetermined} ($m \geq n$), {\it consistent} (solvable), or {\it inconsistent} (unsolvable). We interchangeably use the terms {\it consistency} and {\it solvability}.

A least-squares solution seeks to minimize $\Vert Ax - b \Vert$, where $x \in \mathbb{R}^n$. An optimality condition reveals that a least-squares solution also satisfies the normal equation, $A^TAx =A^Tb$. The solvability of the normal equation can also be established through the well-known Farkas lemma, a duality concept in linear programming. According to this lemma, either $Ax=b$ has a solution, or there exists a vector $y$ such that $A^Ty=0$ and $b^Ty < 0$. Applying the lemma to the normal equation trivially proves its solvability.
Thus, even in cases where there is no solution to a linear system, a {\it certificate} of verification exists. Throughout we denote the minimum-norm solution of the normal equation as $x_*$. If $Ax=b$ is consistent, then $x_*$ also serves as a solution to it. The following are the definitions of the approximate solutions of interest

\begin{definition} \label{deffirst}
For a given real system $Ax=b$ and a tolerance value $\varepsilon \in (0,1),$ we define $x_\varepsilon$ as an $\varepsilon$-approximate solution if $\Vert Ax_\varepsilon - b \Vert \leq \varepsilon.$ Furthermore, $x_\varepsilon$ is an $\varepsilon$-approximate minimum-norm solution to $Ax=b$ if it meets two conditions: first, it is an $\varepsilon$-approximate solution, and second, it can be expressed as $x_\varepsilon = A^Tw$ for some $w \in \mathbb{R}^m.$
Similarly, we define approximate solutions for the normal equation by replacing $A$ with $A^TA$ and $b$ with $A^Tb$ in the above definitions.
\end{definition}

A problem closely linked to the solvability of $Ax=b$ is to determine if, for a given $\rho > 0$, $b$ resides within the ellipsoid $E_{A, \rho}= \{Ax: \Vert x \Vert \leq \rho\}$. This ellipsoid represents the image of the ball with radius $\rho$, centered at the origin, transformed by the linear map defined by $A$. Algorithmic attempt for solving this problem will play a crucial role in the development of novel iterative methods. Using the ellipsoid, another definition of approximate minimum-norm solution can be given.

\begin{definition} \label{minnormapprx}
We define $x_\varepsilon$ as an $\varepsilon$-approximate minimum-norm solution of $Ax=b$ if:

1. $\Vert Ax_\varepsilon - b \Vert \leq \varepsilon$.

2. There exists a positive number $\underline \rho$ such that $\underline \rho \leq \rho_\varepsilon = \Vert x_\varepsilon \Vert$ and
$b \not \in E^\circ_{A, \underline \rho}= \{Ax: \Vert x \Vert < \underline \rho\} $.

3. $\rho_\varepsilon - \underline \rho \leq \varepsilon$.

Similarly, we define an approximate minimum-norm solution to the normal equation.
\end{definition}

According to Condition 2,  $x_* \not \in E^\circ_{A, \underline \rho}$.
This implies when $Ax=b$ is solvable and $\Vert x_* \Vert \leq \Vert x_\varepsilon \Vert$, then $\Vert x_\varepsilon \Vert \leq \Vert x_* \Vert + \varepsilon$.
Thus  $x_\varepsilon$ and $x_*$ are close in terms of norm. Even when $Ax=b$ is unsolvable, if $x_\varepsilon$ as described above is computable, it gives a measure of proximity  of $Ax=b$ to feasibility.

The history of linear systems is rich and extensive, as solving linear equations stands as one of the most fundamental and practical challenges in various domains of scientific computing. A vast body of literature exists on this subject, encompassing the exploration of theoretical and numerical aspects, the development of algorithms, and their applications in solving a wide array of problems. Among the many well-known books related to this field are Atkinson \cite{Atkinson}, Bini and Pan \cite{Bini}, Golub and van Loan \cite{Golub}, Higham \cite{Higham}, Saad \cite{Saad}, and Strang \cite{Strang}. For those interested in the historical perspective, Brezinski et al. \cite{Brezinski} provides valuable insights.

Iterative methods provide crucial alternatives to direct methods and find extensive applications in solving problems that involve very large or sparse linear systems. Several major references on iterative methods include Barrett et al. \cite{Bar}, Golub and van Loan \cite{Golub}, Greenbaum \cite{Green}, Hadjidimos \cite{Hadjid}, Liesen and Strakos \cite{Liesen}, Saad \cite{Saad}, Saad and Schultz \cite{saad1986gmres}, Simoncini and Szyld \cite{Szyld}, van der Vorst \cite{van1}, Varga \cite{Varga}, and Young \cite{Young}. In these methods, the primary computational effort in each iteration typically involves matrix-vector multiplication. Consequently, if $N$ denotes the number of nonzero entries in the matrix, each iteration requires $O(N)$ operations, rendering iterative methods highly appealing for solving large and sparse systems.

Additionally, the speed of convergence, memory requirements, and stability serve as critical criteria for assessing the practicality and effectiveness of iterative methods. Another important factor is their suitability for parallelization, as explored in works such as Demmel \cite{Dem}, Dongarra et al. \cite{Don}, Duff and van der Vorst \cite{Duff}, van der Vorst \cite{van1}, and van der Vorst and Chan \cite{Van2}.

The convergence rates of iterative methods can often be substantially improved through the use of {\it preconditioning}. Preconditioning not only accelerates convergence but also enhances the accuracy of approximate solutions. However, it's important to note that preconditioning introduces its own computational costs. For a symmetric positive definite matrix $A$, preconditioning can be as simple as pre and post-multiplying by the square root of the diagonal matrix, composed of the diagonal entries of $A$. This provides an approximation to the optimal diagonal scaling, with an error factor proportional to the matrix dimension, as discussed in Braatz and Morari \cite{BRAATZ94}.

A more elaborate preconditioning method for symmetric positive definite matrices is the {\it incomplete Cholesky factorization}. Preconditioning of linear systems is a sophisticated and well-developed field in its own right. Notable works on this subject, along with a rich set of references, include  Benzi \cite{benzisur}, Benzi et al. \cite{benzi3},  Benzi et al. \cite{benzi2}, and Wathen \cite{Wathen15}. While preconditioning is often recommended in practice, if the condition number of the underlying matrix is not excessively large, an iterative method may not require preconditioning. There is a trade-off between improved convergence and increased computational time when preconditioning is applied. For guidance on the usage and effective utility of preconditioning, consult \cite{benzisur}, \cite{benzi3}, \cite{benzi2}, and \cite{Wathen15}.

Conditions that ensure the convergence of iterative methods for solving square systems typically encompass properties such as symmetry and positive definiteness, and diagonal dominance of the matrix $A$. Specialized algorithms have been developed for the case when $A$ is symmetric positive definite, notably including the {\it Steepest Descent} (SD) and the {\it Conjugate Gradient method} (CG). In exact arithmetic, CG theoretically converges in $n$ iterations, where $n$ is the dimension of $A$. However, it is widely recognized that in practice, these methods do not always converge in just $n$ iterations. In practical scenarios, our objective is to obtain solutions with a specified precision level $\varepsilon$. Iteration complexity bounds are typically derived in terms of $\varepsilon$, matrix dimensions, and input parameters like the condition number when $A$ is invertible.

In the context of iterative methods, CG treats each iteration as an opportunity to reduce the induced norm of the error $e_k=x_k-x_*$, denoted as $\Vert e_k \Vert_A=(e_k^TAe_k)^{1/2}$. It accomplishes this by a factor of $({\sqrt{\kappa} -1})/({\sqrt{\kappa}+ 1})$, where $\kappa$ represents the condition number of $A$. Consequently, $x_k$ converges to $x_*$ and, as a result, the corresponding residual $r_k=A(x_*-x_k)$ converges to zero.  For the Steepest Descent method, the above-bound substitutes $\sqrt{\kappa}$ with $\kappa$, leading to a slower worst-case convergence rate.

When $A$ is invertible but not positive definite, these methods can be applied to the equivalent system $A^TAx=A^Tb$. However, a general consensus suggests that solving the normal equations can be inefficient when $A$ is poorly conditioned, as discussed in Saad \cite{Saad}. The primary reason is that the condition number of $A^TA$ is the square of the condition number of $A$. For more details and references, refer to Hestenes and Stiefel \cite{Hesten52} for CG and Shewchuk \cite{10.5555/865018} for both SD and CG.

The GCR method, along with its variants, was developed by Eisenstat et al. \cite{EES83} and is applicable to solving the linear system $Ax=b$, where $A$ is an invertible $n \times n$ nonsymmetric matrix. However, it possesses a positive definite symmetric part, namely $(A+A^T)/2$. Similar to the CG method, GCR theoretically yields the exact solution in just $n$ steps. As shown by Eisenstat et al. \cite{EES83}, the residual in two consecutive iterations of GCR diminishes by a factor of $(1-\lambda^2_{\min}(\frac{1}{2}(A+A^T))/\lambda_{\max}(A^TA))$. Notably, when $A$ is symmetric positive definite, this reduction factor for GCR simplifies to $((\kappa^2(A) -1)/\kappa^2(A))$.

When dealing with a nonsymmetric but invertible matrix $A$, the development and analysis of iterative methods for convergence become significantly more complex compared to the PSD case of $A$. One of the popular iterative methods designed for solving general nonsymmetric linear systems with an invertible coefficient matrix is the {\it Bi-Conjugate Gradient Method Stabilized} (BiCGSTAB), as introduced by van der Vorst \cite{Vorst1}. This method falls under the category of Krylov subspace methods and is known for its faster and smoother convergence when compared to the {\it Conjugate Gradient Squared Method} (CGS).

The {\it generalized minimal residual method} (GMRES) is a highly popular approach for solving invertible linear systems with nonsymmetric matrix $A$. It was initially developed by Saad and Schultz \cite{saad1986gmres} and is actually an extension of the MINRES method introduced by Paige and Saunders \cite{Paige75}. While MINRES applies specifically when the matrix $A$ is invertible, symmetric, and indefinite, GMRES convergence and its analysis only requires the invertibility of $A$. For comprehensive insights into GMRES and related topics, one can refer to Chen \cite{Chen}, Liesen and Strakos \cite{Liesen}, van der Vorst \cite{Vorst1}, Saad \cite{Saad}, Saad and Schultz \cite{saad1986gmres}, and Simoncini and Szyld \cite{Szyld}. MINRES and GMRES can be applied to singular matrices too but their convergence becomes ambiguous.

GMRES initiates with a given $x_0$ and computes an initial residual $r_0=b-Ax_0$. The $k$-th {\it Krylov subspace} for a given $k \geq 1$ is defined as $K_k(A,r_0)= {\rm span}\{r_0, Ar_0, \dots, A^{k-1}r_0\}$. GMRES, instead of using this basis, employs the {\it Arnoldi iteration} to discover an orthonormal basis for $K_k(A_,r_0)$, given the current residual $r_{k-1}=b-Ax_{k-1}$. The residual norms exhibit non-increasing behavior, and ideally, the method terminates with $r_n=0$. However, according to a theorem by Greenbaum et al. \cite{GPS96}, there is no direct link between the eigenvalues of $A$ and the sequence of residuals. In the worst-case scenario, the residuals could remain the same as the initial one, except for the final one, which becomes zero. In practice, GMRES typically executes $k$ iterations, and the resulting approximate solution then serves as the initial guess for the subsequent $k$ iterations. Some approaches for solving nonsymmetric matrices involve these truncated cases of GMRES, which can yield peculiar outcomes, as discussed in Baker et al. \cite{Baker9}, Embree \cite{Embree}, and Saad and Schultz \cite{saad1986gmres}. Furthermore, when matrix $A$ is both nonsymmetric and singular, GMRES can face challenges and may require modifications. In such scenarios, convergence bounds in the analysis become less straightforward, as highlighted in Reichel and Ye \cite{Ye2005}.

A general consensus is that solving the normal equation, $A^TAx=A^Tb$, can be an inefficient
approach in the case when $A$ is poorly conditioned.  This view may extend to
other methods that attempt to symmetrize linear systems. Nevertheless, there are scenarios where such approaches can yield benefits. For instance, in the CGNE method, instead of directly solving $Ax=b$, we tackle $AA^Tw=b$, all while working with the original variable $x$ without the explicit computation of $w$. Quoting Saad \cite{Saad}, p. 259:

{\it However, the normal equations approach may be adequate in some situations. Indeed, there are even applications in which it is preferred to the usual Krylov subspace techniques.}

Remarks on CGNE can also be found in  Golub and van Loan \cite{Golub}, Higham \cite{Higham}.

While certain iterative methods may be viewed as impractical or susceptible to numerical instability, it's widely accepted that, even among well-established methods, there is no one-size-fits-all solution. There will always be instances where a specific class of iterative methods outperforms others. It is reasonable to acknowledge that, given the paramount importance of linear equations, theoretical research on this problem will persist indefinitely. This becomes particularly apparent as the need to solve increasingly larger systems arises in various realms of scientific computing.

\subsection{Summary of Our Results}

In this article, we present innovative algorithms along with their theoretical complexity bounds for computing exact or $\varepsilon$-approximate solutions to a general real system $Ax=b$ or its associated {\it normal equation} $A^TAx=A^Tb$, where $A$ is an $m \times n$ real matrix of arbitrary rank.
These considerations also include exact or  approximate minimum-norm solutions.
The algorithms implicitly tackle the {\it normal equation of the second type}, $AA^Tw=b$, eliminating the need to deal with $w$ or computing $AA^T$ directly. Consequently, the algorithms exclusively utilize matrix-vector multiplications. When $w$ satisfies $AA^Tw=b$, it's a well-known fact that $x=A^Tw$ represents the {\it minimum-norm solution} to $Ax=b$ (see Proposition \ref{minnorm}). As a result, when $Ax=b$ is solvable, the proposed algorithms yield an exact or approximate minimum-norm solution.

For situations where $A$ is symmetric positive semidefinite (PSD), we introduce more efficient algorithm variants, although in this case they may sacrifice the computation of exact or approximate minimum-norm solutions. In this context, the algorithms calculate the minimum-norm solution to $A^{1/2}u=b$, leveraging $A^{1/2}$ implicitly. In this scenario, the algorithms deliver either exact or approximate solutions to $Ax=b$. However, we also provide an algorithm capable of computing an $\varepsilon$-approximate minimum-norm solution to $Ax=b$ (or its normal equation) from any given $\varepsilon$-approximate solution. The results are organized as follows:

Section \ref{sec2} introduces the {\it Triangle Algorithm} (TA) for solving linear systems. TA is a specialized version of an algorithm originally developed for addressing the {\it convex hull membership problem} (CHMP) \cite{kalfull, kalcon}. CHMP entails determining whether a given point $p \in \mathbb{R}^m$ belongs to a given compact convex set in $\mathbb{R}^m$, defined as the convex hull of a given compact subset of $\mathbb{R}^m$. While we only state the general complexity of the algorithm for CHMP from \cite{kalcon}, we delve into adapting it for linear systems. Subsection \ref{sec2.1} discusses the properties of TA in the context of linear systems. We then introduce Algorithm \ref{2.1}, which, given $\varepsilon$ and $\varepsilon'$ in the range (0,1), computes either an $\varepsilon$-approximate solution for $Ax=b$ or an $\varepsilon'$-approximate solution for $A^TAx=A^Tb$. By adjusting $\varepsilon'$, TA generates the desired approximation for the system $Ax=b$ if it is solvable. If $Ax=b$ is unsolvable, TA computes an $\varepsilon'$-approximate solution for the normal equation. In its basic form, with a given $\rho>0$, TA tests whether the linear system $Ax=b$ has an approximate solution $x_\varepsilon$ with a norm bounded by $\rho$. It does so by verifying whether $b$ lies within the ellipsoid $E_{A,\rho}=\{Ax: \Vert x \Vert \leq \rho\}$,  the image of a ball with radius $\rho$ centered at the origin, transformed by the linear mapping defined by $A$. In each iteration, TA either computes progressively improving approximations $b_k=Ax_k \in E_{A,\rho}$ or produces a {\it witness} point $b'$ in $E_{A,\rho}$ that demonstrates $b \not \in E_{A,\rho}$, thereby establishing $\rho$ as a lower bound on the norm of any solution.

By adjusting $\rho$, TA can compute $\varepsilon$-approximate minimum-norm solutions for $Ax=b$ or $\varepsilon'$-approximate solutions for $A^TAx=A^Tb$. If $Ax=b$ is solvable, Algorithm \ref{2.1} computes an $\varepsilon$-approximate minimum-norm solution in $O(\Vert A\Vert^2\Vert x_* \Vert^2/\varepsilon^2))$ iterations of TA, with matrix-vector multiplications being the dominant computation component in each iteration. However, if $Ax=b$ is unsolvable and $\delta_*=\Vert Ax_*- b \Vert$, Algorithm \ref{2.1} computes an $\varepsilon'$-approximate solution for $A^TAx=A^Tb$ in $O(\Vert A \Vert^2 \Vert b \Vert^4/\delta_*^2 \varepsilon'^2)$ iterations, each dominated by the complexity of a matrix-vector multiplication.
Subsection \ref{sec2.2} introduces Algorithm \ref{2.2}, a nontrivial modification of Algorithm \ref{2.1}, tailored for cases where $A$ is PSD. This variant offers more efficient complexity per iteration. Finally, Subsection \ref{sec2.3} outlines Algorithm \ref{2.3}, which computes an $\varepsilon$-approximate minimum-norm solution from a given $\varepsilon$-approximate solution to a linear system. Notably, this algorithm can be employed to compute approximate minimum-norm solutions for the normal equation.

Section \ref{sec3} introduces the first-order {\it Centering Triangle Algorithm} (CTA), a method where residuals $r_k$ are generated using the iteration formula:
$F_1(r)=r-({r^THr}/{r^TH^2r})Hr$,
where $H=AA^T$. However, when $A$ is symmetric PSD, $H$ can be taken to be $A$ itself. The formal algorithm is detailed in Algorithm \ref{3.1}.

Subsection \ref{sec3.1} presents CTA as a geometric algorithm for solving linear systems. Subsection \ref{sec3.2} explores CTA from an algebraic perspective. Subsection \ref{sec3.3} interprets the iteration formula when $H=A$. Subsection \ref{sec3.4} describes CTA as a dynamic variant of TA, where each iterate $b_k=Ax_k$ becomes the center of a new ellipsoid with an appropriate $\rho_k$, and the corresponding residual $r_k=b - Ax_k$ is generated through the iterations of $F_1(r)$. This subsection establishes a strong connection between TA and CTA.

Subsection \ref{sec3.5} proves an essential auxiliary lemma on PSD matrices, which has independent significance. This lemma serves as the basis for analyzing the convergence of CTA and later becomes crucial for the convergence analysis of the more general case of CTA.

Subsection \ref{sec3.6} explores the relationship between the magnitudes of residuals in two consecutive iterations of $F_1(r)$. Subsection \ref{sec3.7} delves into the convergence properties of CTA.

Letting $\kappa^+ \equiv \kappa^+(H)$ represent the ratio of the largest to the smallest positive eigenvalues of $H$, we establish that when $Ax=b$ is solvable, starting with $r_0=b-Ax_0$, where $x_0=A^Tw_0 \in \mathbb{R}^n$ and $w_0 \in \mathbb{R}^m$, arbitrary, CTA converges in $k=O( \kappa^+ \ln \varepsilon^{-1})$ iterations of $F_1$ to achieve $\Vert r_k \Vert \leq \varepsilon$. Moreover, when $H=AA^T$, the corresponding $x_k$ becomes an $\varepsilon$-approximate minimum-norm solution.

When $H=AA^T$ is invertible, the sequence of $x_k$'s converges to the minimum-norm solution. Regardless of the consistency of $Ax=b$, in $k =O(\kappa^+/\varepsilon)$ iterations, CTA produces $b_k=Ax_k$ such that $\Vert A^T(b-b_k) \Vert = O(\sqrt{\varepsilon})$. As the normal equation is always consistent, employing the above-mentioned complexity result  directly to normal equation itself, CTA computes an $\varepsilon$-approximate least-squares solution in $O( \kappa^+(A^TA) \ln \varepsilon^{-1})$ iterations.

The above complexity results imply that when dealing with arbitrary linear systems $Ax=b$, without any knowledge of its consistency, applying CTA to $Ax=b$ is a useful first step. Then, if $r_k$ shows little improvement over several iterations, we apply CTA to the normal equation, starting with $A^Tr_k$. The use of $AA^T$ or $A^TA$ is implicit, and the operation count per iteration is $O(N)$, where $N$ represents the number of nonzero entries in $A$. These factors collectively establish the robustness of CTA via iterations of $F_1(r)$ in solving linear systems or their normal equations.

Section \ref{sec4} extends the iteration formula from $F_1(r)$ to {\it high-order} Centering Triangle Algorithms, where $F_t(r)=r- \sum_{i=1}^t \alpha_{t,i}(r) H^i r$ for $t=1, \dots, m$. The coefficients are selected to minimize the norm of $F_t(r)$. Iterations of $F_t$ correspond to a Krylov subspace method with restart.
To analyze this family of algorithms, we need to examine the minimal polynomial of the initial residual, $r_0$.

Subsection \ref{sec4.1} provides an algebraic definition of the CTA family and introduces an associated $t \times t$ linear system called the {\it auxiliary system}. Solving this system yields the coefficients of $F_t(r)$. Subsection \ref{sec4.2} investigates the relationship between the magnitudes of residuals in two consecutive iterations of the high-order CTA.

Subsection \ref{sec4.3} delves into the convergence properties of the high-order CTA. Due to its generality, the proof is more intricate than in the case of $t=1$. The subsection establishes more general complexity bounds, demonstrating that to achieve a given precision, the number of iterations of $F_t$ is at most $1/t$ times those of $F_1$. Additionally, it establishes a collective property of the CTA family iterations: given arbitrary $r_0=b-Ax_0$, evaluating the sequence $\{F_t(r_0)\}_{t=1}^m$ will yield an exact solution to $Ax=b$ or its normal equation. Specifically, if $s$ denotes the degree of the minimal polynomial of $r_0$ with respect to $H$, then $Ax=b$ is solvable if and only if $F{s}(r_0)=0$. Furthermore, exclusively $A^TAx=A^Tb$ is solvable if and only if $F_{s}(r_0) \not= 0$ and $A^T F_s(r_0)=0$. Moreover, the {\it point-wise orbit} $\{F_t(r_0)\}_{t=1}^s$ is computable in $O(Ns+s^3)$ operations. Therefore, without prior knowledge of $s$, evaluating the orbit solves either $Ax=b$ or $A^TAx=b$.

Subsection \ref{sec4.4} formally presents two algorithms based on the CTA family: Algorithm \ref{4.1} relies on iterations of $F_t$ for a particular $t$, while Algorithm \ref{4.2} utilizes the point-wise orbit. Subsection \ref{sec4.5} derives the iterates of CTA as applied to the normal equation and indicates their differences with the iterates of CTA as applied to $Ax=b$. Subsection \ref{sec4.6} provides an analysis of the space-time complexity of Algorithms \ref{4.1} and \ref{4.2}. Lastly, Subsection \ref{sec4.7} presents sample computational results obtained using CTA, Algorithms \ref{4.1} and \ref{4.2}, albeit with small values of $t$. The subsection also includes comparisons with CG and GMRES. We end with final remark.

The proposed methods serve as versatile algorithms for solving linear systems and their associated normal equations. They offer novel theoretical complexity bounds applicable to linear systems with arbitrary matrix coefficients. Additionally, these algorithms are straightforward to implement and robust, making them suitable for solving general linear systems, even without any preconditioning. However, when preconditioning the matrix is applicable and beneficial, they can be applied prior to executing CTA or TA.

In a separate article, we conduct a comprehensive computational study to assess the performance of these methods across a wide range of square and rectangular matrices. The study includes comparisons with state-of-the-art algorithms in the field.

\section{Triangle Algorithm (TA): A Convex Hull Membership Algorithm} \label{sec2}
In this section, we introduce the {\it Triangle Algorithm} (TA) for addressing the approximation problems detailed in Definitions \ref{deffirst} and \ref{minnormapprx}. TA is a specialized algorithm originating from a broader algorithm designed to tackle the {\it convex hull membership problem} (CHMP):

Given a compact subset $S$ in $\mathbb{R}^m$, a specific point $p \in \mathbb{R}^m$, and a user-defined parameter $\varepsilon \in (0, 1)$, the CHMP involves either computing a point $p' \in C=conv(S)$, the convex hull of $S$, such that $\Vert p' - p \Vert \leq \varepsilon$, or determining a hyperplane capable of separating $p$ from $C$.

The computational complexity per iteration of the Triangle Algorithm depends on the characteristics of the convex set $C$ being examined. In the worst-case scenario, the primary computational task in each iteration revolves around solving a linear programming problem over $C$. While this article primarily concentrates on the utilization of TA for solving linear systems, we will use the iteration complexity theorem stated for the general convex set $C$:

\begin{theorem} \label{thmzero} {\rm(\cite{kalcon})} Let $\rho_*$ be the diameter of $C$, $\varepsilon \in (0,1)$. Then

(i) If there exits $p' \in C$ such that $\|p' - p\| \leq \varepsilon$, the number of iterations of the Triangle Algorithm to find such a point is $O(\rho_*^2/\varepsilon^2)$.

(ii) If the distance between $p$ and $C$ is $\delta_* >0$, the number of iterations of the Triangle Algorithm to identify a hyperplane that separates $p$ from $C$  is $O(\rho_*^2/\delta_*^2)$. \qed
\end{theorem}

For details on the Triangle Algorithm for CHMP the reader may consult \cite{kalfull, KalY22, kalcon}.  These also include  extensions and  applications of the Triangle Algorithm in optimization. For applications of the Triangle Algorithm in computational geometry and machine learning, see \cite{AKZ}.

\subsection{Triangle Algorithm for Solving Linear Systems}
\label{sec2.1}
In the context of solving a linear system $Ax=b$, in the Triangle Algorithm the compact convex set $C$ corresponds to the ellipsoid $E_{A, \rho} = \{ Ax: \Vert x \Vert \leq \rho \}$. In this case, the point $p$ is simply $b$. This section outlines several properties of the general Triangle Algorithm, specifically tailored to determine if $b \in E_{A, \rho}$ within a tolerance of $\varepsilon$. Additionally, it introduces variations of this algorithm. While a preliminary version of the Triangle Algorithm for linear systems was presented in \cite{kallauzhang}, the results provided here are significantly more comprehensive and refined. The algorithms detailed in this section are designed to effectively address the specified approximation problems.

\begin{definition} \label{defpTA} Given $\rho >0$ and $x' \in \mathbb{R}^n$, if $b'=Ax' \in E_{A, \rho}$, such that $b' \not =b$, we refer to $b'$ as an {\it iterate}. A point $v \in E_{A, \rho}$ is defined as {\it pivot} at $b'$ (or simply a pivot) if it satisfies the equivalent conditions given by the inequalities:
\begin{equation} \label{pivotTA}
\|b' - v\| \geq \|b - v\| \quad \iff \quad
(b-b')^Tv \geq \frac{1}{2} \big (\Vert b \Vert^2 - \Vert b' \Vert^2 \big).
\end{equation}
 A pivot $v$ is called a {\it strict pivot} if
\begin{equation} \label{pivotsTA}
(b-b')^T(v-b) \geq 0.
\end{equation}
Notably, when all three points $b, b', v$ are distinct, the angle $\angle b'bv$ is at least $\pi/2$ (refer to Figure \ref{AAD}).
If no pivot exists at $b'$, we refer to $b'$ as a $b$-{\it witness} (or simply a {\it witness}).
\end{definition}

\begin{proposition}
An iterate $b' \in E_{A, \rho}$ is a witness if and only if the orthogonal bisecting hyperplane to the line segment between $b$ and $b'$ separates $b$ from $E_{A, \rho}$, thus indicating that $b$ does not belong to $E_{A, \rho}$. \qed
\end{proposition}

TA operates as follows: Given $\varepsilon \in (0,1)$, $\rho >0$ and  $x' \in \mathbb{R}^n$ such that $b'=Ax' \in E_{A, \rho}$, it first checks if $\Vert b - b' \Vert \leq \varepsilon$. If this condition is met, TA terminates.
If $b'$ is a witness, then $b \not \in E_{A, \rho}$. Otherwise, TA proceeds to compute a pivot $v_\rho=A x_\rho$ and the next iterate $b''$ as the nearest point of $b$ on the line segment between $b'$ and $v_\rho$ (see Figure \ref{AAD}).
Due to the convexity of $E_{A, \rho}$, $b''$ represents a point within the ellipsoid that is getting closer to $b$ compared to $b'$. Specifically,
\begin{proposition} \label{prop1TA} If $v_\rho=A x_\rho$ is a strict pivot,
\begin{equation} \label{alph}
    b'' = (1 - \alpha)b' + \alpha v_\rho, \quad  x'' = (1 - \alpha)x' + \alpha x_\rho, \quad  \alpha = {(b - b')^T(v_\rho - b')}/{\|v_\rho - b'\|^2}. \qed
\end{equation}
\end{proposition}
TA replaces $b'$ with $b''$ and $x'$ with $x''$. It then proceeds to repeat the aforementioned iteration. The correctness of TA is guaranteed by the following theorem, the general case of which is proved in \cite{kalcon}.

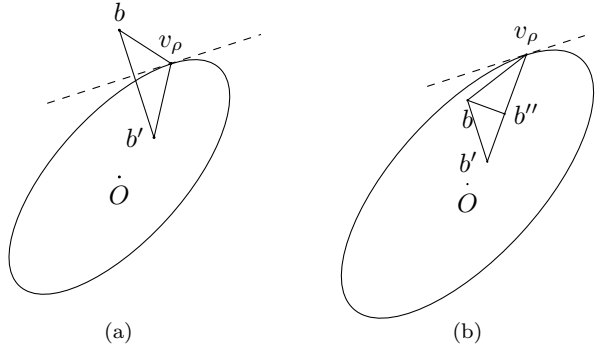
\begin{figure}[htpb]
	\centering
	\subfloat[] {
	\raisebox{2ex}{
	\begin{tikzpicture}[scale=.7]
	\begin{scope}
	\begin{scope}
	[rotate=47.5829]
	\draw \boundellipse{(0,0}{2.810}{1.210};
	\filldraw (0,0) circle (.5pt) node[below] {$O$};
	\end{scope}
	\filldraw (0,2.79) circle (.5pt)  node[above] {$b$};
	\filldraw (0.9815,2.158)  circle (.5pt)  node[above] {$v_\rho$};
	\filldraw (0.6580,0.75)  circle (.5pt)  node[left] {$b'$};
	\draw[dashed] (-1.37,1.4)--(2.69,2.709);
	\draw (0.6580,0.75)--(0.9815,2.158)--(0,2.79)--(0.6580,0.75);
	\end{scope}
	\label{1st}
	\end{tikzpicture}}}
	\qquad
	\subfloat[] {
	\begin{tikzpicture}[scale=.4]
	\begin{scope}
	\begin{scope}
	[rotate=47.5829]
	\draw  \boundellipse{(0,0}{5.620}{2.420};
	\filldraw (0,0) circle (.5pt) node[below] {$O$};
	\end{scope}
	\filldraw (0,2.79) circle (.5pt)  node[below] {$b$};
	\filldraw (1.963,4.315)  circle (.5pt)  node[above] {$v_\rho$};
	\filldraw (1.234,2.336)  circle (.5pt)  node[right] {$b''$};
	\filldraw (0.6580,0.75)  circle (.5pt)  node[left] {$b'$};
	\draw[dashed] (-1.34,3.25)--(3.9,4.94);
	\draw (1.963,4.315)--(1.234,2.336)--(0.6580,0.75)--(0,2.79)--(1.963,4.315);
	\draw (0,2.79)--(1.234,2.336);
	\end{scope}
	
	\end{tikzpicture}
	\label{2nd}
	}
	
	\caption{In Figure \ref{1st}, $v_\rho$ is not a strict pivot, proving $b$ is exterior to the ellipsoid.
In Figure \ref{2nd},  $b'$ admits a pivot, $v_\rho$, used to compute the next iterate $b''$  as the nearest point to $b$ on the line segment $b'v_\rho$. Geometrically, $v_\rho$ is found by moving the orthogonal hyperplane to line $bb'$ (dashed line) in the direction from $b'$ to $b$ until the hyperplane is tangential to the ellipsoid.}
	\label{AAD}
\end{figure}

\begin{theorem} \label{thm1TA} {\rm {(Distance Duality)}}
$b$ belongs to $E_{A, \rho}$ if and only if, for every $b' \in E_{A, \rho}$ (excluding $b$ itself), there exists a (strict) pivot $v \in E_{A, \rho}$. Conversely, $b$ does not belong to $E_{A, \rho}$ if and only if there exists a witness $b' \in E_{A, \rho}$. \qed
\end{theorem}

To complete the description of TA, we need to determine whether a (strict) pivot exists and, if so, how to compute it. If not, we need to compute a witness. This is achieved by computing:
\begin{equation} \label{pivottestTA}
v_\rho= {\rm argmax}\{c^Tx : Ax \in E_{A, \rho}\}, \quad c = A^T(b - b').
\end{equation}
Geometrically, to find $v_\rho$, consider the orthogonal hyperplane to the line segment $bb'$ and move it from $b'$ toward $b$ until it becomes tangential to the boundary of $E_{A, \rho}$. For an illustration, refer to Figure \ref{AAD}. It follows from (\ref{pivotTA}) that $v_\rho$ is either a pivot or $b'$ is a witness. In fact, if $b \in E_{A, \rho}$, then $v_\rho$ is a strict pivot.

Testing whether $b \in E_{A, \rho}$ is a crucial subproblem in solving $Ax=b$ or the normal equation. In the following, we first state a theorem that concerns the solvability of $Ax=b$ in $E_{A, \rho}$, then describe Algorithm \ref{2.1} and state its complexity.

\begin{theorem} \label{prop0TA} Given $x' \in \mathbb{R}^n$, where $\|x'\| \leq \rho$, let $b' = Ax'$. Assume $b' \not =b$. Let
$c= A^T(b-b')$.

(1) If $c=0$, $x'$ is a solution to the normal equation, $Ax=b$ is unsolvable and $b'$ is witness.

(2) Suppose $c \not =0$. Then
\begin{equation}  \label{vr}
v_\rho= \rho {Ac}/{\Vert c \Vert}= {\rm argmax}\{(b-b')^Ty : y \in E_{A, \rho}\}, \quad i.e. \quad \max\{c^Tx: Ax\in  E_{A, \rho} \} = c^T v_\rho= \rho \Vert c \Vert.
\end{equation}

(3) $v_\rho$ is a strict pivot if and only if
\begin{equation}  \label{eqlem2}
\rho \Vert c \Vert \geq (b-b')^Tb.
\end{equation}

(4)  If $b' \in E_{A,\rho}$ is a witness and $x_*$ the minimum-norm solution of the normal equation, then
\begin{equation}
\frac{(b-b')^Tb}{\Vert A^T(b-b')} < \Vert x_* \Vert.
\end{equation}
\end{theorem}
\begin{proof}
(1):  Since $c=0$, $A^TAx'=A^Tb$. Since any solution to the normal equation is the minimizer of $\Vert Ax- b \Vert$ and $b' \not =b$, $Ax=b$ is unsolvable.  If $b'$ is not a witness, from Theorem \ref{thm1TA} there exists a pivot in $E_{A,\rho}$ and this in turn implies there exists  $b'' \in E_{A,\rho}$ such that $\Vert b''-b \Vert < \Vert b'-b \Vert$. However, as $x'$ is a solution to the normal equation, $\Vert b'-b \Vert$ is minimum, leading to a contradiction.

(2): The Largarnge multiplier condition implies
$x_\rho= \rho \frac{c}{\Vert c \Vert}= {\rm argmax}\{c^Tx : \Vert x \Vert \leq  \rho\}$. Thus $v_\rho =A x_\rho$ and $\max\{c^Tx: Ax\in  E_{A, \rho} \} = c^T v_\rho= \rho \Vert c \Vert$.

(3): Substituting $v_\rho$ for $v$ in (\ref{pivotsTA}) implies (\ref{eqlem2}).

(4):  Suppose $b' \in E_{A, \rho}$ is a witness. Let $\rho'$ be any value between $\rho$ and ${(b-b')^Tb}/{\Vert A^T(b-b') \Vert}$.  From (2), $v_{\rho'}=\rho'Ac/\Vert c \Vert$.  Since $b'$ is a witness, from (3),  $\rho \Vert A^T(b-b') \Vert < (b-b')^Tb$. This implies $\rho' <  \Vert x_* \Vert$.
\end{proof}
Algorithmically, an important aspect of the Triangle Algorithm relies on the fact that the optimization of a linear function over an ellipsoid can be computed efficiently.

The Triangle Algorithm for the general matrix equation $Ax=b$ is described in Algorithm \ref{2.1}. Given a parameter $\rho > 0$, it tests if $b$ lies in the ellipsoid $E_{A, \rho}$. If $\rho < \Vert x_* \Vert$ (where $x_*$ is the minimum-norm solution of the normal equation), it computes a witness. Thus, each time a given $\rho$ results in a witness $b'$, using Theorem \ref{prop0TA} parts (3) and (4), we increase $\rho$ to the least amount where a pivot exists in the corresponding enlarged ellipsoid. However, we also want to ensure that $\rho$ increases sufficiently to eventually catch up to $\Vert x_*\Vert$. To achieve this, we take the new $\rho$ to be the maximum of $2 \rho$ and $(b-b')^Tb/\Vert c \Vert$. We justify that this value will work.

This algorithm plays a crucial role in determining whether $Ax=b$ is solvable or not. It iteratively adjusts the parameter $\rho$ to explore the possibility of solvability while efficiently checking if $b \in E_{A, \rho}$.

\begin{algorithm}[!htb]
\scriptsize
\scriptsize
\SetAlgoNoLine
\KwIn{$A \in \mathbb{R}^{m \times n}$, $b \in \mathbb{R}^m$, $b \not =0$, $\varepsilon, \varepsilon' \in (0,1)$ }
$\rho \gets 0$, $x' \gets 0$, $b' \gets 0$.

\While{$(\|b - b'\| > \varepsilon)$ $\wedge$  $( \Vert A^T(b-b') \Vert > \varepsilon')$}{$c=A^T(b-b')$,
$v_\rho \gets \rho
A{c}/{\|c\|}$.

  \lIf{$ \rho \Vert c \Vert \geq (b-b')^Tb$} { $\alpha \gets (b - b')^T(v_\rho - b')/\|v_\rho - b'\|^2$,

   $b' \gets (1-\alpha) b' + \alpha v_\rho$, \quad $x' \gets (1-\alpha) x' +  \alpha {\rho c}/{\Vert c \Vert}$}
   \Else{{$b'=Ax'$ is a witness,}
{$\rho \gets \max \{2 \rho,  {(b-b')^Tb }/{\Vert c \Vert}\}$}}}
    \caption{(TA) Computes $\varepsilon$-approximate solution of $Ax=b$ or $\varepsilon'$-approximate soln of $A^TAx=A^Tb$.} \label{2.1}
\end{algorithm}

\begin{remark}  Initially in Algorithm \ref{2.1} we may choose $\varepsilon'=\varepsilon$. If $Ax=b$ admits an $\varepsilon$-approximate solution,   the algorithm may still terminate with an $\varepsilon'$-approximate solution to the normal equation. In such case, we can halve the value of $\varepsilon'$ and restart the algorithm with the current approximate solution. This process will eventually cause the algorithm to terminate with one approximate solution or another.
\end{remark}

\begin{remark}
In the terminology of iterative solvers, Algorithm \ref{2.1}, as well as the remaining ones in this section, are based on two-term recurrences that generate vector sequences $x'$ and $b'$. Here, $x'$ represents the approximate solution of the given linear system. However, while $b' = Ax'$, the algorithm first computes the new $b'$ as a convex combination of the previous $b'$ and the pivot $v_\rho$. Then, it computes $x'$ using the same convex combination of the previous approximation and the $x$ vector that defines the pivot. If $A$ has $N$ nonzero entries, each iteration takes $O(N)$ operations.
\end{remark}
To derive the iteration complexity of Algorithm \ref{2.1}, we refer to the iteration complexity bound of the general Triangle Algorithm for solving CHMP, as taken from \cite{kalcon}, and stated as Theorem \ref{thmzero}.

\begin{theorem} \label{thm2pTA} {\rm {(TA Iteration Complexity Bounds)}}
Algorithm \ref{2.1} either terminates with $x'$ such that $\Vert Ax' - b \Vert \leq \varepsilon$ or $\Vert A^TAx'-A^Tb \Vert \leq \varepsilon'$. Furthermore:

(i) If $\Vert Ax' - b \Vert \leq \varepsilon$, then $x'=A^Tw'$ for some $w'$, so in fact, $x'$ is an $\varepsilon$-approximate minimum-norm solution to $Ax=b$.

(ii) If $Ax=b$ is solvable, by selecting $\varepsilon'$ small enough, the algorithm produces $x'$ satisfying (i) in $O(\Vert A \Vert^2 \Vert x_* \Vert^2/\varepsilon^2)$ iterations.

(iii) Suppose $Ax=b$ is not solvable. Let $\delta_*=\Vert Ax_* - b \Vert$. Let $\rho_{\varepsilon'} = {\Vert b \Vert^2}/{\varepsilon'}$. Suppose the value of $\rho$ in the while loop satisfies $\rho \geq \rho_{\varepsilon'}$. Then, if Algorithm \ref{2.1} does not compute an $\varepsilon$-approximate solution $x_\varepsilon$ of $Ax=b$ satisfying $\Vert x_\varepsilon \Vert \leq \rho$, it computes an $\varepsilon'$-approximate solution $x_{\varepsilon'}$ of $A^TAx=A^Tb$ in $O(\Vert A \Vert^2 \rho_{\varepsilon'}/\delta_*^2)=O(\Vert A \Vert^2 \Vert b \Vert^4/\delta_*^2 \varepsilon'^2)$ iterations, satisfying $\Vert x_{\varepsilon'} \Vert \leq \rho$.
\end{theorem}

\begin{proof} The first statement in the theorem follows from Theorem \ref{prop0TA} and the description of  Algorithm \ref{2.1}. We prove (i)-(iii).

(i): First, note that in the algorithm, if $x'=Aw'$ for some $w'$, then in the next iteration, the new $x'$ also satisfies the same. Since we start with $x'=0$, throughout the iterations of TA, $x'=Aw'$ for some $w'$. To argue that $x'$ is a claimed approximate minimum-norm solution, we use the known fact that if $Ax_*=b$ but $x_*=A^Tw_*$, then $x_*$ is the minimum-norm solution to $Ax=b$ (see Proposition \ref{minnorm} for a proof).

(ii): Since the initial value of $\rho=0$ and $b'=0$, the algorithm will produce a witness, and the next value of $\rho$ is $\Vert b \Vert^2/\Vert A^T b\Vert$. So, this value at least gets doubled each time and eventually exceeds $\Vert x_* \Vert$, and at that point, Algorithm \ref{2.1} produces an $\varepsilon$-approximate solution of $Ax=b$ or an $\varepsilon'$-approximate solution of the normal equation. If the latter is the case, we reduce $\varepsilon'$ and continue, say by halving it. By repeating this, TA will eventually compute an $\varepsilon$-approximate solution of $Ax=b$. This follows from the fact that if $Ax=b$ is solvable, there exists $\varepsilon'$ so that when $\Vert A^TA x' - A^Tb \Vert \leq \varepsilon'$, then $\Vert A x' - b \Vert \leq \varepsilon$. In other words, if $Ax=b$ is solvable, by halving $\varepsilon'$ and running Algorithm \ref{2.1}, we will eventually compute an $\varepsilon$-approximate solution of $Ax=b$. Now substituting this into part (i) in Theorem \ref{thmzero} and using that the diameter of $E_{A, \rho}$ is at most $\Vert A \Vert \rho$, we get the proof of the bound on the number of iterations to get $\Vert b - b' \Vert \leq \varepsilon$.

(iii): Suppose for a given $\rho  \geq \rho_{\varepsilon'}$ the algorithm provides a witness $b' \in E_{A,\rho}$. Then by part (4) of Theorem \ref{prop0TA},
Cauchy-Schwarz inequality and the fact that the algorithm starts with $b'=0$, we have
\begin{equation} \label{bound1TA}
\rho  \Vert c \Vert < (b-b')^T b  \leq \Vert b- b' \Vert \cdot \Vert b \Vert \leq \Vert b \Vert^2.
\end{equation}
Since we must have $\Vert c \Vert \geq \varepsilon'$, (\ref{bound1TA}) implies $\rho < \rho_{\varepsilon'}$, a contradiction. Now using part (ii) in Theorem \ref{thmzero} we get the claimed complexity bound.
\end{proof}

\subsection{Triangle Algorithm for Symmetric PSD Linear System} \label{sec2.2}

Here, we develop a more efficient version of Algorithm \ref{2.1} for the case when $A$ is symmetric and positive semi-definite (PSD). First, we state a useful known result, and for the sake of completeness, provide a brief proof.
\begin{proposition}  \label{minnorm} $Ax=b$ is solvable if and only if $AA^Tw=b$ is solvable. Moreover, if $w$ is any solution to $AA^Tw=b$, then $x_*=A^Tw$ is the minimum-norm solution to $Ax=b$.  In particular, if $A$ is symmetric PSD, and $A^{1/2}$ is its square-root, then $A^{1/2}y=b$ is solvable if and only if $Ax=b$ is solvable. Moreover, if $x$ is any solution to $Ax=b$, then $y_*=A^{1/2}x$ is the minimum-norm solution to $A^{1/2}y=b$.
\end{proposition}
\begin{proof}  Suppose $Ax=b$ is solvable. Then it has a solution of minimum norm, $x_*$.  On the other hand, $x_*$ is the optimal solution of the convex programming problem $\min \{ \frac{1}{2}x^Tx:  Ax=b \}$. According to the Lagrange multiplier condition, $x_*$ exists and satisfies $x_*= A^Tw$,  where $w$ is the vector of multipliers.
This implies $AA^Tw=Ax_*=b$. Conversely, if $AA^Tw=b$ then $\widehat x =A^Tw$ satisfies the Lagrange multiplier condition and by the convexity of the underlying optimization problem, $\widehat x=x_*$.  In the special case where $A$ is symmetric PSD, $A^{1/2}$ exists, hence the proof is straightforward.
\end{proof}
We utilize the above results to outline a version of Algorithm \ref{2.1} for solving $Ax=b$, where $A$ is symmetric and PSD. The key idea is to approach the problem of solving $Ax=b$ as if we were computing the minimum-norm solution of $A^{1/2}y=b$. However, we will do this implicitly, without the necessity of computing $A^{1/2}$ or the iterate $y'$ that approximates the solution of the linear system or its normal equation. This approach offers an advantage in that each iteration will only require one matrix-vector computation, in contrast to the two such operations required in the original algorithm.

Let's consider testing the solvability of $A^{1/2}y=b$ within a fixed ellipsoid using Algorithm \ref{2.1}. Specifically, given $\rho >0$, we aim to determine if $b \in E_{A^{1/2}, \rho}= \{A^{1/2} y: \Vert y \Vert \leq \rho\}$.

For a given $b'=A^{1/2}y' \in E_{A^{1/2}, \rho}$, the corresponding values of $c$ and $v_\rho$ in Algorithm \ref{2.1}, denoted as $\widehat c$ and $\widehat v_\rho$, are calculated as follows:
\begin{equation}
\widehat c=A^{1/2} (b-b'), \quad \widehat v_{\rho} = \rho A^{1/2}{\widehat c}/{\Vert \widehat c \Vert}.
\end{equation}
It's important to note that while $\widehat c $ is expressed in terms of $A^{1/2}$, $\Vert \widehat c \Vert= \sqrt{(b-b')^TA(b-b')}$, and $\widehat v_\rho = A(b-b')/\Vert \widehat c \Vert$. The corresponding step-size is computed as:

Note that while $\widehat c $ is expressed in terms of $A^{1/2}$, $\Vert \widehat c \Vert= \sqrt{(b-b')^TA(b-b')}$ and $\widehat v_\rho = A(b-b')/\Vert \widehat c \Vert$.  The corresponding step-size  is
\begin{equation}
\widehat \alpha = (b - b')^T(\widehat v_\rho - b')/\|\widehat v_\rho - b'\|^2.
\end{equation}
These calculations lead to the updated values of $b''=A^{1/2}y''$, where
\begin{equation}
b''=(1- \widehat \alpha) b' + \widehat \alpha \widehat v_\rho, \quad y''= (1-\widehat \alpha) y' +  \widehat \alpha \rho {A^{1/2}(b-b')}/{\Vert \widehat c \Vert}.
\end{equation}
Now, let's suppose that $y'$ in $b'=A^{1/2}y'$ is given as $y' =A^{1/2}x' \in E_{A^{1/2}, \rho}$ for some $x'$, which is an approximate solution to $Ax=b$. This suggests that by factoring $A^{1/2}$ from the expression for $y''$, we get $y''=A^{1/2}x''$, where
\begin{equation}
x''=(1- \widehat \alpha) x'+ \widehat \alpha \rho {(b-b')^T}/{\Vert \widehat c \Vert}.
\end{equation}
The final modification to Algorithm \ref{2.1} with respect to solving $A^{1/2}y=b$ is to note that the condition in the while loop, $\Vert A^T(b-b') \Vert > \varepsilon'$, reduces to $\Vert A^{1/2}(b-b' \Vert = \sqrt{(b-b')^TA(b-b')} > \varepsilon'$.  These modifications result in Algorithm \ref{2.2}.

\begin{algorithm}[!htb]
\scriptsize
\SetAlgoNoLine
\KwIn{$A \in S_+^{m \times m}$ ($m \times m$ symmetric PSD), $b \in \mathbb{R}^m$, $b \not =0$, $\varepsilon, \varepsilon' \in (0,1)$.}
$\rho \gets 0$, $x' \gets 0$, $b' \gets 0$.

 \While{$(\|b - b'\| > \varepsilon)$ $\vee$ $\sqrt{(b-b')^TA(b-b')} > \varepsilon' )$}{
 $\Vert \widehat c \Vert \gets \sqrt{(b-b')^TA(b-b')}$,
  $\widehat v_{\rho} = \rho {A(b-b')}/{\Vert \widehat c \Vert}$

  \lIf{$ \rho \Vert \widehat c \Vert \geq (b-b')^Tb$} { $\widehat \alpha \gets (b - b')^T(\widehat v_\rho - b')/\|\widehat v_\rho - b'\|^2$,

  $b' \gets (1-\widehat \alpha) b' + \widehat \alpha \widehat v_\rho$, \quad $x' \gets (1-\widehat \alpha) x' +  \widehat \alpha \rho {(b-b')}/{\Vert \widehat c \Vert} $}

   \Else{{$b'=Ax'$ is a witness,}
   {$\rho \gets \max \{2 \rho,  {(b-b')^Tb }/{\Vert c \Vert}\}$}}}
    \caption{(TA) Computes $\varepsilon$-approximate solution of $Ax=b$ or $\varepsilon'$-approximate solution of $A^2x=Ab$.} \label{2.2}
\end{algorithm}
An analogous complexity bound to Theorem \ref{thm2pTA} can be stated for Algorithm \ref{2.2}.

\subsection{Triangle Algorithm for Approximation of Minimum-Norm Solution} \label{sec2.3}

Algorithm \ref{2.3}, as described below, computes an $\varepsilon$-approximate minimum-norm solution to the system $Ax = b$. It comprises multiple phases. In each phase, it utilizes an $\varepsilon$-approximate solution, denoted as $x_\varepsilon$. Therefore, if $\overline \rho = \Vert x_\varepsilon \Vert$,  $\overline b = Ax_\varepsilon \in E_{A, \overline \rho}$. The initial $x_\varepsilon$ can be computed using any algorithm.
Additionally, in each phase, a value $\underline \rho$ is utilized, representing a range where $E_{A, \underline \rho}$ does not contain any $\varepsilon$-approximate solution. Initially, $\underline \rho = 0$. The algorithm consistently decreases $\overline \rho = \Vert x_\varepsilon \Vert$ or increases $\underline \rho$ until $\overline \rho - \underline \rho \leq \varepsilon$. If this inequality is not satisfied for the given $\overline \rho$ and $\underline \rho$, Algorithm \ref{2.1} is used to test if there exists an $\varepsilon$-approximate solution in $E_{A, \rho}$, where $\rho = (\overline \rho + \underline \rho)/2$. If such a solution exists, we set $\overline \rho = \rho$. Otherwise, Algorithm \ref{2.1} provides a witness $b'$. In this scenario, $\underline \rho$ is set as $\max{\rho, {(b - b')^Tb /\Vert c \Vert}}$, and the process is repeated. For simplicity of description of the algorithm, we will assume $\Vert A^T(b-b') \Vert=0$ in the inner while loop will not occur. Otherwise, we take the new $x_\varepsilon= \alpha x'$, where $\alpha \in (0,1)$ is so that $\Vert A x_\varepsilon - b \Vert =\varepsilon$ and $A^TAx' \not = A^Tb$.

Figure \ref{minnormx} illustrates the initial phase with $E_{A, \overline \rho}$ (the largest ellipse) and the first $E_{A, \rho}$ (the smallest ellipse).
The algorithm tests if $E_{A, \rho}$  contains an $x_\varepsilon$-approximate solution.  In this test the figure shows an intermediate iterate $b'$ corresponding to a strict pivot $v_\rho$. $b'$ happens to be a witness. Then
$E_{A, \rho}$ is expanded to the next $E_{A, \underline \rho}$. Thus the initial gap $\overline \rho - \underline \rho$ is reduced by a factor of at least $1/2$. The second largest ellipse in the figure is the smallest ellipse containing $b$, that is $E_{A, \rho_*}$,  $\rho_*=\Vert x_* \Vert$.

\begin{algorithm}[!htb]
\scriptsize
\SetAlgoNoLine
\KwIn{$A \in \mathbb{R}^{m \times n}$, $b \in \mathbb{R}^m$, $b \not =0$,
$\varepsilon$, $x_\varepsilon$ ($\varepsilon$-approximate solution of $Ax=b$)}
$\overline \rho \gets \Vert x_\varepsilon \Vert$, $\underline \rho \gets 0$,
$x' \gets 0$, $b' \gets Ax'$.

\While{$\overline \rho - \underline \rho > \varepsilon$}{
 $\rho \gets \frac{1}{2}(\overline \rho + \underline \rho)$,

 \While{$((\|b - b'\| > \varepsilon) \wedge (\Vert A^T(b-b') \Vert > 0)$}
 {$c \gets A^T(b-b')$, $v_\rho \gets \rho
A{c}/{\|c\|}$.

  \lIf{$ \rho \Vert c \Vert \geq (b-b')^Tb$} { $\alpha \gets (b - b')^T(v_\rho - b')/\|v_\rho - b'\|^2$,\\
$b' \gets (1-\alpha) b' + \alpha v_\rho$, \quad $x' \gets (1-\alpha) x' +  \alpha {\rho c}/{\Vert c \Vert}$}
   \Else{{ STOP ($b'$ is a witness, the while loop terminates) }}}
     \lIf{$\Vert b-b' \Vert \leq \varepsilon$} { $\overline \rho \gets \rho$}
   \Else{{$\underline \rho \gets \max \{\rho, (b-b')^Tb/\Vert c \Vert\}$}}}
    \caption{(TA) Given $\varepsilon$-approximate solution of $Ax=b$,  the algorithm computes an $\varepsilon$-approximate minimum-norm solution of
    $Ax=b$ or a solution to $A^TAx=A^Tb$ (final $x'$ is approximate solution).} \label{2.3}
\end{algorithm}
Noting that each time $\rho$ is changed, we are attempting to compute an $\varepsilon$-approximate solution within an ellipsoid with $\rho \leq \Vert x_\varepsilon \Vert$, we can apply Theorem \ref{thm2pTA}. The complexity of this problem is $O(\Vert A \Vert^2 \rho^2/ \varepsilon^2)$. Utilizing this information, we obtain the following.

\begin{theorem} Algorithm \ref{2.3} either computes a solution to the normal equation (when $c=0$) or it computes an $\varepsilon$-approximate minimum-norm solution to $Ax=b$. Moreover, if $k$ is the smallest integer such that $\Vert x_\varepsilon \Vert/2^k \leq \varepsilon$, the number of iterations in the while loop is at most $k$. In other words, the complexity of Algorithm \ref{2.3} is  $O(\ln(\Vert x_\varepsilon \Vert/\varepsilon)) \times O(\Vert A \Vert^2 \rho^2/ \varepsilon^2)$. \qed
\end{theorem}
\begin{figure}[htpb]
	\centering
	\raisebox{3ex}{
	\begin{tikzpicture}[scale=.8]
	\begin{scope}
	\begin{scope}
 [rotate=47.5829]
	\draw \boundellipse{0,0}{2.81}{1.21};
	\filldraw (0,0) circle (.5pt) node[below] {$O$};
	\end{scope}
	\filldraw (0,2.79) circle (.9pt)  node[right] {$b$};
\filldraw (-.2,3.1) circle (.9pt)  node[above] {$\overline b$};
\filldraw (-2.5,-2.79)   node[above] {$E_{A, \underline \rho}$};
\filldraw (-2,-2.2)   node[above] {$E_{A, \rho}$};
\filldraw (-3,-4.1)   node[above] {$E_{A, \rho_*}$};
\filldraw (-3.90,-4.)   node[left] {$E_{A, \overline \rho}$};
	\filldraw (0.9815,2.158)  circle (.9pt)  node[above] {$v_\rho$};
	\filldraw (0.6580,0.75)  circle (.9pt)  node[left] {$b'$};
	\draw[dashed] (-1.37,1.4)--(2.69,2.709);
	\draw (0.6580,0.75)--(0.9815,2.158)--(0,2.79)--(0.6580,0.75);
    \begin{scope}
	[rotate=47.5829]
	\draw \boundellipse{0,0}{5.62}{2.42};
	\filldraw (0,0) circle (.5pt) node[below] {$O$};
	\end{scope}
	\end{scope}
     \begin{scope}
	[rotate=47.5829]
	\draw \boundellipse{0,0}{3.653}{1.573};
	\filldraw (0,0) circle (.5pt) node[below] {$O$};
	\end{scope}
\begin{scope}
	[rotate=47.5829]
	\draw \boundellipse{0,0}{3.653*1.31}{1.573*1.31};
	\filldraw (0,0) circle (.5pt) node[below] {$O$};
	\end{scope}
	\end{tikzpicture}}
	\caption{{\small $E_{A, \overline \rho}$ is initial ellipsoid. Initial $E_{A, \underline \rho}$ is the origin. Thus $\rho=\overline \rho/2$. The ellipsoid
$E_{A, \rho}$ gives a witness, $b'$, and it is expanded to give new $E_{A, \underline \rho}$,  $\underline \rho=\max\{\rho, (b-b')^Tb/\Vert c \Vert\}$. The process is repeated until $\overline \rho - \underline \rho \leq \varepsilon$.}}
	\label{minnormx}
\end{figure}
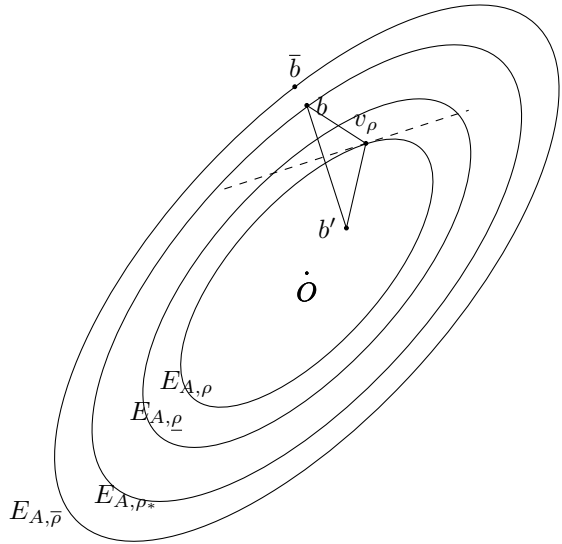
A modified Algorithm \ref{2.3} for when $A$ symmetric PSD can be stated.

\section{The First-Order Centering Triangle Algorithm (CTA)} \label{sec3}

Having described the Triangle Algorithm (TA), inspired by it, in this section, we introduce the {\it first-order} {\it Centering Triangle Algorithm} (CTA). CTA is an iterative method for solving $Ax=b$ or $A^TAx=A^Tb$. It is the first member of a family of iterative functions with a simple iterative formula but possessing powerful convergence properties, which will be described and proved here. The first-order CTA can be interpreted as a dynamic version of TA. Therefore, it can be described both geometrically and algebraically.

\subsection{CTA as a Geometric Algorithm}
\label{sec3.1}

When solving $Ax=b$, given $x' \in \mathbb{R}^n$, we compute $b' = Ax'$. The residual is then $r' = A(x-x') = b - b'$. To calculate the next iterate $b'' = Ax''$, we set $b''$ as the sum of $b'$ and the orthogonal projection of the vector of the residual $r'$ onto $Hr'$, where $H=AA^T$. However, for the case of $A$ symmetric PSD, we can take $H=A$. It's important to note that when $Hr'=0$, $A^TAx'=A^Tb$, and thus, if $Ax=b$ is solvable, $Ax'=b$. The orthogonal projection of a given $u \in \mathbb{R}^m$ onto a given nonzero $v \in \mathbb{R}^m$ is ${u^Tv}/{\Vert v \Vert^2}v$. Using this,
\begin{equation} \label{CTA1}
 b'' = b' + \alpha(r') H r',  \quad \alpha(r') =  \frac{{r'}^T\!\! H r'}{{r'}^TH^2r'}.
\end{equation}
Equivalently, the new residual $r''=b-b''$ and approximate solution $x''$ satisfy $r''=b-Ax''$, where
\begin{equation}  \label{CTAxnew}
\quad  r'' =F_1(r')= r' - \alpha(r') H r', \quad
x''= x'+ \begin{cases}
\alpha(r') r', & \text{if ~} H=A\\
\alpha(r') A^Tr', & \text{if~} H=AA^T.
\end{cases}
\end{equation}
When $H=A$, $(x''-x')$ is a scalar multiple of the residual $r'$, and when $H=AA^T$, $(x''-x')$ is a scalar multiple of the {\it least-squares} residual $A^Tr'$.
\subsection{CTA as an Algebraic Algorithm} \label{sec3.2}
Given residual $r' \in \mathbb{R}^m$, with $H r' \not  =0$, the next residual $r''$ is given as
\begin{equation}  \label{F1def}
r''=F_1(r')=r'-  \alpha_{1,1}(r')Hr',  \quad
\alpha_{1,1}(r')= {\rm argmin} \{\Vert r'- \alpha Hr' \Vert^2: \alpha \in \mathbb{R} \} =\frac{{r'}^THr'}{{r'}^TH^2r'}.
\end{equation}

Algorithm \ref{3.1} describes  first-order CTA. In the implementation of the algorithm, when $H=AA^T$, it will only be used implicitly and thus there is no need to compute the product explicitly.
In the implementation of Algorithm \ref{3.1}, we may apply preconditioning  to the matrix $H$.  For instance, a simple preconditioning method is to pre and post-multiplying $H$ by $D^{-1/2}$, where $D$ is the diagonal matrix of diagonal entries of $H$.

\begin{algorithm}[!htb]
\scriptsize
\SetAlgoNoLine
\KwIn{$A \in \mathbb{R}^{m \times n}$, $b \in \mathbb{R}^m$, $\varepsilon \in (0, 1)$.}
$x' \gets 0$ (or $x'=A^Tw'$, $w' \in \mathbb{R}^m$ random), $r' \gets b-Ax'$.

\While{$ {\rm (}\|r'\|  > \varepsilon {\rm )} $ $\wedge$ $ {\rm (} r'^T H r'> \varepsilon {\rm )}$}{
$\alpha \gets \frac{{r'}^T\!\! Hr'}{{r'}^TH^2r'}$,

$r' \gets F_1(r')= r' - \alpha H r'$, \quad $x' \gets  x'+ \begin{cases}
\alpha r', & \text{if ~} H=A\\
\alpha A^Tr', & \text{if~} H=AA^T. \end{cases}$}
\caption{(CTA) Iteration of $F_1$ for computing an $\varepsilon$-approximate solution of $Ax=b$ or $A^TAx=A^Tb$ (final $x'$ is approximate solution).} \label{3.1}
\end{algorithm}

As in the Triangle Algorithm, the dominant computational part of the algorithm is matrix-vector multiplications. While $b''=Ax''$, the algorithm first computes  $b''$ and later $x''$.  The main work in each iteration is computing $Hr'$. This takes two matrix-vector operations when $H=AA^T$ and one when $H=A$. All  other computations take $O(m+n)$ operations. If the number of nonzero entries of $A$ is $N$, each matrix-vector computation takes $O(N)$ elementary operations.

If the algorithm terminates with $r'^THr' \leq \varepsilon$,  when $H=AA^T$, it implies $\Vert A^T r' \Vert \leq \sqrt{\varepsilon}$, and when $H=A$ it implies  $\Vert A^{1/2} r' \Vert \leq \sqrt{\varepsilon}$. In either case, this simple algorithm is capable of approximating a solution to the linear system or to the normal equation. We will derive bounds on the iteration complexity for the symmetric positive definite case of $H$ as well as positive semidefinite case.

\begin{remark}
Similar to the Triangle Algorithm, the first-order CTA is based on two-term recurrences that generate two vector sequences: $r_k$ and $x_k$. The overall computational cost of CTA per step is comparable to that of the Bi-Conjugate Gradient (BiCG) method for solving square and possibly nonsymmetric linear algebraic systems. Due to the short recurrences, the cost and storage requirements of CTA per step remain fixed throughout the iteration.
\end{remark}

\begin{example} \label{example1} Here we consider an example of a single iterate of first-order CTA for a symmetric PSD matrix.
Consider $Ax=b$, where $H=A={\rm diag}(1,2, \dots, m)$, $b=(1, \dots, 1)^T$.  We compute the error in the first iteration of CTA. With $x_0=0$, $r_0=b$.
It is easy to show that in general  with $\phi_i(r)=r^TH^ir$,
$\Vert F_1(r_0) \Vert^2 = \Vert r_0 \Vert^2- {\phi^2_1(r_0)}/ {\phi_2(r_0)}.$
Using that $\Vert r_0 \Vert^2=m$, $\phi_1(r_0)= \sum_{j=1}^m j=m(m+1)/2$,
$\phi_2(r_0)= \sum_{j=1}^m j^2=m(m+1)(2m+1)/6$, we get
${\Vert F(r_0) \Vert}/{\Vert r_0 \Vert}= \sqrt{(m-1)/{2(2m+1)}} \approx \frac{1}{2}.$
\end{example}

\begin{example} \label{example1p} Here we consider the example as in the previous case except that we use $H=AA^T={\rm diag}(1^2,2^2, \dots, m^2)$, $x_0=0$, $r_0=b$. We would expect the ration of $\Vert F_1(r_0) \Vert/ \Vert r_0 \Vert$ to be worse. In this case
$\phi_1(r_0)= \sum_{j=1}^m j^2=m(m+1)(2m+1)/6 \approx m^3/3$.  Moreover,
$\phi_2(r_0)= \sum_{j=1}^m j^4 \approx m^5/5$.
Thus $\Vert F_1(r_0) \Vert^2 \approx m- \frac{5}{9}m= \frac{4}{9}m.$ Hence
${\Vert F_1(r_0) \Vert}/{\Vert r_0 \Vert}  \approx {2}/{3}.$
In other words, while the condition number of $H$ in the examples are $m$ and $m^2$, respectively, the corresponding reduction of residual,  $1/2$ and $2/3$ is  reasonably close.
\end{example}

\subsection{Interpretation of CTA for Symmetric PSD Matrices} \label{sec3.3}

In this subsection we give an interpretation of the first-order CTA for the case where $A$ is symmetric PSD.

\begin{theorem}  Consider the linear system $Ax=b$, where  $A$ is symmetric PSD. Applying  Algorithm \ref{3.1} for this case can be interpreted as applying the algorithm for the general matrix to solve $A^{1/2} y= b$, as if computing its minimum-norm solution, followed by an adjustment of the iterate via an implicit use of $A^{1/2}$.
\end{theorem}
\begin{proof} If $A^{1/2} y= b$ is solvable, by Proposition \ref{minnorm}
its minimum-norm solution is the solution of $Ax=b$. That is, if $x$ is any solution of the latter equation, $y_*=A^{1/2}x$ is the minimum-norm solution of the former.   Now applying Algorithm \ref{3.1} for general matrix to the system $A^{1/2} y= b$, $H=A^{1/2}A^{1/2}=A$ so that given the current residual $r'=b - b'$, where $b'=A^{1/2} y'$, from (\ref{CTAxnew}) the next residual is
\begin{equation}
r''= r' - \frac{r'^TAr'}{r'^TA^2r'}Ar', \quad
y''= y'+ \frac{{r'}^T\!\!Ar'}{{r'}^TA^2r'} A^{1/2} r' = A^{1/2}  \bigg (
x'+ \frac{{r'}^T\!\!Ar'}{{r'}^TA^2r'}  r' \bigg).
\end{equation}
Thus we set $x''= x' + ({{r'}^T\!\!Ar'}/{{r'}^TA^2r'})r'$.
\end{proof}

\subsection{CTA as a Dynamic Version of Triangle Algorithm} \label{sec3.4}

Consider the iterative step in solving $Ax=b$ via Algorithm \ref{2.1}. For a given $x' \in \mathbb{R}^n$ with $b'=Ax' \in E_{A, \rho}$ and residual $r'=b-b'$, if $c=A^T(b-b') \not =0$, the algorithm computes $v_\rho = Ac/\Vert c\Vert$. If $\rho \geq (b-b')^Tb/\Vert c \Vert$, it sets $\alpha =(b - b')^T(v_\rho - b')/|v_\rho - b'|^2$ and defines $b''=Ax''$, where $b''=(1-\alpha) b' + \alpha v_\rho$, and $x''=(1-\alpha) x' + \alpha {\rho c}/{\Vert c \Vert}$.

Alternatively, let $y=x-x'$ and consider the iterative step of Algorithm \ref{2.1} with respect to the equation $Ay=r'$ at $y'=0$ over $E_{A, \rho'}$, where $\rho'$ is the smallest radius $\rho$ so that the corresponding $v_{\rho}$, denoted by $v'$, is a strict pivot at $y'=0$. Noting that $r'-Ay'=r'$, it follows that

\begin{equation}
\rho'= \frac{\Vert r' \Vert^2}{\Vert A^Tr' \Vert},  \quad v'=\rho' \frac{AA^Tr'}{\Vert A^Tr' \Vert}.
\end{equation}
Thus
\begin{equation} \label{eq99}
v'=\frac{\Vert r' \Vert^2}{\Vert A^Tr' \Vert^2} AA^Tr'.
\end{equation}
Denoting the corresponding step-size $\alpha$ by $\alpha'$, computing it and substituting  for $v'$ from (\ref{eq99}), and substituting $H=AA^T$, we have
\begin{equation}
\alpha' = \frac{{r'}^Tv'}{\Vert v' \Vert^2}=  \frac{ \Vert r' \Vert^2}{\Vert v' \Vert^2}= \frac{\Vert A^T r' \Vert}{\Vert AA^Tr' \Vert^2}= \frac{{r'}^THr'}{{r'}^TH^2 r'}.
\end{equation}
The new residual $r''$, hence $b''$ and $x''$ are,
\begin{equation}
r''= \alpha' v' = \frac{{r'}^THr'}{{r'}^TH^2r'} Hr', \quad
b''=b'+ \frac{{r'}^THr'}{{r'}^TH^2r'} Hr', \quad x''=x'+ \frac{{r'}^THr'}{{r'}^TH^2r'} A^Tr'.
\end{equation}

\begin{theorem} The iteration of CTA in Algorithm \ref{3.1} to solve $Ax=b$ at a given iterate $b'=Ax'$ corresponds to the iteration of TA in Algorithm \ref{2.1} to solve $Ay=r'$, where $r'=b-b'$, at $y'=0$, over the ellipsoid $E_{A,\rho'}$. Here, $\rho'$ is defined as ${\Vert r' \Vert^2}/\Vert A^Tr'\Vert$, which represents the smallest radius ball that allows $v'$ in (\ref{eq99}) to function as a strict pivot. This is followed by computing the next iterate of TA, $b''=Ax''$, and determining the corresponding residual $r''=b-b''$. \qed
\end{theorem}

\begin{remark}
As shown, CTA can be considered a dynamic version of TA, and in this regard, it appears to be more powerful. However, these two algorithms can actually complement each other. For example, suppose we are given an $\varepsilon$-approximate solution to a linear system $Ax=b$. We may wonder if we can obtain such an approximate solution with a smaller norm. TA can achieve this with Algorithm \ref{2.3}. The same question may arise concerning an approximate solution of the normal equation. Furthermore, there may be situations where we aim to find the best solution to a linear system but with a constraint on its norm. This can be accomplished using Algorithm \ref{2.1}.
\end{remark}

\subsection{An Auxiliary Lemma on Symmetric PSD Matrices} \label{sec3.5}
Our goal here is to state an important auxiliary lemma regarding PSD matrices. This lemma plays a significant role in the analysis of the convergence properties of the first-order CTA. The same result will also be used in the analysis of the convergence of high-order CTA.

\begin{lemma} \label{lem1} Let $H$ be an $m \times m$ symmetric PSD matrix, $H \not =0$.  Let $\kappa^+$ denote the ratio of its largest to smallest positive eigenvalues.

(i) Suppose $H$ is positive definite, i.e. $\kappa^+=\kappa$, its condition number.
For any  $x \not =0$,
\begin{equation}  \label{lema1eq1}
\frac{(x^TH x)^2}{x^TH^2x}  \frac{1}{\Vert x \Vert^2} \geq  \frac{4\kappa}{(\kappa+1)^2}.
\end{equation}

Moreover, if $u_{\min}$ and $u_{\max}$ are orthogonal unit-norm eigenvectors corresponding to the eigenvalues $\lambda_{\min}$ and $\lambda_{\max}$, respectively, then equality is achieved in (\ref{lema1eq1}) when $x$ is any nonzero multiple of the following:

\begin{equation} \label{lema1eq1xxz}
\sqrt{\frac{\kappa}{\kappa+1}} u_{\min} + \sqrt{\frac{1}{\kappa+1}} u_{\max}
\end{equation}

(ii) If $x \not =0$ is a linear combination of eigenvectors of $H$ corresponding to positive eigenvalues,
\begin{equation}  \label{lema1eq2}
\frac{(x^TH x)^2}{x^TH^2x}  \frac{1}{\Vert x \Vert^2} \geq  \frac{4\kappa^+}{(\kappa^++1)^2}.
\end{equation}

(iii) Let $c(H)={1}/{\kappa^+ \lambda_{\max}}$.
If for a given $x$, $x^THx >0$, then
\begin{equation} \label{LEM1eq2}
\frac{(x^TH x)^2}{x^TH^2x} \geq c(H) \times x^T Hx.
\end{equation}

\end{lemma}

\begin{proof}  Let $u_1, \dots, u_m$ be an orthonormal set of eigenvectors of $H$. Let $\lambda_1, \dots, \lambda_m$ be the
corresponding eigenvalues. Without loss of generality assume
$\lambda_1 \leq  \cdots  \leq \lambda_m$.

(i):  If $ \lambda_1= \lambda_m$,  (i) is trivial. Thus assume there are at least two distinct eigenvalues. We show
\begin{equation} \label{eqrho2}
{\rm inf} \bigg \{ \frac{(w^TH w)^2}{w^TH^2w} \frac{1}{\Vert w \Vert}:   w \not =0  \bigg \} =  \min \bigg \{ \frac{(w^TH w)^2}{w^TH^2w}: \Vert  w \Vert^2 = 1 \bigg \} =
\min \bigg \{ \frac{(\sum_{i=t}^m \lambda_i x_i^2)^2}{(\sum_{i=t}^m \lambda^2_i x_i^2)}: \sum_{i=t}^m x_i^2=1 \bigg \}.
\end{equation}
The first equality in (\ref{eqrho2}) follows since the objective function is homogeneous of degree zero in $w$, i.e. replacing $w$ by $\alpha w$, $\alpha \not =0$, the objective function remains unchanged. Thus we may assume $\Vert w \Vert =1$.

Given $w \in \mathbb{R}^m$ we can write $w = \sum_{i=1}^m x_i u_i$. Thus
\begin{equation}
w^Tw = \sum_{i=1}^m x_i^2, \quad w^THw= \sum_{i=t}^m \lambda_i x_i^2, \quad w^TH^2w= \sum_{i=t}^m \lambda^2_i x_i^2.
\end{equation}
Substituting these into the objective function in (\ref{eqrho2}) we get the proof of the second equality.

Now consider the last optimization in (\ref{eqrho2}). If two  eigenvalues $\lambda_i$ and $\lambda_j$ are identical we combine the corresponding sum $x_i^2+x_j^2$ into a single variable, say $y_i^2$, thereby simplifying the optimization problem into a corresponding one in dimension one less. Thus we assume eigenvalues are distinct.
Let $x^*=(x_1^*, \dots, x_m^*)^T$ be an optimal solution of this optimization in (\ref{eqrho2}).
We claim $x_i^*$ is nonzero only for two distinct indices.
To prove the claim, from the Lagrange multiplier optimality condition applied to this optimization we  have
\begin{equation} \label{eq333}
4 \lambda_i x^*_i \big (\sum_{i=1}^m \lambda_i {x^*_i}^2 \big ) \big (\sum_{i=1}^m \lambda^2_i {x^*_i}^2 \big )- 2 \lambda^2_i x^*_i \big ( \sum_{i=1}^m \lambda_i {x^*_i}^2 \big )^2=
\delta x^*_i,   \quad i=1, \dots, m,
\end{equation}
where $\delta$ is a  constant. Dividing each equation in (\ref{eq333}) with $x^*_i \not =0$ by $x^*_i$ and  letting $\delta'= \delta/ (\sum_{i=1}^m \lambda_i {x^*_i}^2)$ we get
\begin{equation} \label{eqn22}
4 \lambda_i \sum_{i=1}^m \lambda^2_i {x^*_i}^2- 2 \lambda^2_i \sum_{i=1}^m \lambda_i {x^*_i}^2=
\delta', \quad i \in \{1, \dots, m\}, \quad x^*_i \not =0.
\end{equation}

If more than two of $x^*_i$'s in the equations of (\ref{eqn22}) are nonzero, the corresponding $\lambda_i$'s are distinct roots of the same quadratic equation in $\lambda$, defined  by (\ref{eqn22}), with the summations as its coefficients. But this is a contradiction since the quadratic equation has at most two solutions. Thus there exist only two indices $1 \leq i < j \leq m$, where $x^*_i, x^*_j$ are nonzero. Set $\gamma_{ji}= \lambda_j/\lambda_i$. Since by assumption on the eigenvalues $\lambda_i < \lambda_j$,  $\gamma_{ji} >1$.  Thus the corresponding optimization in (\ref{eqrho2}) reduces to an optimization in two variables:

\begin{equation} \label{eqopti}
\min \bigg \{ \frac{(\lambda_i x_i^2+ \lambda_j x_j^2)^2}{\lambda^2_i x_i^2+ \lambda_j x_j^2}: x_i^2+x_j^2=1 \bigg \}.
\end{equation}

Using the equation $x_i^2+x_j^2=1$, we reduce the above to an optimization problem with a single-variable, shown below:

\begin{equation} \label{eqopti2h}
\min \bigg \{ \frac{\big (1+ (\gamma_{ji}-1) x_j^2 \big )^2}{ 1+ (\gamma^2_{ji} -1) x_j^2}: x_j^2 \leq 1 \bigg \}.
\end{equation}

When $x_j^2=1$, the objective value is $1$. We claim the minimum is attained when $x_j^2 <1$. To this end we differentiate the objective function  in (\ref{eqopti2h}) and set equal to zero to get
\begin{equation}
4x_j(\gamma_{ji}-1) \big (1+ (\gamma_{ji}-1) x_j^2 \big ) \big (1+ (\gamma^2_{ji} -1) x_j^2 \big )-
2x_j(\gamma^2_{ji}-1)\big (1+(\gamma_{ji}-1) x_j^2 \big )^2=0.
\end{equation}
Dividing by $2x_j (\gamma_{ji}-1) \big (1+ (\gamma_{ji}-1) x_j^2 \big )$ gives
\begin{equation}
2\big (1+ (\gamma^2_{ji} -1) x_j^2 \big )-
(\gamma_{ji}+1)\big (1+ (\gamma_{ji} -1) x_j^2 \big )=0.
\end{equation}
Equivalently,
\begin{equation}
2(\gamma^2_{ji} -1) x_j^2 -(\gamma_{ji}+1)(\gamma_{ji} -1) x_j^2=
(\gamma_{ji}+1) - 2.
\end{equation}
Simplifying gives
\begin{equation}
(\gamma^2_{ji} -1) x_j^2 =(\gamma_{ji}-1).
\end{equation}
It is now straightforward to show that the solution of the above equation gives the optimal solution of (\ref{eqopti2h}). Thus the optimal solutions of (\ref{eqopti}) satisfies:
\begin{equation} \label{eqopti2}
x^2_{j*}=\frac{1}{\gamma_{ji} +1}= \frac{\lambda_i}{\lambda_i + \lambda_j}
,  \quad x^2_{i*}=\frac{\gamma_{ji}}{\gamma_{ji} +1}=\frac{\lambda_j}{\lambda_i + \lambda_j}.
\end{equation}
Substituting (\ref{eqopti2}) into the objective function in (\ref{eqopti}) yields the following objective value, which happens to be less than one; hence, it is the optimal value.
\begin{equation} \label{eqlam4}
\frac{4 \gamma_{ji} }{(\gamma_{ji}+1)^2} < 1.
\end{equation}
Considering the function $\gamma/(\gamma+1)^2$, it decreases over the interval $[1, \infty)$. Hence, the minimum value (\ref{eqlam4}) occurs when $\gamma_{ji}= \gamma_{m1}=\kappa$. This proves that (\ref{lema1eq1xxz}) is an optimal solution, and due to the homogeneity of the objective function, any scalar multiple of it is also optimal. Thus, the proof of (i) is complete.

(ii):  Suppose the positive eigenvalues of $H$ are $\lambda_t, \dots, \lambda_m$. Let $\Lambda_+={\rm diag}(\lambda_t, \dots, \lambda_m)$. By assumption on $x$, we can write $x=\sum_{i=t}^m \beta_i u_i$ for some coefficients $\beta_t, \dots, \beta_m$. Let $y=(\beta_t, \dots, \beta_m)^T$. Note that $\Vert x \Vert = \Vert y \Vert$,  $x^THx=y^T\Lambda_+y$ and $x^TH^2x=y^T\Lambda^2_+y$.
Using these and applying the bound from part (i) to the positive definite matrix $\Lambda_+$, we complete the proof of (ii).

(iii): When $x$ is arbitrary, we still have $x^THx=y^T\Lambda_+y$ and $x^TH^2x=y^T\Lambda^2_+y$. However, $\Vert x \Vert \geq \Vert y \Vert$. But we have,
\begin{equation} \label{eqyy}
\Vert y \Vert ^2 = \sum_{i=t}^m \beta_i^2 \geq \frac{1}{\lambda_{\max}} \sum_{i=t}^m \lambda_i \beta_i^2 = \frac{x^THx}{\lambda_{\max}}.
\end{equation}
Using (\ref{eqyy}) and applying the bound on part (i) to the positive definite matrix $\Lambda_+$ we may write
\begin{equation}
\frac{(x^THx)^2}{x^TH^2x}  = \frac{(y^T\Lambda_+ y)^2}{ y^T\Lambda^2_+y} =
\frac{(y^T\Lambda_+y)^2}{ y^T \Lambda^2_+y}  \frac{1}{ \Vert y \Vert^2} \Vert y \Vert^2 \geq
\frac{4\kappa^+}{(\kappa^++1)^2} \frac{x^THx}{\lambda_{\max}}.
\end{equation}
Since $\kappa^+ \geq 1$, we have
\begin{equation} \label{eq31}
\frac{4\kappa^+} {(\kappa^++1)^2}   \geq \frac{1}{\kappa^+}.
\end{equation}
This completes the proof of (iii).
\end{proof}

\subsection{Relating  Magnitudes of Consecutive Residuals in First-Order CTA} \label{sec3.6}

\begin{lemma} \label{newlemz} Given $H$ an $m \times m$ symmetric PSD matrix, for any  $r \in \mathbb{R}^m$ with $Hr \not =0$, set
$F_1(r)=r- (r^THr/r^TH^2r)Hr$. Then
\begin{equation} \label{erf}
\Vert F_1(r) \Vert^2=  \Vert r \Vert^2- \frac{(r^TH r)^2}{r^TH^2r}.
\end{equation}
\end{lemma}
\begin{proof} It suffices to note $\Vert F_1(r) \Vert^2 = F_1(r)^TF_1(r)$.
\end{proof}

\begin{theorem} \label{thm1} Let $F_1(r)$  be as defined previously in terms of a symmetric PSD matrix $H$. Let $\kappa^+$ be
the ratio of the largest to smallest positive eigenvalues of $H$.  Let  $c(H)={1}/{\kappa^+ \lambda_{\max}}$.

\noindent (i) Suppose $r \in \mathbb{R}^m$ is a linear combination of eigenvectors of $H$ corresponding to positive eigenvalues. Then,

\begin{equation} \label{coreq1}
\Vert F_1(r) \Vert \leq   \bigg (\frac{\kappa^+-1}{\kappa^++1} \bigg ) \Vert r \Vert.
\end{equation}

\noindent (ii)  Given arbitrary $r \in \mathbb{R}^m$,
if $Hr \not =0$,
\begin{equation} \label{coreq2}
\Vert F_1(r) \Vert^2  \leq   \Vert r \Vert^2   -  c(H) \times r^THr,
\end{equation}
\end{theorem}

\begin{proof}

 (i): If $r=0$, $(\ref{coreq1})$ holds. Otherwise, by assumption on $r$, $Hr \not =0$.  From Lemma \ref{newlemz}, (\ref{erf}), followed by the bound (\ref{lema1eq2}) in Lemma \ref{lem1} we get
\begin{equation}
\Vert F_1(r) \Vert^2 \leq   \Vert r \Vert^2 \bigg (1 - \frac{4 \kappa^+}{(\kappa^++1)^2} \bigg)= \Vert r \Vert^2 \bigg (\frac{ \kappa^+-1}{\kappa^++1} \bigg )^2.
\end{equation}
Hence  (\ref{coreq1}) holds.

(ii): Proof of (\ref{coreq2})
follows trivially from (\ref{erf}) and the bound (\ref{LEM1eq2}) in Lemma \ref{lem1}.
\end{proof}

\begin{example} \label{example3}
We consider a worst-case example for $F_1(r)$. Let $A={\rm diag}(1,2, \dots, m)$, $b=(1, \dots, 1)^T$.  Let
$r=({\sqrt{m}}/{\sqrt{m+1}}, 0, \cdots,0, 1/{\sqrt{m+1}})^T$.
Then $r^Tr=1$, $r^TAr=\frac{2m}{(m+1)}$, $r^TA^2r=m$. Substituting these, and using $\kappa^+=m$, we get $\Vert F_1(r) \Vert= (\kappa^+-1)/(\kappa^++1)$ so that the worst-case bound it attained.
\end{example}

\subsection{Convergence Properties of First-Order CTA} \label{sec3.7}

We characterize the convergence properties of iterations of $F_1$. Firstly, we establish a simple yet crucial property concerning PSD linear systems. This result, in conjunction with the auxiliary lemma on symmetric PSD matrices (Lemma \ref{lem1}), will be employed to derive complexity bounds for first-order CTA and its generalizations. These bounds will be expressed in terms of $\kappa^+$ rather than $\kappa$. In other words, when considering the use of CTA, the crucial factor is the solvability of the system $Ax=b$, rather than the invertibility of $A$ or $H=AA^T$, which implies solvability.

\begin{proposition} \label{syst}  Consider $Hx=b$, where $H$ is an $m \times m$ symmetric PSD matrix, $b \not=0$.

(1) The system $Hx=b$ is solvable if and only if $b$ is a linear combination of eigenvectors of $H$ corresponding to positive eigenvalues.

(2) The system $Hx=b$ is solvable if and only if, for any $x_0 \in \mathbb{R}^m$, $r_0=b-Hx_0$ is a linear combination of eigenvectors of $H$ corresponding to positive eigenvalues.

(3) The system $Ax=b$  is solvable if and only if $b$ is a linear combination of eigenvectors of $AA^T$ corresponding to positive eigenvalues.

(4) The system $Ax=b$ is solvable if and only if, for any $w_0$, $r_0=b-AA^Tw_0$ is a linear combination of eigenvectors of $H$ corresponding to positive eigenvalues.
\end{proposition}

\begin{proof} Let $H=U \Lambda U^T$ be the spectral decomposition of $H$, where $U=[u_1, \cdots, u_m]$ is the matrix of orthonormal eigenvectors, and $\Lambda={\rm diag}(\lambda_1, \cdots, \lambda_m)$ the diagonal matrix of eigenvalues. Without loss of generality, assume $\lambda_i=0$ for $i=1, \dots, t-1$, and $\lambda_i >0$ for $i \geq t$.

Multiplying the equation $Hx=b$ by $U^T$ on both sides, we get $\Lambda U^Tx=U^Tb$. Let $y=U^Tx$. We can write $b=\sum_{i=1}^m \beta_i u_i$.
Then, $U^Tb=(\beta_1, \cdots, \beta_m)^T$.

(1): Suppose $Hx=b$ is solvable. Then $\lambda_iy_i=\beta_i$ for $i=1, \dots, m$. This implies $\beta_i=0$ for $i=1, \dots, t-1$. Consequently, this implies $b$ is a linear combination of eigenvectors with positive eigenvalues.

Conversely, suppose $b$ is such a linear combination. Then in $\lambda_i y_i=\beta_i$, we set $y_i=0$ for $i=1, \dots, t-1$. For $i \geq t$, we set $y_i=\beta_i/\lambda_i$. This choice results in a solution to the system $\Lambda y=U^Tb$.  Then, by computing $x=U^Ty$, we obtain is a solution to $Hx=b$.

(2): Given that $r_0=b-Hx_0$, it follows that $Hx=b$ is solvable if and only if $H(x-x_0)=r_0$ is solvable. Thus, we can apply the first part of the theorem to the latter system.

(3) and (4): The proofs for (3) and (4) follow directly because $Ax=b$ is solvable if and only if $AA^Tw=b$ is solvable, and these equations can be analyzed similarly to (1) and (2) by applying the results from part (1) of the theorem.
\end{proof}

\begin{theorem}  \label{thm2} {\rm (Properties of Residual Orbit)} Consider solving $Ax=b$ or $A^TAx=A^Tb$, where $A$ is an $m \times n$ matrix.
Let $H=AA^T$. If $A$ is symmetric PSD, we can take $H=A$.
Let $\kappa^+$ be the ratio of the largest to smallest positive eigenvalues of $H$.  Let  $c(H)={1}/{\kappa^+ \lambda_{\max}}$. Given $r \in \mathbb{R}^m$ with $Hr \not =0$, let

\begin{equation} \label{thmeq200}
F_1(r) = r - \alpha_{1,1}(r)Hr, \quad  \alpha_{1,1}(r)=\frac{r^THr}{r^TH^2r}.
\end{equation}

Given $w_0 \in \mathbb{R}^m$, set $x_0 =A^Tw_0$, $r_0 =b-Ax_0\in \mathbb{R}^m$. For any $k \geq 1$, let
\begin{equation} \label{thmeq3}
r_{k}=F_1(r_{k-1})= F^{\circ k}_1(r_0)= \overbrace{F_1 \circ F_1 \circ \cdots \circ F_1}^{k ~ \text{times}}(r_0),
\end{equation}
the $k$-fold composition of $F_1$, where we assume for $j=0, \dots, k-1$, $Hr_j \not =0$.  Also for any $k \geq 1$ define,

\begin{equation} \label{thmeq4s1}
x_{k}= x_{k-1} +\begin{cases}
\alpha_{1,1}(r_{k-1}) r_{k-1}, & \text{if ~} H=A\\
\alpha_{1,1}(r_{k-1}) A^Tr_{k-1}, & \text{if~} H=AA^T.
\end{cases}
\end{equation}

\noindent (I) For any $k \geq 0$ we have
\begin{equation}  \label{rescorect}
r_{k}=b-Ax_{k}.
\end{equation}

\noindent (II) Suppose $Ax=b$ is solvable. Then

\begin{equation} \label{eqrate2}
\Vert r_k \Vert \leq \bigg (\frac{\kappa^+-1}{\kappa^++1} \bigg)^{k}
\Vert r_0 \Vert, \quad \forall k \geq 1.
\end{equation}

\noindent (III) Moreover, given $\varepsilon \in (0,1)$, for some $k$ iterations of $F_1$, satisfying
\begin{equation}  \label{epbond}
k \leq \kappa^+ \ln \bigg (\frac{\Vert r_0 \Vert}{\varepsilon} \bigg ),
\end{equation}
\begin{equation}  \label{bdonrk}
\Vert r_k \Vert \leq \varepsilon.
\end{equation}

\noindent (IV) Regardless of the solvability of $Ax=b$,
\begin{equation} \label{eqbounderrr77}
\Vert r_k \Vert^2 \leq \Vert r_0 \Vert^2 - c(H) \times  \sum_{j=0}^{k-1}
r_{j}^T H r_{j}.
\end{equation}

\noindent (V)
Moreover, given $\varepsilon \in (0,1)$, for some $k$ iterations of $F_1$, satisfying
\begin{equation} \label{iterbdsing}
k \leq \frac{\kappa^+}{\varepsilon} \lambda_{\max} \Vert r_0 \Vert^2,
\end{equation}
\begin{equation} \label{iterbdsingtone}
r_k^THr_k  \leq \varepsilon.
\end{equation}

\noindent (VI) The corresponding approximate solutions and error bounds are:

\begin{itemize}

\item
If $\Vert r_k \Vert \leq \varepsilon$, then $x_k$, defined in (\ref{thmeq4s1}), satisfies
\begin{equation}  \label{thm4case1}
\Vert A x_k - b \Vert \leq \varepsilon.
\end{equation}

\item
If $\Vert r_k \Vert > \varepsilon$ but $r_k^THr_k \leq \varepsilon$, then $\Vert A^Tr_k \Vert =O(\sqrt{\varepsilon})$. More precisely,
\begin{equation} \label{thmeq4gggnew}
\Vert A^TA x_{k} - A^Tb \Vert \leq \begin{cases}
\sqrt{\varepsilon} \Vert A \Vert^{1/2} , & \text{if ~} H=A\\
\sqrt{\varepsilon}, & \text{if~} H=AA^T.
\end{cases}
\end{equation}
\end{itemize}
\end{theorem}

\begin{proof}

\noindent (I): We prove (\ref{rescorect}) by induction on $k$.  By definition it is true for $k=0$. Assume true for $k-1$. Multiplying (\ref{thmeq4s1}) by $A$, using the definition of $H$,  $F_1(r)$ and $\alpha_{1,1}(r)$, we get
$$Ax_k=Ax_{k-1} +r_{k-1}- F_1(r_{k-1}).$$
Substituting for $r_{k-1}= b - Ax_{k-1}$, and for $F_1(r_{k-1})=r_k$, we get  $Ax_k=b - r_k$. Hence the statement is true for $k$.

\noindent (II): Since $Ax=b$ is solvable, $AA^Tw=b$ is solvable. Since $x_0=A^Tw_0$, $r_0=b-AA^Tw_0$. From Proposition \ref{syst}, $r_{k-1}$ is a linear combination of eigenvectors of $H$ corresponding to positive eigenvalues and since we assumed $Hr_{k-1} \not =0$, from Theorem \ref{thm1}, (\ref{coreq1}) we have
\begin{equation}
\Vert r_k \Vert \leq \bigg (\frac{\kappa^+ -1}{\kappa^++1} \bigg )\Vert r_{k-1} \Vert.
\end{equation}
The above inequality can be written for $k-1, k-2, \dots, 1$. Hence the  proof of (\ref{eqrate2}).

\noindent (III):
To get a bound on the number of iteration in (\ref{epbond}), we use the well-known inequality $ 1+ \alpha \leq \exp(\alpha)$. Using this,
$$ \frac{\kappa^+ -1}{\kappa^+ +1} = 1 - \frac{2}{\kappa^++1} \leq \exp \bigg ( \frac{-2}{ \kappa^++1} \bigg).$$
To get $\Vert r_k \Vert \leq \varepsilon$ it suffices to have
$$\exp \bigg (\frac{-2k}{\kappa^++1} \bigg ) \Vert r_0 \Vert \leq \varepsilon.$$
This implies the bound in (\ref{bdonrk}).

\noindent (IV):
From the assumption $Hr_{k-1} \not =0$, and Theorem \ref{thm1}, (\ref{coreq1}),
we have
\begin{equation} \label{eqbounderrr}
\Vert r_k \Vert^2 \leq \Vert r_{k-1} \Vert^2 - c(H) \times r_{k-1}^T H r_{k-1}.
\end{equation}
From repeated applications of (\ref{eqbounderrr}) for $k-1, \dots, 1$, we get (\ref{eqbounderrr77}).

\noindent (V):  From (\ref{eqbounderrr77}), it follows that if for $j =0, \dots, k-1$,
$r_{j}^T H r_{j} > \varepsilon$, then
\begin{equation} \label{eqbounder2}
 \Vert r_k \Vert^2
\leq  \Vert r_0 \Vert^2 - k \times c(H) \varepsilon.
\end{equation}
Since $\Vert r_k \Vert \geq 0$, from (\ref{eqbounder2}) we get the bound (\ref{iterbdsing}).

\noindent (VI):
We have shown
$x_k$, defined recursively in (\ref{thmeq4s1}),
satisfies $r_k=b-Ax_k$. This proves (\ref{thm4case1}).

Suppose $\Vert r_k \Vert > \varepsilon$ but $r_k^THr_k \leq \varepsilon$. If $H=AA^T$, then  $\Vert A^T r_k \Vert \leq \sqrt{\varepsilon}$ and hence  (\ref{thmeq4gggnew}) is satisfied for $H=AA^T$.  If $H=A$,  $r_k^THr_k= \Vert H^{1/2} r_k \Vert^2 \leq \varepsilon$. Thus
\begin{equation}
\Vert H r_k \Vert = \Vert H^{1/2} H^{1/2} r_k \Vert \leq
 \Vert H^{1/2} \Vert \times \Vert H^{1/2} r_k \Vert  \leq \Vert H \Vert ^{1/2} \sqrt{\varepsilon}.
\end{equation}
Hence (\ref{thmeq4gggnew}) is satisfied for $H=A$ as well.
This completes the proof of the theorem.
\end{proof}

\begin{corollary} \label{corsmall} Let $r_k=b-Ax_k$ and $x_k$ be defined as in Theorem \ref{thm2}.

(1) Suppose ${\rm rank}(A)=m$. Then, $x_k$ converges to a solution of $Ax=b$.  If $H=AA^T$, $x_k$ converges to the minimum-norm solution of $Ax=b$.

(2) $A^Tr_k$ converges to zero.

(3) Suppose $Ax=b$ is solvable.  If $\{x_k\}$ is bounded, then $r_k$ converges to zero. Moreover, any accumulation point of $x_k$'s is a solution of $Ax=b$.

(4) Suppose $Ax=b$ is solvable, $H=AA^T$, and $x_0=A^Tw_0$ for some $w_0 \in \mathbb{R}^m$.  Then, for all $k \geq 1$, $x_k=A^Tw_k$ for some $w_k \in \mathbb{R}^m$. If $\{w_k\}$ is a bounded, $x_k$ converges to the minimum-norm solution $x_*$.
\end{corollary}

\begin{proof}

(1): Suppose $H=A$. Then $x_k= A^{-1}b -A^{-1} r_k$. Since the bound on  $r_k$ given in (\ref{eqrate2}) implies that $r_k$ converges to zero, $x_k$ converges to $A^{-1}b$.  Suppose $H=AA^T$. Since $AA^T$ is invertible, $AA^Tw=b$ is solvable. Starting from $x_0=A^Tw_0$, we can use the formulas for $x_k$ (see (\ref{thmeq4s1})) and induction on $k$ to conclude that for all $k \geq 1$, $x_k=A^Tw_k$ for some $w_k$.
Now consider $r_k=b-Ax_k = b - AA^Tw_k$. Since $r_k$ converges to zero and $AA^T$ is invertible, it follows that $w_k$ converges to some $\overline w$, and hence,  $x_k$ converges to some $\overline x=A^T \overline w$. Given that $A \overline x=A A^T\overline w=b$, by Proposition \ref{minnorm}, we can conclude that $\overline x$ is the minimum-norm solution to $Ax=b$.

(2): If $r_k^T H r_k=0$ for some $k$, then in either case of $H$, $A^Tr_k=0$.  Suppose $r_k^T H r_k >0$ for all $k$. Then from (\ref{eqbounderrr77}), it follows that
\begin{equation}
C(H) \sum_{k=0}^\infty r_k^T H r_k \leq \Vert r_0 \Vert^2.
\end{equation}
Since the infinite series above has positive terms and stays bounded, it must be the case that $r_k^T H r_k$ converges to zero.  This implies that in either case of $H$, $A^Tr_k$ converges to zero.

(3): From (\ref{eqbounderrr77}), it follows that $r_k$ is a bounded sequence. Consider any accumulation point $\overline{r}$.  Let $S_1$ be the set of indices of  corresponding  subsequence of $r_k$'s converging to $\overline{r}$. Since $x_k$'s are bounded, there exists a subset $S_2 \subseteq S_1$ such that the corresponding $x_k$'s converge to $\overline{x}$. By part (1), as $k$ ranges in $S_2$,
$A^Tb - A^TAx_k$ converges to $A^Tb - A^TA\overline{x} = 0$. Since $Ax=b$ is solvable, it implies that $\overline{r} = 0$. Since $\overline{r}$ is any accumulation point of $r_k$'s, we conclude that $\{r_k \}$ itself converges to zero.

(4): By induction, starting from $x_0 = A^Tw_0$ and using the definition of $x_k$ in (\ref{thmeq4s1}), it follows that $x_k = A^Tw_k$ for some $w_k$. If  $\{w_k\}$ is bounded, then$\{x_k\}$ is also bounded. From part (2), any accumulation point of $x_k$'s, denoted as $\overline{x}$, is a solution to $Ax=b$. Since $\overline{x}$ must equal $A^T\overline{w}$, for some accumulation point, $\overline{w}$, of $\{w_k\}$, we have $AA^T\overline{w} = b$. Hence, $\overline{x} = x_*$.
\end{proof}

\section{High-Order Centering Triangle Algorithm (CTA)} \label{sec4}

In this section, we introduce a family of iterative methods designed for solving the linear system $Ax=b$ when a solution exists. These methods are extended to address the normal equation $A^TAx=A^Tb$ when the linear system is unsolvable. This family comprises generalized versions of $F_1(r)$, and its development not only enables the utilization of individual members as iteration functions but also facilitates the application of the entire family in a sequential manner.
This can be achieved by evaluating the family members using a single arbitrary input $r_0=b-Ax_0$. We shall refer to the family as high-order {\it Centering Triangle Algorithm} (CTA).

Before formally defining the high-order family, we provide a definition and state and prove a proposition that will be used to establish properties of the family.

\begin{definition} \label{defcharact} Let $H$ be an $m \times m$ symmetric positive semidefinite matrix (PSD).  Given $r \in \mathbb{R}^m$, its {\it minimal polynomial with respect to $H$}, denoted by $\mu_{\!_{H,r}}(\lambda)$,  is the monic polynomial of least degree $s$:
\begin{equation} \label{Hrmu}
\mu_{\!_{H,r}}(\lambda)=a_0+a_1 \lambda+ \dots +a_{s-1} \lambda^{s-1} + \lambda^s,
\end{equation}
such that $\mu_{\!_{H,r}}(Hr)=a_0 r+a_1Hr+ \dots + a_{d-1} H^{s-1}r+ H^sr=0$.
\end{definition}

\begin{proposition} \label{prop1} The degree of $\mu_{\!_{H,r}}(\lambda)$ is $s$ if and only if $r$ can be represented as
$r=\sum_{i=1}^s x_i u_i$, where $x_i$'s are nonzero and $u_i$'s are eigenvectors corresponding to $s$ distinct eigenvalues of $H$.
\end{proposition}


\begin{proof} Suppose $r = \sum_{i=1}^s x_i u_i$, $x_i \not =0$, $u_i$'s eigenvectors of $H$ corresponding to distinct eigenvalues, $\lambda_1, \dots, \lambda_s$.  Consider the monic polynomial with roots $\lambda_i$, $i=1, \dots, s$. This polynomial can be expressed as:
\begin{equation}
p(\lambda)=(\lambda - \lambda_1)\times  \cdots \times(\lambda - \lambda_s)= a_0+ a_1 \lambda+ \dots+ a_{s-1} \lambda^{s-1}+ \lambda^s.
\end{equation}
Since $r=\sum_{i=1}^s x_i u_i$, $H^ir= \sum_{i=1}^s x_i \lambda^i u_i$. Thus,
\begin{equation}
p(Hr)=a_0r+a_1Hr+ \cdots + a_{d-1}H^{d-1}r+H^dr = \sum_{i=1}^s x_ip(\lambda_i)u_i.
\end{equation}
Since $p(\lambda_i)=0$ for $i=1, \dots, s$, it follows that $p(Hr)=0$.

Conversely, suppose $\mu_{\!_{H,r}}(Hr)=a_0 r+a_1Hr+ \dots + a_{d-1} H^{s-1}r+ H^sr$. We will show $r$ is a linear combination of $s$ eigenvectors of $H$ corresponding to distinct eigenvalues. On the one hand, by the first part of the proof, $r$ cannot be such a linear combination with fewer than $s$ eigenvectors.  On the other hand, suppose $r$ is a linear combination of $t >s$ eigenvectors of $H$, i.e., $r=\sum_{i=1}^t x_i u_i$. Substituting this for $r$  and $H^ir= \sum_{i=1}^s x_i \lambda^i u_i$ for $i=1, \dots, t$, into $\mu_{\!_{H,r}}(Hr)$ we get:
\begin{equation}
\mu_{\!_{H,r}}(Hr)=a_0 r+a_1Hr+ \dots + a_{d-1} H^{s-1}r+ H^sr= \sum_{i=1}^y x_i \mu_{\!_{H,r}}(\lambda_i) u_i=0.
\end{equation}
Since $u_i$'s are linearly dependent, we get $x_i \mu_{\!_{H,r}}(\lambda_i)=0$. Since $x_i \not =0$ for $i=1, \dots, t$, $\mu_{\!_{H,r}}(\lambda_i)=0$ for $i=1, \dots, t$.  But this implies $\mu_{\!_{H,r}}(\lambda)$, which is a polynomial of degree $s$, has $t >s$ roots, a contradiction.
\end{proof}

\subsection{Algebraic Development of High-Order CTA} \label{sec4.1}
Here, we will develop the $t^{\text{th}}$-order iteration function of the CTA family, $t=1, 2, \dots, m$, formally denoted as $F_t(r)$. On the one hand, given any nonzero $r_0 =b-Ax_0 \in \mathbb{R}^m$ and $t \in {1, \dots, m}$, the iterations of $F_t(r)$ can either solve $Ax=b$ or $A^TAx=A^Tb$. On the other hand, given any nonzero $r_0 =b-Ax_0 \in \mathbb{R}^m$, by evaluating the sequence $\{F_t(r_0)\}_{t=1}^m$, we will be able to determine a solution of $Ax=b$ or its normal equation. In fact, it suffices to evaluate the sequence up to $s$,  the degree of the minimum polynomial of $r_0$ with respect to $H$. We will prove that $Ax=b$ is solvable if and only if $F_s(r_0)=0$. Exclusively, $A^TAx=A^Tb$ is solvable if and only if $A^TF_s(r_0)=0$ but $F_s(r_0) \neq 0$. Thus, the first index for which one of the equations is satisfied determines the value of $s$.

To develop the formula for $F_t$ at a given $r \in \mathbb{R}^m$, as before, let $H=AA^T$. If $A$ is symmetric PSD, we can let $H=A$. We assume $Hr \not = 0$.
For $j=0, \dots, 2t$, set
$\phi_j=\phi_j(r)={r}^T\!\!H^jr$. Define
\begin{equation}
 F_t(r)=r - \sum_{i=1}^t \alpha_{t,i}(r) H^ir,
 \end{equation}
where, $\alpha_t(r)=(\alpha_{t,1}(r), \dots, \alpha_{t,t}(r))^T$ is the {\it minimum-norm} optimal solution to the minimization of
\begin{equation}
E_t(\alpha)=E_t(\alpha_1, \dots, \alpha_t) =\Vert r - \sum_{i=1}^t \alpha_i H^ir \Vert^2, \quad \alpha =(\alpha_1, \dots, \alpha_t)^T \in \mathbb{R}^t.
\end{equation}

Clearly, the minimum of $E_t(\alpha)$ is attained.  Since $E_t(\alpha)$ is convex and differentiable over $\mathbb{R}^t$, any minimizer is a stationary point. Expanding,
\begin{equation}
E_t(\alpha)
= \phi_0^2 + \sum_{i=1}^l \alpha_i^2 \phi_{2i} - 2 \sum_{ 1 \leq i \not =j \leq t} \alpha_i \alpha_j  \phi_{i+j}.
\end{equation}
For $i=1, \dots, t$, setting the partial derivatives of $E_t(\alpha)$  with respect to $\alpha_i$ equal to zero, we get the following $t \times t$ linear system, referred as {\it auxiliary equation}:

\begin{equation} \label{eqcoeff}
\begin{pmatrix}
\phi_2&\phi_3& \ldots&\phi_{t+1}\\
\phi_3&\phi_4& \ldots&\phi_{t+2}\\
\vdots&\vdots&\ddots&\vdots\\
\phi_{t+1}&\phi_{t+2}& \ldots&\phi_{2t}\\
\end{pmatrix}
\begin{pmatrix}
\alpha_1\\
\alpha_2\\
\vdots\\
\alpha_t\\
\end{pmatrix}
=
\begin{pmatrix}
\phi_1\\
\phi_2\\
\vdots\\
\phi_t\\
\end{pmatrix}.
\end{equation}
Let the coefficient matrix in (\ref{eqcoeff}) be denoted by  $M_t(r)=(m_{ij}(r))$. Then  $m_{ij}(r)=\phi_{i+j}(r)=r^TH^{i+j}r$, where $1 \leq i,j \leq t$. Let $\beta_t(r)=(\phi_1(r), \dots, \phi_t(r))^T$, and $\alpha=(\alpha_1, \dots, \alpha_t)^T$. Then $\alpha_t(r)=(\alpha_{t,1}(r), \dots, \alpha_{t,t}(r))^T$ is the minimum-norm solution to the linear system:
\begin{equation}
M_t(r) \alpha = \beta_t(r).
\end{equation}

\begin{example} When $t=2$ and $M_2(r)$ is invertible we get:
\begin{equation} \label{thmeq200XX}
F_2(r)=r-\frac{\phi_1(r) \phi_4(r)-\phi_2(r) \phi_3(r)}{\phi_2(r) \phi_4(r)- \phi_3^2(r)}Hr-\frac{\phi^2_2(r) - \phi_1(r) \phi_3(r)}{\phi_2(r) \phi_4(r)- \phi^2_3(r)}H^2r.
\end{equation}
For the same input of Example \ref{example1}, namely $H=A={\rm diag}(1,\cdots, m)$, $r_0=(1, \dots, 1)^T$, setting  $m=100$ it can be shown $\Vert F_2(r_0) \Vert /\Vert r_0 \Vert \approx .327$.  This is an improvement over  $\Vert F_1(r_0) \Vert /\Vert r_0 \Vert \approx .497$.
\end{example}

Each iteration of $F_t$ requires solving the auxiliary equation. We observe the following property on $M_t(r)$.

\begin{remark}
Given any residual $r$ where $Hr \not =0$,
the matrix $M_t(r)$ is a {\it positive} (i.e. all its entries are positive).
Also, it can be shown  $M_t(r)$ is symmetric PSD (see \ref{vand1}). It can be shown that there is a unique diagonal matrix $D$ with positive diagonal entries such that $DM_t(r)D$ is
{\it doubly stochastic}. In theory, this can be achieved in polynomial-time in two different ways. Firstly,  as a positive matrix, via the ellipsoid method as proved in \cite{KalKhaDAD1}, applicable to positive matrices. Secondly, this diagonally scaling can be achieved via the path-following interior-point algorithm developed in \cite{KalKhaDAD3}.  However, one could also compute diagonal scaling via the RAS (row-column scaling) method, a {\it fully polynomial-time approximation scheme}, see \cite{KAlKhaDAD2}.

Thus, to solve the auxiliary equation, we could use the RAS method to diagonally scale $M_t(r)$ as a preconditioning. This is in the sense that all entries of the scaled matrix will lie within the interval $(0,1)$, without all being arbitrarily small. In practice, $t$ need not be large. In fact $t \leq 10$ may suffice.
Once such a diagonal matrix $D$, or an approximation of it, is at hand, instead of solving $M_t(r) \alpha = \beta(r)$, we solve for $z$ in $DM_t(r)D z=D\beta$ and set $\alpha(r)=Dz$.  Indeed pre and post multiplication of a positive definite matrix  by diagonal matrices in order to improve its condition number is a preconditioning technique used in linear system, see e.g. Braatz and Morari \cite{BRAATZ94}. However, the utility of the above-mentioned diagonal scaling of $M_t(r)$ as means of improving condition number remains to be investigated theoretically.
\end{remark}

We will next examine the convergence properties under iterations of $F_t$. First a definition on terminology:

\begin{definition}
Given $t \in \{1, \dots, m\}$,  the {\it orbit} at $r_0 \in \mathbb{R}^m$ of $F_{t}(r)$ is the infinite sequence
\begin{equation}
O^+_t(r_0)=\{ r_k =F_{t}(r_{k-1})\}_{k=1}^\infty.
\end{equation}
The {\it point-wise orbit} at $r_0 \in \mathbb{R}^m$ is  the finite sequence
\begin{equation}
O^*(r_0)=\{F_{t}(r_0)\}_{t=1}^m.
\end{equation}
The corresponding  {\it least-squares residual orbit} at $r_0$ and {\it least-squares residual point-wise orbit}  are defined as
\begin{equation}
O^+_t(r_0, A^T)=\{ A^Tr_k =A^TF_{t}(r_{k-1})\}_{k=1}^\infty, \quad
O^*(r_0,A^T)=\{A^TF_{t}(r_0)\}_{t=1}^m.
\end{equation}
\end{definition}

\subsection{Relating Magnitudes of Consecutive Residuals in Iterations of High-Order CTA}\label{sec4.2}

The following auxiliary theorem generalizes Theorem \ref{thm1} proved for first-order CTA.

\begin{theorem} \label{thm3} With $A$ an $m \times n$ matrix, let $H=AA^T$. If $A$ is symmetric PSD, we can let $H=A$.  Let $\kappa^+$ be the ratio of the largest eigenvalue, $\lambda_{\max}$, of $H$ to its smallest positive eigenvalue.  Let $c(H)={1}/({\kappa^+ \lambda_{\max}})$.

\noindent (i)  Given $t \in \{1, \dots, m\}$, suppose $r \in \mathbb{R}^m$ is a linear combination of eigenvector of $H$ corresponding to positive eigenvalues. Assume $Hr \not =0$. Then,

\begin{equation} \label{thm3eq1}
\Vert F_t(r) \Vert \leq   \bigg (\frac{\kappa^+-1}{\kappa^++1} \bigg )^t \Vert r \Vert.
\end{equation}

\noindent (ii) Given arbitrary $r \in \mathbb{R}^m$, let $F_1^{\circ 0}(r)=r$ and for $j \geq 1$, let $F_1^{\circ j}(r)$ be the $j$-fold composition of $F_1$ with itself at $r$. Suppose for $j=0, \dots, t-1$, $F^{\circ j}_1(r)^THF^{\circ j}_1(r) >0$. Then
\begin{equation} \label{thm32eqgeneral}
\Vert F_t(r) \Vert^2  \leq   \Vert r \Vert^2   -  c(H) \bigg (\sum_{j=0}^{t-1}F^{\circ j}_1(r)^THF^{\circ j}_1(r) \bigg ).
\end{equation}
\end{theorem}

\begin{proof}

(i):  We have already proved (\ref{thm3eq1}) for $t=1$, see Theorem \ref{thm1}.
We claim:
\begin{equation} \label{claimt}
\Vert F_t(r) \Vert \leq \Vert F_1^{\circ t}(r) \Vert.
\end{equation}
Using a straightforward induction argument and the definition of $F_1(r)$, we get the following equation:
\begin{equation}
F_1^{\circ t}(r) =r- \sum_{i=1}^t  \beta_{t,i}H^ir,
\end{equation}
where $\beta_{t,i}$'s are a set of constants dependent on $r$.
By the definition of $F_t(r)$ and the fact that $\alpha_r(t)$ (the solution of the auxiliary system) optimizes over all constant coefficients of $H^ir$ terms for $i=1, \dots, t$, we can conclude that (\ref{claimt}) follows naturally. Using (\ref{claimt}), the fact that for $t=1$, (\ref{thm3eq1}) is already proven, and by induction on $t$, for arbitrary $t \leq m$ we get,
\begin{equation} \label{claimt2}
\Vert F_t(r) \Vert \leq \Vert F_1(F^{\circ (t-1)}_1(r)) \Vert \leq \bigg (\frac{\kappa^+-1}{\kappa^++1} \bigg ) \Vert F^{\circ (t-1)}_1(r) \Vert \leq \bigg (\frac{\kappa^+-1}{\kappa^++1} \bigg )^t \Vert r \Vert.
\end{equation}

(ii): We have already proved (\ref{thm32eqgeneral}) for $t=1$, see Theorem \ref{thm1}.  Using (\ref{claimt}) and  the case of $t=1$, for arbitrary $t \leq m$ we have,
$$
\Vert F_t(r) \Vert^2 \leq \Vert F_1(F^{\circ (t-1)}_1(r)) \Vert^2 \leq \Vert F^{\circ (t-1)}_1(r) \Vert^2 -c(H) F^{\circ (t-1)}_1(r)^THF^{\circ (t-1)}_1(r) \leq
$$
$$
\Vert F^{\circ (t-2)}_1(r) \Vert^2 -c(H) \bigg (F^{\circ (t-1)}_1(r)^THF^{\circ (t-1)}_1(r)  +
F^{\circ (t-2)}_1(r)^THF^{\circ (t-2)}_1(r) \bigg) \leq \cdots \leq $$
\begin{equation} \label{firsteq2}
\Vert r \Vert^2   -  c(H) \bigg(\sum_{j=0}^{t-1}F^{\circ j}_1(r)^THF^{\circ j}_1(r) \bigg).
\end{equation}
\end{proof}

\subsection{Convergence Properties of High-Order CTA} \label{sec4.3}

In Theorem \ref{thm4} and Theorem \ref{thm4part2}, we prove the main result on the convergence properties of the orbit and point-wise orbit of the family $\{F_t\}$. Some parts of Theorem \ref{thm4part2} generalize Theorem \ref{thm2}, which was proved for $t=1$. The proof of some parts is analogous to corresponding parts in Theorem \ref{thm2}. However, the proof of other parts, by virtue of their generality, is more technical than the proof of the special case of $t=1$. Moreover, Theorem \ref{thm4part2} establishes properties that are attributed to the point-wise evaluation of the sequence $\{F_t, t=1, \dots, m\}$. We first state and prove an auxiliary result.

\begin{proposition} \label{propxx} Let $H$ be as before.
Given $r \in \mathbb{R}^m$, suppose for $j=0, \dots, t-1$, $F^{\circ j}_1(r)^THF^{\circ j}_1(r) >0$. Then for $j=1,\dots, t-1$ we have,
\begin{equation} \label{rprimegen22}
F_1^{\circ j}(r)= r- \sum_{i=0}^{j-1} \alpha_{1,1}\big (F_1^{\circ i}(r) \big ) H F_1^{\circ i}(r), \quad \alpha_{1,1}\big (F_1^{\circ i}(r) \big)= \frac{F_1^{\circ i}(r)^T H F_1^{\circ i}(r))}{F_1^{\circ i}(r)^T H^2 F_1^{\circ i}(r)}.
\end{equation}
If $r=b-Ax$, for $j=1 \dots, t-1$, then
\begin{equation} \label{xhathat}
A^TF_1^{\circ j}(r) =  A^Tb-A^TA \widehat x(j),
\end{equation}
where
\begin{equation} \label{xxhat}
\widehat x(j)= x+ \begin{cases}
\sum_{i=0}^{j-1} \alpha_{1,1}\big (F_1^{\circ i}(r) \big ) F_1^{\circ i}(r), & \text{if ~}
H=A\\
\sum_{i=0}^{j-1} \alpha_{1,1}\big (F_1^{\circ i}(r) \big ) A^T F_1^{\circ i}(r), & \text{if~} H=AA^T.
\end{cases}
\end{equation}
\end{proposition}

\begin{proof} For $j=1$, $F_1^{\circ 1}(r)=F_1(r)$ so that (\ref{rprimegen22}) gives $F_1(r)=r-\alpha_{1,1}(r)Hr$, as defined before. For arbitrary $j$, we prove this by induction.  For $j \geq 2$,  we have:
\begin{equation}
F_1^{\circ j}(r)= F_1(F_1^{\circ (j-1)}(r))=F_1^{\circ (j-1)}(r) - \alpha_{1,1} \big (F_1^{\circ (j-1)}(r) \big )H^{j-1}r.
\end{equation}
By induction hypothesis, (\ref{rprimegen22}) is true for $j-1$. Substituting for $F_1^{\circ (j-1)}(r)$, (\ref{rprimegen22}) is true for $j$.

Next we prove (\ref{xhathat}). Using (\ref{rprimegen22}), and substituting $r=b-Ax$  we have,
$$A^TF^{\circ j}(r)= A^Tb-A^TAx - A^T\sum_{i=0}^{j-1} \alpha_{1,1}\big (F_1^{\circ i}(r) \big ) H F_1^{\circ i}(r)=$$
\begin{equation}
A^Tb -\begin{cases}
A^TA \big (x + \sum_{i=0}^{j-1} \alpha_{1,1}\big (F_1^{\circ i}(r) \big ) F_1^{\circ i}(r) \big ), & \text{if ~}
H=A\\
A^TA \big (x + \sum_{i=0}^{\circ (j-1)} \alpha_{1,1}\big (F_1^{\circ i}(r) \big ) A^T F_1^{\circ i}(r) \big ), & \text{if~} H=AA^T.
\end{cases}= A^Tb-A^TA \widehat x(j).
\end{equation}

\end{proof}

\begin{theorem}  \label{thm4} Consider solving $Ax=b$ or $A^TAx=A^Tb$, where $A$ is an $m \times n$ matrix. Let $H=AA^T$. If $A$ is symmetric PSD, $H$ can be taken to be $A$ itself.  Let the  spectral decomposition of $H$ be $U\Lambda U^T$, where $U=[u_1, \dots, u_m]$ is the matrix of orthonormal eigenvectors, and $\Lambda={\rm diag}(\lambda_1, \dots, \lambda_m)$ the diagonal matrix of eigenvalues.  Let $\kappa^+$, be as before and $c(H)={1}/{\kappa^+ \lambda_{\max}}$. Given $t \in \{2, \dots, m\}$,   $r \in \mathbb{R}^m$ with $Hr \not =0$, let
\begin{equation}  \label{FTR}
F_t(r)= r - \sum_{i=1}^t \alpha_{t,i}(r) H^ir,
\end{equation}
where $\alpha_t(r)=(\alpha_{t,1}(r), \cdots, \alpha_{t,t}(r))^T$ is the solution to the auxiliary equation (\ref{eqcoeff}).

Given $x_0 \in \mathbb{R}^n$, let $r_0 =b-Ax_0\in \mathbb{R}^m$.
Given $t \leq m$, set $x_0(t)=x_0$, $r_0(t)=r_0$ and for any $k \geq 1$, let
\begin{equation} \label{thmeq3g}
r_{k}(t)=F_t(r_{k-1}(t))= F^{\circ k}_t(r_0)= \overbrace{F_t \circ F_t \circ \cdots \circ F_t}^{k ~~ \text{times}}(r_0),
\end{equation}
the $k$-fold composition of $F_t$, where we assume for $j=0, \dots, k-1$, $Hr_j(t) \not =0$. Also for any $k \geq 1$, let
\begin{equation} \label{thmeq4gen}
x_{k}(t)= x_{k-1}(t)  + \begin{cases}
\sum_{i=1}^t \alpha_{t,i} \big (r_k(t) \big ) H^{i-1} r_k(t), & \text{if ~} H=A\\
\sum_{i=1}^t \alpha_{t,i} \big (r_k(t) \big ) A^TH^{i-1} r_k(t), & \text{if~} H=AA^T.
\end{cases}
\end{equation}

\noindent (I) For any $k \geq 0$,
\begin{equation} \label{rkt}
r_{k}(t)=b-Ax_{k}(t).
\end{equation}

\noindent (II) Suppose $Ax=b$ is solvable. Then
\begin{equation} \label{eqrate2p}
\Vert r_k(t) \Vert \leq \bigg (\frac{\kappa^+-1}{\kappa^++1} \bigg)^{tk}
\Vert r_0 \Vert, \quad \forall k \geq 1.
\end{equation}

\noindent (III)
Given $\varepsilon \in (0,1)$, for some $k$ iterations of $F_t$, satisfying
\begin{equation}
k \leq \frac{\kappa^+}{t} \ln \bigg (\frac{\Vert r_0 \Vert}{\varepsilon} \bigg ),
\end{equation}
\begin{equation}
\Vert r_k(t) \Vert \leq \varepsilon.
\end{equation}

\noindent (IV)
Regardless of the solvability of $Ax=b$,
\begin{equation} \label{eqbounderrr88gg}
\Vert r_k(t) \Vert^2 \leq \Vert r_0 \Vert^2 - c(H) \bigg (\sum_{i=0}^{t-1}
\sum_{j=0}^{k-1}F^{\circ i}_1(r_{j}(1))^THF^{\circ i}_1(r_{j}(1)) \bigg).
\end{equation}

\noindent (V) Given $\varepsilon \in (0,1)$, for some $k$ iterations of $F_t$, satisfying
\begin{equation}
k \leq \frac{\kappa^+}{\varepsilon} \lambda_{\max} \Vert r_0 \Vert^2,
\end{equation}
\begin{equation} \label{twoineq}
r_k(t)^THr_k(t)  \leq \varepsilon.
 \end{equation}
Moreover, for some $k$  iterations of $F_t$, satisfying
\begin{equation} \label{bbbt}
k \leq \frac{\kappa^+}{t\varepsilon} \lambda_{\max} \Vert r_0 \Vert^2,
\end{equation}
at least one of the following $t$ conditions is satisfied,
\begin{equation} \label{tp1ineq}
 F^{\circ j}_1(r_k(t))^THF^{\circ j}_1(r_k(t))  \leq \varepsilon, \quad j=0, \dots, t-1.
 \end{equation}

\noindent (VI) The corresponding approximate solutions and error bounds are:  \\

\begin{itemize}

\item
If $\Vert r_k(t) \Vert \leq \varepsilon$, then $x_k(t)$, see (\ref{thmeq4gen}),
satisfies
\begin{equation} \label{thm4case1gen}
\Vert A x_k(t) - b \Vert \leq \varepsilon.
\end{equation}

\item
If $\Vert r_k(t) \Vert > \varepsilon$ but $r_k(t)^THr_k(t) \leq \varepsilon$, then
\begin{equation} \label{thmeq4gggnewgen}
\Vert A^TA x_{k}(t) - A^Tb \Vert \leq \begin{cases}
\sqrt{\varepsilon} \Vert A \Vert^{1/2} , & \text{if ~} H=A\\
\sqrt{\varepsilon}, & \text{if~} H=AA^T.
\end{cases}
\end{equation}

\item
If  $t \geq 2$ and  $\Vert r_k(t) \Vert > \varepsilon$, $r_k(t)^THr_k(t) > \varepsilon$ but
$\Vert F^{\circ j}_1(r_k(t))^THF^{\circ j}_1(r_k(t)) \Vert  \leq \varepsilon$, where $j$ is the smallest index from the set  $\{1, \dots, t-1\}$ such that the  inequality is satisfied, set
\begin{equation} \label{xxhatxx}
\widehat x_k(t)= x_k(t)+ \begin{cases}
\sum_{i=0}^{j-1} \alpha_{1,1}\big (F_1^{\circ i}(r) \big ) F_1^{\circ i}(r), & \text{if ~}
H=A\\
\sum_{i=0}^{j-1} \alpha_{1,1}\big (F_1^{\circ i}(r) \big ) A^T F_1^{\circ i}(r), & \text{if~} H=AA^T.
\end{cases}
\end{equation}
Then,
\begin{equation} \label{thmeq4ggg}
\Vert A^TA \widehat x_{k}(t) - A^Tb \Vert \leq \begin{cases}
\sqrt{\varepsilon} \Vert A \Vert^{1/2} , & \text{if ~} H=A\\
\sqrt{\varepsilon}, & \text{if~} H=AA^T.
\end{cases}
\end{equation}
\end{itemize}
\end{theorem}

\begin{proof}

\noindent (I): We prove (\ref{rkt}) by induction on $k$.  By definition, this holds for $k=0$. Assume true for $k-1$. Multiplying (\ref{thmeq4gen}) by $A$, using the definition of $H$, and  $F_t(r)$, we get:
$$Ax_k(t)=Ax_{k-1}(t) +r_{k-1}(t)- F_t(r_{k-1}(t)).$$
Substituting $r_{k-1}(t)= b - Ax_{k-1}(t)$, and $F_t(r_{k-1}(t))=r_k(t)$, we get $Ax_k(t)=b - r_k(t)$. Hence true for $k$.

\noindent (II):  Since $Ax=b$ is solvable, $AA^Tw=b$ is solvable. Since $x_0=A^Tw_0$, $r_0=b-AA^Tw_0$. From Proposition \ref{syst}, $r_{k-1}(t)$ is a linear combination of eigenvectors of $H$ corresponding to positive eigenvalues. Using this and the assumption $Hr_{k-1}(t) \not =0$,  from Theorem \ref{thm1}, (\ref{coreq1}) we have
\begin{equation}
\Vert r_k(t) \Vert \leq \bigg (\frac{\kappa^+ -1}{\kappa^++1} \bigg )^t\Vert r_{k-1}(t) \Vert.
\end{equation}
Thus, the above inequality can be written for $k-1, k-2, \dots, 1$. Hence, the  proof of (\ref{eqrate2p}) follows.  To prove (\ref{eqrate2p}),  in (\ref{thm3eq1}) of Theorem \ref{thm3} we replace $r$ with $r_{k-1}(t)$, using that $r_{k}(t)=F_t(r_{k-1}(t))$.

\noindent (III):  Analogous to the argument in Theorem \ref{thm2}, to have the right-hand-side of (\ref{eqrate2p}) bounded by $\varepsilon$, it suffices to choose $k$ so that
$\exp \big (-{2kt}/(\kappa^++1) \big) \Vert r_0 \Vert \leq \varepsilon$.

\noindent (IV):  To prove (\ref{eqbounderrr88gg}),  in (\ref{thm32eqgeneral}) of Theorem \ref{thm3} we replace $r$ with $r_{j-1}(t)$, using that $r_{j}(t)=F_t(r_{j-1}(t))$ and repeat this for $j=k, k-1, \dots, 0$.

\noindent (V): Replacing the right-hand-side of (\ref{eqbounderrr88gg})  by a larger summation, we have:
\begin{equation}
\Vert r_k(t) \Vert^2 \leq \Vert r_0 \Vert^2 - c(H) \sum_{j=0}^{k-1}
r_{j}(1)^T H r_{j}(1).
\end{equation}
Thus if $r_{j}(1)^T H r_{j}(1) > \varepsilon$ for each $j=0, \dots, k-1$, we get
$0 \leq \Vert r_0 \Vert^2 - c(H) \times  k \times \varepsilon$, implying the upper bound on $k$ in order to get one of the two inequalities in (\ref{twoineq}) satisfied.
If $F^{\circ i}(r_{j}(1))^T H F^{\circ i}(r_{j}(1)) > \varepsilon$ for all $i =0, \dots, t-1$, $j=0, \dots, k-1$, then from (\ref{eqbounderrr88gg}) we have
\begin{equation} \label{eqbounder2g}
\Vert r_k(t) \Vert^2 \leq \Vert r_0 \Vert^2 -   c(H) \times t \times k \times \varepsilon.
\end{equation}
This implies the bound  (\ref{bbbt}) on $k$ to get at least one of the inequalities in (\ref{tp1ineq}) satisfied.

\noindent (VI): The proof of (\ref{thm4case1gen}) and  (\ref{thmeq4gggnewgen}) follow from the fact that $r_k(t)=b-Ax_k(t)$, proved in (I).  Finally, Proposition \ref{propxx} implies (\ref{thmeq4ggg}).
\end{proof}

\begin{remark}
To summarize the remarkable convergence properties of iterations of $F_t$, when $Ax=b$ is solvable, whether or not $H$ is invertible, Algorithm \ref{3.1} terminates in $O((\kappa^+/t) \ln \varepsilon^{-1})$ iterations with an $\varepsilon$-approximate solution to $Ax=b$. Also, regardless of the solvability of $Ax=b$, in $O(\kappa^+ \lambda_{\max}/\varepsilon)$ iterations, it terminates with an $\sqrt{\varepsilon}$-approximate solution of $A^TAx=A^Tb$. Of course, it could be the case that the algorithm computes within this number of iterations an $\varepsilon$-approximate solution to $Ax=b$. By computing the augmented approximate solution $\widehat x_k(t)$, as defined in the theorem, we improve on the theoretical bound on the number of iterations by a factor of $t$. However, since the normal equation is always solvable, when the reduction in $r_k(t)$ appears to be small over a number of iterations, indicating that the system may be unsolvable,
we can apply the algorithm to solve the normal equation starting with $A^Tr_k(t)$. In fact, in certain cases, such as overdetermined systems, applying CTA directly to the normal equation can be more advantageous.  In summary,  CTA provides a highly robust set of iteration functions for solving arbitrary linear systems.
\end{remark}

The following  generalizes Corollary \ref{corsmall}. Its proof is analogous to the special case, hence omitted.

\begin{corollary} Let $r_k(t)=b-Ax_k(t)$ and $x_k(t)$ be defined as in Theorem \ref{thm4}.

(1) Suppose ${\rm rank}(A)=m$. Then $x_k(t)$ converges to a solution of $Ax=b$. If $H=AA^T$, $x_k(t)$ converges to the minimum-norm solution of $Ax=b$.

(2) For any $t \in \{1, \dots, m\}$, $A^Tr_k(t)$ converges to zero.

(3) Suppose $Ax=b$ is solvable.  If $\{x_k(t)\}$ is bounded, then $r_k(t)$ converges to zero. Moreover, any accumulation point of $x_k(t)$'s is a solution of $Ax=b$.

(4) Suppose $Ax=b$ is solvable, $H=AA^T$.  If  $x_0=A^Tw_0$ for some $w_0 \in \mathbb{R}^m$, then for $k \geq 1$, $x_k(t)=A^Tw_k(t)$ for some $w_k(t) \in \mathbb{R}^m$. If $\{w_k(t)\}$ is bounded, $x_k(t)$ converges to the minimum-norm solution $x_*$. When $t \geq 2$, for all $k \geq 1$, $\widehat x_k(t)=A^T \widehat w_k(t)$ for some $\widehat w_k(t) \in \mathbb{R}^m$. If $\{\widehat w_k(t)\}$ is  bounded, $\widehat x_k(t)$ converges to $x_*$.
\end{corollary}

The next theorem characterizes different properties of $F_t$, where the minimal polynomial of $H$ and that of an initial residual $r_0$ with respect to $H$ enter into the analysis.

\begin{theorem}  \label{thm4part2} Consider solving $Ax=b$ or $A^TAx=A^Tb$, where $A$ is an $m \times n$ matrix. Let $H=AA^T$. If $A$ is symmetric PSD, $H$ can be taken to be $A$. Let the  spectral decomposition of $H$ be $U\Lambda U^T$, where $U=[u_1, \dots, u_m]$ is the matrix of orthonormal eigenvectors, and $\Lambda={\rm diag}(\lambda_1, \dots, \lambda_m)$ the diagonal matrix of eigenvalues.  For $t \in \{1, \dots, m\}$, let $F_t(r)$ be as defined in Theorem \ref{thm4}.\\

Given $x_0 \in \mathbb{R}^n$, let $r_0=b-Ax_0 \in \mathbb{R}^m$. Assume $r_0 \not =0$ and  suppose its minimal polynomial with respect to $H$ has degree $s$. Then the following parts hold:

\noindent (1)  $F_s(r_0)=0$ if and only if $Ax=b$ is solvable. Moreover, a corresponding   solution is
\begin{equation} \label{solution}
x_*(s)= \begin{cases}
x_0 + \sum_{i=1}^s \alpha_{t,i}(r_0) H^{i-1} r_0, & \text{if ~} H=A\\
x_0 + \sum_{i=1}^s \alpha_{t,i}(r_0) A^TH^{i-1} r_0, & \text{if~} H=AA^T.
\end{cases}
\end{equation}
Additionally, if $x_0=A^Tw_0$ for some $w_0 \in \mathbb{R}^m$, $H=AA^T$, then $x_*(s)$ is the minimum-norm solution to $Ax=b$.

\noindent (2) $F_s(r_0) \not =0$ and $A^TF_s(r_0)=0$ if and only if $A^TAx=A^Tb$ is solvable and $Ax=b$ is unsolvable. Moreover, a corresponding solution to the normal equation is $x_*(s)$  in (\ref{solution}).

\noindent  (3)  Suppose the degree of the minimal polynomial of $H$ is $s$. Then, given any $x_0 \in \mathbb{R}^n$, setting $r_0=b-Ax_0 \in \mathbb{R}^m$, either $F_s(r_0)=0$ (solving $Ax=b$), or $A^TF_s(r_0)=0$ (solving $A^TAx=A^Tb$) with corresponding solution $x_*(s)$ above. Thus,  one iteration of $F_s$ solves $Ax=b$ or $A^TAx=A^Tb$.

\noindent (4)  Given $r \not =0  \in \mathbb{R}^m$, suppose its minimal polynomial with respect to $H$ has $s \geq 1$ distinct positive roots. Then, for any $t \leq s$ the coefficients of $F_t(r)$ are explicitly defined via the Cramer's rule. Specifically, the coefficient matrix of (\ref{eqcoeff}),
$M_t(r)=(\phi_{i+j}(r))=(r^TH^{i+j}r)$, is invertible so that
\begin{equation} \label{thmeq1}
F_t(r)= r - \sum_{i=1}^t \alpha_{t,i}(r) H^ir, \quad
\alpha_{t,i}(r)= \frac{|M_{t,i}(r)|}{|M_t(r)|}, \quad i=1, \dots, t,
\end{equation}
where $|\cdot|$ denotes determinant and $M_{t,i}(r)$ denotes the matrix that replaces the $i$-th column of $M_t(r)$ with $\beta_t(r)=(\phi_1(r), \dots, \phi_t(r))^T$. Moreover,
\begin{equation} \label{alphatt}
\alpha_{t,t}(r) \not =0.
\end{equation}

\noindent (5)  Given $t \in \{2, \dots, m\}$, for any   $r \in \mathbb{R}^m$,
$\Vert F_t(r) \Vert \leq  \Vert F_{t-1}(r) \Vert$.  If $F_{t-1}(r) \not =0$, then
\begin{equation} \label{eqrate0gg}
\Vert F_t(r) \Vert < \Vert F_{t-1}(r) \Vert.
\end{equation}
Also, for any $t > s$, we have
\begin{equation} \label{fts}
F_t(r)=F_s(r).
\end{equation}

\end{theorem}

\begin{proof}

\noindent (1):  Suppose $F_s(r_0)=0$.  Multiplying $x_*(s)$ in (\ref{solution}) by $A$, from the definition of $H$ and $F_t(r)$, see (\ref{thmeq1}),  the product of $A$ and the summation term gives $r_0-F_t(r_0)$.  Thus, we get the following, proving $x_*(s)$ is a solution to $Ax=b$:
\begin{equation}  \label{axseqzgen}
Ax_*(s)= Ax_0+r_0-F_s(r_0)= Ax_0+b -Ax_0=b.
\end{equation}
Conversely, suppose $Ax=b$ is solvable.

If $H$ is invertible, by Proposition \ref{prop1} the constant term of the minimal polynomial of $r_0$ with respect to $H$, namely $\mu_{H,r_0}(y)$, is nonzero. But this implies $F_s(r_0)=0$.  Thus, if $F_s(r_0) \not =0$, $H$ must be singular.

If $r_0^THr_0=0$, whether or not  $H=AA^T$ or $H=A$, $A^Tr_0=0$. But this implies $A^TAx_0=A^Tb$ so that $x_0$ is a solution to the normal equation. But since we assumed $Ax=b$ is solvable, $x_0$ also satisfies $Ax_0=b$, implying $r_0=0$, a contradiction. Hence $F_s(r_0)=0$.

If $F_1^T(r_0)HF_1(r_0)=0$, it follows that $A^TF_1(r_0)=0$. This implies $x_*(1)$ is a solution to the normal equation and hence a solution to $Ax=b$. Thus again $F_s(r_0)=0$. Similarly, if $F_t^T(r_0)HF_t(r_0)=0$ for some $t \in \{2, \dots, t-1\}$, then $A^TF_t(r_0)=0$. This implies $x_*(s)$ is a solution to the normal equation and hence a solution to $Ax=b$.

We have thus proved if $F_s(r_0) \not=0$, $r_0^THr_0>0$ and for all $t \in \{1, \dots, s-1 \}$, $F_t^T(r_0)HF_t(r_0) >0$.  Hence by  Theorem \ref{thm3}, (\ref{thm32eqgeneral}), substituting $F_s(r_0)$ for $r$, we conclude
\begin{equation}  \label{contradictgen}
\Vert F_s(F_s(r_0)) \Vert < \Vert F_s(r_0) \Vert.
\end{equation}
We prove strict inequality is not possible in (\ref{contradictgen}). Since the minimal polynomial of $r_0$ with respect to $H$ has degree $s \geq 1$, there exists a set of constants $a_0, \dots, a_s \not =0$, such that:
$$a_0r_0+ \cdots + a_sH^sr_0=0.$$
It follows that for $i >s$, $H^{i}r_0$ can be written as a linear combination of $r_0, Hr_0, \dots, H^sr_0$ so that for some set of constant $\beta_{s,1}, \dots, \beta_{s,s}$,
$$F_s(F_s(r_0))=r_0- \beta_{t,1}Hr_0 - \cdots - \beta_{t,t}H^sr_0.$$
But the above equation and definition of $F_t(r_0)$ contradict the strict inequality (\ref{contradictgen}). Thus $F_s(r_0)=0$.

To prove the last part of (1), if $x_0=A^Tw_0$ for some $w_0 \in \mathbb{R}^m$, where $H=AA^T$, then from the formula for $x_*(s)$, $x_*(s)=A^Tw_*(s)$, for some $w_*(s) \in \mathbb{R}^m$. Then, by Proposition (\ref{minnorm}), $x_*(s)$ is the minimum-norm solution.

\noindent (2):  Suppose $F_s(r_0) \not =0$ and $A^TF_s(r_0)=0$.  Then, from $A^TF_s(r_0)=0$, it is straightforward to show $A^TAx_*(s)=A^Tb$. Also, $Ax=b$ is not solvable since otherwise we must have $Ax_*(s)=b$,  a contradiction by part (1).

Conversely, suppose $A^TAx=A^Tb$ is solvable and $Ax=b$ is not solvable.  Then, from (1) it follows that $F_s(r_0) \not =0$.
We prove $A^TF_s(r_0)=0$. Since $Ax=b$ is unsolvable, $H$ must be singular. As in the case of (1), we claim
${F^{\circ j}_1(r_0)}^THF^{\circ j}_1(r_0) >0$ for $j=0, \dots, s-1$.  Otherwise, if
${F^{\circ j}_1(r_0)}^THF^{\circ j}_1(r_0) =0$ for some $j=0, \dots, s-1$, then  $A^TF^{\circ j}_1(r_0)=0$. This implies $x_*(s)$ is a solution to the normal equation.  Then from  Theorem \ref{thm3}, (\ref{thm32eqgeneral}), it follows that
\begin{equation}  \label{contradict1}
\Vert F_s(F_s(r_0)) \Vert < \Vert F_s(r_0) \Vert.
\end{equation}
But as in proof of case (1), this contradicts the definition of $F_s(r_0)$ so that we must have $A^TF_s(r_0)=0$.

\noindent (3):  If the degree of the minimal polynomial of $H$ is $s$,
then for any nonzero $r$, its minimal polynomial with respect to $H$ must have degree $s' \leq s$.  Thus, the proof follows from (1) and (2).

\noindent (4): From Proposition \ref{prop1}, it follows that if the minimal polynomial of $r$ with respect to $H$ has $s \geq 1$ positive roots, then either $r$ is a linear combination of $s$ eigenvectors with positive eigenvalues, or a linear combination of $s+1$ eigenvector, where one of the eigenvectors must corresponds to zero eigenvalue. Thus, we may assume
$r= u_0+ \sum_{i=1}^s x_i u_i$, where  either $u_0=0$ or $Hu_0=0$, and
the eigenvalues corresponding to the eigenvectors $u_1, \dots, u_s$ are positive. Thus, $x_i \not =0$ for $i=1, \dots, s$.  Without loss of generality, assume
$0 < \lambda_1 < \cdots < \lambda_m$ and $u_0=0$.  Now, it is easy to verify that the coefficient matrix $M_t(r)$ in the auxiliary equation can be written as:

\begin{equation} \label{vand1}
M_t(r)=VDV^T=
\begin{pmatrix}
\lambda_1&\lambda_2& \ldots&\lambda_s\\
\lambda_1^2&\lambda_2^2& \ldots&\lambda_s^2\\
\vdots&\vdots&\ddots&\vdots\\
\lambda_1^t&\lambda_2^t& \ldots&\lambda_s^t\\
\end{pmatrix}
\begin{pmatrix}
x_1^2&0& \ldots&0\\
0&x_2^2& \ldots&0\\
\vdots&\vdots&\ddots&\vdots\\
0&0& \ldots&x_s^2\\
\end{pmatrix}
\begin{pmatrix}
\lambda_1&\lambda_1^2& \ldots& \lambda_1^t\\
\lambda_2&\lambda_2^2& \ldots& \lambda_1^t\\
\vdots&\vdots&\ddots&\vdots\\
\lambda_s&\lambda_s^2& \ldots&\lambda_1^t\\
\end{pmatrix}.
\end{equation}

Thus, $V$ is $t \times s$ and $D={\rm diag}(x_1^2, \dots, x^2_s)$.
Suppose $M_t(r)$ is not invertible.  Then, there exists $y=(y_1, \dots, y_t)^T \not = 0$ such that $M_t(r)y=0$. We claim $V^T y \not =0$. Otherwise, we get
\begin{equation}
y_1 \lambda_i+ \cdots + y_t \lambda_i^t=0, \quad i=1, \dots, s.
\end{equation}
Since $\lambda_i >0$ for all $i=1, \dots, s$, dividing by $\lambda_i$  gives
\begin{equation}
y_1+ \cdots + y_t \lambda_i^{t-1}=0, \quad i=1, \dots, s.
\end{equation}
But the above implies $\lambda_i, i=1, \dots, s$ are all distinct roots of a polynomial of degree at most $t-1$. Since $s \geq t$, this proves the claim that $w=V^Ty \not =0$. This implies $w^TD^{1/2}w = y^TVDV^Ty=y^TM_t(r)y >0$, contradicting that $M_t(r)y=0$.  This proves $M_t(r)$ is invertible.

Next, we consider $\alpha_{t,t}$. Let $W$ be the matrix $V^T$ with its last column replaced with
$(1,1, \dots, 1)^T \in \mathbb{R}^s$, i.e.,
\begin{equation}
W=
\begin{pmatrix}
\lambda_1&\lambda_1^2& \ldots&\lambda_1^{t-1}&1\\
\lambda_2&\lambda_2^2& \ldots&\lambda_2^{t-1}&1\\
\vdots&\vdots&\ddots&\vdots\\
\lambda_s&\lambda_s^2& \ldots&\lambda_s^{t-1}&1\\
\end{pmatrix}.
\end{equation}
By definition, $M_{t,t}(r)$ is the matrix $M_t(r)$ with its $t$-th column replaced with $\beta_t(r)=(\phi_1(r), \cdots, \phi_t(r))^T$. Note that
\begin{equation}
\phi_j(r)=r^TH^jr=\big (\sum_{i=1}^s x_iu_i \big )^T H^j \big ( \sum_{i=1}^s x_iu_i \big )=\sum_{i=1}^s \lambda_i^j x_i^2.
\end{equation}
From the above it is easy to verify that
\begin{equation}
M_{t,t}(r) = V D W.
\end{equation}
By the Cramer's rule,
\begin{equation}
\alpha_{t,t}(r)= \frac{|M_{t,t}(r)|}{|M_t(r)|}.
\end{equation}
Thus, to prove (\ref{alphatt}) that $\alpha_{t,t}(r) \not =0$, is suffices to show $M_{t,t}(r)$ is invertible. We have,
\begin{equation}
VD= V_0 D_0,
\end{equation}
where $D_0={\rm diag}(\lambda_1 x_1^2, \dots, \lambda_s x_s^2)$ and $V_0$ is the $s \times t$ Vandermonde matrix:
\begin{equation}
V_0=
\begin{pmatrix}
1&1& \ldots&1\\
\lambda_1&\lambda_2& \ldots&\lambda_s\\
\vdots&\vdots&\ddots&\vdots\\
\lambda_1^{t-1}&\lambda_2^{t-1}& \ldots&\lambda_s^{t-1}\\
\end{pmatrix}.
\end{equation}
Given $y=(y_1, y_2, \dots, y_{t-1}, y_t)^T \not =0$,  permuting the coordinates by one, let $\overline y=(y_t, y_1, y_2, \dots, y_{t-1})^T$, also a nonzero vector.
It is straightforward to verify that $W y=V_0^T \overline y$.  Thus we have,
\begin{equation} \label{ttta}
 \overline y^T M_{t,t}(r)y= \overline y^T V_0 D_0 W y= \overline y^T V_0 D_0 V^T_0 \overline y= \Vert D_0^{1/2} V_0^T \overline y \Vert^2.
\end{equation}
We claim $V_0^T \overline y \not = 0$.  Then, assuming the correctness of the claim and since $D_0$ is positive definite, (\ref{ttta}) would imply
$M_{t,t}(r) y \not =0$ for any nonzero $y$, hence proving $M_{t,t}(r)$ must be invertible.  To prove the claim, suppose $V_0^T \overline y = 0$. Then, the polynomial $y_t+ y_1 \lambda + \cdots + y_{t-1} \lambda^{t-1}$ has $s \geq t$ distinct roots, a contradiction. This completes the proof of (4).

\noindent (5): The inequality $\Vert F_t(r) \Vert \leq \Vert F_{t-1}(r) \Vert$ follows from the definition, as  $F_t(r)$ optimizes the norm over a larger domain than $F_{t-1}(r)$ does.  To prove (\ref{eqrate0gg}), assume $F_{t-1}(r) = r - \sum_{i=1}^{t-1} \alpha_{t-1,i}(r) H^i r$ is nonzero. Suppose
\begin{equation} \label{eqfs}
\Vert F_{t-1}(r) \Vert = \Vert F_{t}(r) \Vert  = \Vert  r - \sum_{i=1}^{t} \alpha_{t,i}(r) H^i r \Vert.
\end{equation}
It is a trivial fact that if two nonzero vectors, $u$ and $v$, have the same length, their midpoint $(u+v)/2$ has a length smaller than their common length. Thus,  if (\ref{eqfs}) holds,
\begin{equation}
\Vert \frac{1}{2}F_{t-1}(r)+ \frac{1}{2}F_{t}(r) \Vert =
\Vert r - \sum_{i=1}^{t-1} \frac{1}{2} (\alpha_{t-1,i}+ \alpha_{t,i})H^i r  + \frac{1}{2} \alpha_{t,t})H^t r \Vert < \Vert F_t(r) \Vert,
\end{equation}
contradicting the definition of $F_t(r)$. Hence the proof of (\ref{eqrate0gg}).

To prove (\ref{fts}), since the minimal polynomial of $r$ with respect to $H$ has $s$ positive roots, this means either
$\mu_{\!_{H,r}}(Hr)=a_0 r+a_1Hr+ \dots + a_{d-1} H^{s-1}r+ H^sr=0$,
and $a_0 \not =0$, or
$\mu_{\!_{H,r}}(Hr)=a_1Hr+ \dots + a_{d-1} H^{s}r+ H^{s+1}r=0$.
If follows that for all $i$ satisfying $s < i \leq t$, $H^ir$ can be expressed in terms of $Hr, \dots, H^sr$. But this implies $F_t(r)=F_s(r)$.
\end{proof}

The following two theorems give an overview of the convergence properties of the orbits of CTA family.

\begin{theorem}
If $Ax=b$ is solvable, $O^+_t(r_0)$ converge to the zero vector and any accumulation point of the sequence of $x_k(t)$'s is a solution. If $Ax=b$ is not solvable, $O^+_t(r_0, A^T)$  converge to zero and any accumulation point of the sequence of $x_k(t)$'s is a solution to $A^TAx=A^Tb$. \qed
\end{theorem}

\begin{theorem}
If $Ax=b$ is solvable, the finite sequence $O^*(r_0)$ converge to the zero vector.
If $Ax=b$ is not solvable, $O^*(r_0,A^T)$ converge to zero.
Moreover, if the degree of the minimal polynomial of $r_0$ with respect to $H$ is $s$, and for any $ j \leq m$ we define
\begin{equation} \label{solutionzzz}
x_*(j)= \begin{cases}
x_0 + \sum_{i=1}^j \alpha_{t,i}(r_0) H^{i-1} r_0, & \text{if ~} H=A\\
x_0 + \sum_{i=1}^j \alpha_{t,i}(r_0) A^TH^{i-1} r_0, & \text{if~} H=AA^T,
\end{cases}
\end{equation}
then $x_*(s)$ is either a solution to $Ax=b$, or a solution to $A^TAx=A^Tb$, but for $j <s$, $x_*(j)$ is neither a solution to $Ax=b$ nor a solution to $A^TAx=A^Tb$.  \qed
\end{theorem}

\begin{remark} \label{minpolyfactor}
Given a symmetric positive definite matrix $H$, to compute a factor of its minimal polynomial, we can choose an arbitrary nonzero $b$, and an arbitrary $x_0$. Then, we set $r_0=b-Hx_0$ and attempt to solve the equation $Hx=b$ using the point-wise orbit method. The minimal polynomial of $r_0$ with respect to $H$, which is a factor of the minimal polynomial of $H$, can be computed by finding the smallest $j$ such that either $Ax_*(j)=b$ or $A^TAx_*(j)=A^Tb$. If this does not yield the minimal polynomial of $H$, we can obtain new factors by repeating the above steps with different choices of $r_0$. If $H$ is positive semidefinite, we can adjust $b$ so that $Hx=b$ is solvable and then choose arbitrary $x_0$.
\end{remark}

\subsection{Algorithms Based on Orbit and Point-wise Orbit of CTA} \label{sec4.4}

In this subsection we describe two algorithms for solving a linear system. Given a fixed $t \in \{1, \dots, m\}$, Algorithm \ref{4.1} is based on computing iterates of $F_t$.  To describe the iterative step of the algorithm, given the current residual $r'=b-Ax'$, the algorithm generates new residual $r''(t)$ and approximate solution $x''(t)$, satisfying $r''(t)=b-Ax''(t)$, defined below. Description of these require computing $\alpha_t(r)$, the solution of the auxiliary system,
$M_t(r) \alpha = \beta_t(r)$. The auxiliary system always has a solution. In case it has multiple solutions, we select the minimum-norm solution.  Now set
\begin{equation}  \label{CTAxnewt}
r''(t) = F_t(r')= r'- \sum_{i=1}^t \alpha_{t,i}(r') H^ir', \quad
x''(t)= x'+\begin{cases}
\sum_{i=1}^t \alpha_{t,i}(r') H^{i-1}r', & \text{if ~} H=A\\
\sum_{i=1}^t \alpha_{t,i}(r') A^T H^{i-1}r', & \text{if~} H=AA^T.
\end{cases}
\end{equation}
If $\Vert r''(t) \Vert \leq \varepsilon$ or $\Vert {r''(t)}^T H r''(t) \Vert \leq  \varepsilon$, the algorithm terminates with approximate solution $x''(t)$ defined above. Otherwise,  $\Vert r''(t) \Vert > \varepsilon$ and $\Vert {r''(t)}^T H r''(t) \Vert > \varepsilon$. If it exists, the algorithm selects the smallest index $j \in \{1, \dots, t-1 \}$ such that $\Vert F_1^{\circ j}(r''(t))^TH F_1^{\circ j}(r''(t)) \Vert \leq \varepsilon$. Then, it stops with the approximate solution as $\widehat x''_j(t)$, defined below (as before $\alpha_{1,1}(r)=r^THr/r^TH^2r$):
\begin{equation} \label{widehatxx}
\widehat x''_j(t)= x''(t) +\begin{cases}
\sum_{i=0}^{j-1} \alpha_{1,1} \big (F_1^{\circ i}(r') \big ) F_1^{\circ i}(r'), & \text{if ~}
H=A\\
\sum_{i=0}^{j-1} \alpha_{1,1} \big (F_1^{\circ i}(r') \big) A^T F_1^{\circ i}(r'), & \text{if~} H=AA^T
\end{cases}.
\end{equation}
It the algorithm has not terminated, it replaces $r'$ with $r''(t)$, $x'$ with $x''(t)$, and repeats.

Algorithm \ref{4.1} is fully described for any general matrix $A$ and any $t \leq m$. If it is known $H$ is invertible, whether $H=A$ or $H=AA^T$, the while loop of Algorithm \ref{4.1} only needs the clause  $\Vert r'  \Vert  > \varepsilon$, however as suggested for Algorithm \ref{3.1}, if the algorithm stops prematurely, we reduce the value of $\varepsilon$ for that clause and keep running the algorithm. Thus, Algorithm \ref{4.1} as is, regardless of any information on invertibility $H$ or solvability of $Ax=b$ is essentially complete.

In practice there is no need to start with a large $t$. One strategy is as follows: Given a residual $r$, we compute $F_1(r)$ and if this does not sufficiently reduce $\Vert r \Vert$, we compute $F_2(F_1(r))$ and if needed compute $F_3(F_2(F_1(r)))$ up to a small $t$, then we cycle.  Many other strategies are possible.

\begin{algorithm}[!htb]
\scriptsize
\SetAlgoNoLine
\KwIn{$A \in \mathbb{R}^{m \times n}$, $b \in \mathbb{R}^m$, $\varepsilon \in (0, 1)$, $t \in \{1, \dots, m\}$.}
$x' \gets 0$ (or $x'=A^Tw'$, $w' \in \mathbb{R}^m$ random), $r' \gets b-Ax'$

\While{ $ {\rm (}\|r'\| > \varepsilon{\rm )}$ $\wedge$  $ {\rm (}{r'}^THr' > \varepsilon {\rm )}$  $\wedge$  $ {\rm (} F_1(r')^THF_1(r') > \varepsilon  {\rm )}$   $\wedge$
$\cdots$ $\wedge$   $ {\rm (} F_1^{\circ (t-1)}(r')^THF_1^{\circ (t-1)}(r') > \varepsilon {\rm )}$}
{
$(\alpha_{t,1}(r'), \dots, \alpha_{t,t}(r'))^T \gets \text{minimum-norm solution of~~}  M_t(r') \alpha = \beta_t(r')$ (see (\ref{eqcoeff})),

$r' \gets r''(t)$, $x' \gets x''(t)$ (see (\ref{CTAxnewt}))}

\lIf{$ {\rm (} \Vert r' \Vert \leq \varepsilon {\rm )}$  $\vee$  $ {\rm (} \Vert {r'}^THr' \Vert \leq \varepsilon {\rm )}$} {STOP}
   \Else{$x' \gets \widehat x''_j(t)$ (see (\ref{widehatxx}))}
   \caption{(CTA) Iteration of $F_t$ for a fixed $t$, computing $\varepsilon$-approximate  solution of $Ax=b$ or $A^TAx=A^Tb$ (final $x'$ is approximate solution).} \label{4.1}
\end{algorithm}

\begin{remark}
As mentioned earlier, considering the theoretical iteration complexity of $F_t$, an alternative version of Algorithm \ref{4.1} can be implemented when $r_k(t)$ does not decrease sufficiently over a series of iterations. In such cases, we can switch to applying $F_t$ with respect to the normal equation, starting with $A^Tr_k(t)$. The decision to switch can be based on a user-defined number of iterations, while monitoring the improvement in the residual, or it can depend on whether the original system is overdetermined or underdetermined.
\end{remark}

Next, we describe Algorithm \ref{4.2}, which is based on the computation of the point-wise orbit at a specific $r_0=b-Ax_0 \neq 0$. It is not necessary to know the value of $s$, the degree of the minimal polynomial of $r_0$ with respect to $H$.

For each $t \leq s$, the while loop computes the unique solution $\alpha$ to the auxiliary equation, $M_t(r) \alpha = \beta_t(r)$. Then, it calculates $F_t(r)$ and the corresponding approximate solution $\overline x(t)$. If $|F_t(r)| \leq \varepsilon$ or $F_t(r)^THF_t(r) \leq \varepsilon$, the algorithm terminates.
If neither of these conditions is met, it increments $t$ and repeats. In the worst-case scenario, $t$ may reach $m$.  To define the algorithm, set

\begin{equation}  \label{CTApoint}
\overline r(t) = F_t(r_0)=r_0- \sum_{i=1}^t \alpha_{t,i}(r_0) H^ir_0, \quad
\overline x(t)= x_0+\begin{cases}
\sum_{i=1}^t \alpha_{t,i}(r_0)H^{i-1}r_0, & \text{if ~} H=A\\
\sum_{i=1}^t \alpha_{t,i}(r_0)A^T H^{i-1}r_0, & \text{if~} H=AA^T.
\end{cases}
\end{equation}

\begin{algorithm}[!htb]
\scriptsize
\SetAlgoNoLine
\KwIn{$A \in \mathbb{R}^{m \times n}$, $ b \in \mathbb{R}^m$, $\varepsilon \in (0, 1)$, $\text{Found}$ $\text{Boolean Variable}$ (In theory $\varepsilon=0$, see Remark \ref{rempoint}).}
$x_0 \gets \text{random non-zero point}  \in \mathbb{R}^n$, $r_0 \gets b-Ax_0$, $t \gets 1$, $\text{Found} \gets \text{False}$

\While{{\rm (}$ t \leq m$ {\rm)} $\wedge$ {\rm (} $\text{Found} = \text{False}$ {\rm )}}
{$(\alpha_{t,1}(r_0), \dots, \alpha_{t,t}(r_0))^T \gets \text{solution of~~}  M_t(r_0) \alpha = \beta_t(r_0)$ (see (\ref{eqcoeff})),

  \lIf{$\Vert F_t(r_0)\Vert \leq \varepsilon$} {$r_\varepsilon \gets \overline r(t)$, $x_\varepsilon \gets \overline x(t)$ (see  (\ref{CTApoint})), $\text{Found} \gets \text{True}$}
  \Else{\lIf{$F_t(r_0)^THF_t(r_0) \leq  \varepsilon$} {$r_\varepsilon \gets \overline r(t)$, $x_\varepsilon \gets \overline x(t)$ (see  (\ref{CTApoint})), $\text{Found} \gets \text{True}$}
  {\Else{$t \gets t+1$}}}}
\caption{(CTA) Computes approximate solution $x_\varepsilon$ of $Ax=b$ or $A^TAx=A^Tb$ via point-wise iteration of $\{F_t \}$.} \label{4.2}
\end{algorithm}

\begin{remark}   \label{rempoint} From Theorem \ref{thm4}, for some value $s$, the degree of the minimal polynomial of $r$ with respect to $H$, either $F_s(r)=0$ (solving $Ax=b$) or $A^TF_s(r)=0$ (solving $A^TAx=A^Tb$).  This means in theory the while loop will terminate even with $\varepsilon =0$, however in practice we select $\varepsilon$ to a desired accuracy.
\end{remark}

\subsection{Application of CTA to the Normal Equation} \label{sec4.5}
As suggested earlier, when applying CTA to $Ax=b$ does not sufficiently reduce the norm of the residual, giving indication that the equation is inconsistent,  we  apply CTA to the normal equation. Here we wish to show how the iterates of $F_t$ as applied to the normal equation will look like.  For a given residual $r'=b-Ax'$, let $\widehat r'= A^Tr'$. Let $H'=A^TA$. For a given $t \in \{1, \dots, n\}$,
the corresponding CTA applied to the normal equation is
\begin{equation}
\widehat F_t(\widehat r')= \widehat r' - \sum_{i=1}^t \widehat \alpha_{t,i}(\widehat r') {H'}^i \widehat r',
\end{equation}
where $\widehat \alpha_t(\widehat r') = (\widehat \alpha_{t,1}(\widehat r'), \dots, \widehat \alpha_{t,t}(\widehat r'))^T$ is the solution to the {\it auxiliary equation with respect the normal equation},
$\widehat M_t(\widehat r') \widehat \alpha = \widehat \beta_t(\widehat r')$.
Note that using the definition of $H'$ and $H$, the $ij$ entry of $\widehat M_t(\widehat r')$ satisfies  $\widehat r'^{T} {H'}^{i+j}\widehat r'={r'}^TA{H'}^{i+j}A^Tr'={r'}^TH^{i+j+1}r'$. Also, $\widehat \beta_{t,i}(\widehat r')= {r'}^TA{H'}^{i} A^Tr'={r'}^TH^{i+1}r'$.
Using this and the previously developed formula for iterates of $F_t$ on residual and approximate solution for the system $Ax=b$, see (\ref{FTR}) and the case of symmetric PSD matrix in (\ref{thmeq4gen}),  the theorem below gives the corresponding formulas for the normal equation:
\begin{theorem} Given the residual $r'=b-Ax'$ and  the least-squares residual $\widehat r'=A^Tr'$, let $H=AA^T$.  The corresponding residual and approximate solutions with respect to iteration of $\widehat F_t(\widehat r')$ are:
$$\widehat r''= \widehat F_t(\widehat r')= A^T \bigg( r' - \sum_{i=1}^t \widehat \alpha_{t,i}(r') H^i r' \bigg), \quad
x''=x'+ A^T\sum_{i=1}^t \widehat \alpha_{t,i}(r')  H^{i-1}r'. \qed$$
\end{theorem}
As an example, for $t=1$, the iteration of $F_1$ as applied to the norma equation at $\widehat r'= A^Tr'$ produces new least-squares residual  $\widehat r''
=\widehat r' - ({{r'}^TH^2r'}/{{r'}^TH^3r'})A^THr'$, $x'' =x' + \big ({{r'}^TH^2r'}/{{r'}^TH^3r'} \big ) \widehat r''$.
Thus, given that $Hr$ and $H^2r$ are computed, we can compute $F_1(r')$, as well as $\widehat F_1(\widehat r')$ with moderate computation. Note that the orbit of least-squares residuals is not merely the multiplication of the orbit of residuals by $A^T$.

\subsection{Space and Time Complexity of High-Order CTA} \label{sec4.6}

As before, given  $m \times n$ matrix $A$, let $H=AA^T$. If $A$ is symmetric PSD, $H$ can be taken to be $A$. For a given  $r \in \mathbb{R}^m$ and any $t \in \{0, 1, \dots m\}$, let  $C_t(r)$ be  the $m \times (2t+1)$ matrix with columns $H^ir$, $i=0, \dots, t$; also, for any $t \in \{1, \dots m\}$,
let $M_t(r)$ be the $t \times t$ coefficient matrix (see \ref{eqcoeff}):
\begin{equation}
C_t(r)=
\begin{pmatrix}
r&Hr&H^2r& \ldots&H^{2t}r\\
\end{pmatrix}, \quad
M_t(r)=(\phi_{i+j}(r))=(r^TH^{i+j}r)
\end{equation}

\begin{proposition} Let $N$ be the number of nonzero entries of $A$.

(1) Given that $C_t(r)$ is computed, $C_{t+1}(r)$ can be computed in $2N$ multiplications when $H=A$ and $4N$ multiplication when $H=AA^T$.
Moreover, for any $t \in \{1, \dots, m\}$,
$C_{t}(r)$ can be computed in at most $2Nt$ multiplication when $H=A$, and at most $4Nt$ multiplications when $H=AA^T$.

(2)  Given that $C_t(r)$ and $M_t(r)$ are computed, $M_{t+1}(r)$ can be computed in at most $2N+2m$ multiplications when $H=A$, and at most $4N+2m$ multiplication when $H=AA^T$.
For any $t \in \{1, \dots, m-1\}$,
$M_{t}(r)$ can be computed in at most $(2N+2m)t$ multiplication when $H=A$, and at most $(4N+2m)t$ multiplications when $H=AA^T$.
\end{proposition}
\begin{proof}

(1): Given that $C_t(r)$ is computed, to compute $C_{t+1}(r)$, we only need to compute  the last two columns, $H^{2t+1}r$, $H^{2t+2}r$. This requires multiplying $H^{2t}r$ by $H$ to get $H^{2t+1}r$, followed by multiplying $H^{2t+1}r$ by $H$. This requires to at most $2N$ multiplications
when $H=A$, and at most $4N$ multiplications when $H=AA^T$. This also bounds the overall number of operations to compute $C_t(r)$.

(2):  Given $M_{t}(r)$, the only new entries of $M_{t+1}(r)$ to be computed are
$\phi_{2t+1}(r)$ and $\phi_{2t+2}(r)$.  Since $C_t(r)$ is computed, we only need to compute  the last two columns of $C_{t+1}(r)$,  then their inner products with $r$. This requires $2m$ multiplications.  These imply the claimed complexity. From this we conclude the overall number of multiplications to compute $M_t(r)$.
\end{proof}

\begin{proposition}
Suppose the $LU$ factorization of $M_t(r)$ is at hand. Then, the $LU$ factorization of $M_{t+1}(r)$ can be computed in $O(t^2)$ operations.
The computation of the solution of the auxiliary equation, $M_t(r) \alpha = \beta_t(r)$, can be obtained in $O(m^3)$ operations.
\end{proposition}

\begin{proof}
Given the $LU$ factorization of $M_{t}(r)$, we need to process the last row of $M_{t+1}(r)$ to compute the new $LU$ factorization. This process overall requires $O(t^2)$ operations. With the $LU$ factorization of $M_t(r)$ in hand, computing the solution $\alpha_t(r)$ for the auxiliary equation can be achieved in $O(t^2)$ operations. Summing over $t=1, \dots, m$ to compute all solutions $\alpha_t(r)$, it takes $\sum_{t=1}^m O(t^2)=O(m^3)$ operations.
\end{proof}

Based on the proposition and considering that the computation of $F_t(r)=r - \sum_{i=1}^t \alpha_{t,i}(r)H^ir$ involves the calculation of $C_t(r)$ and $M_t(r)$, we can present the following complexity theorem.

\begin{theorem}  Given $r \in \mathbb{R}^m$, we have:

(i) Each iteration $r_{k+1}= F_t(r_{k})$ can be carried out in $O(Nt+t^3)$ operations, with the required space being that of $C_t(r_{k})$ and $M_t(r_k)$.

(ii) Let $s$ be the degree of the minimal polynomial of $r$ with respect to $H$. Then $\{F_1(r), \dots, F_s(r)\}$ can be computed in $O(Ns+s^3)$ operations, with the required space being that of $C_s(r)$ and $M_s(r)$. \qed
\end{theorem}

\begin{remark}
It is possible to economize on the computation of some iterations of $F_t$ as follows: Suppose we have computed $r' = F_t(r)$. To compute the next iteration, $r'' = F_t(r')$, we have to compute $C_t(r')$ and $M_t(r')$. Instead, we can use $C_t(r)$ but approximate the entries of $M_t(r')$ using the Taylor approximation:
$$\phi_j(r')  \approx \phi_j(r) + (r'-r)^T \nabla \phi_j(r) = \phi_j(r) + 2(r'-r)^T H^{j} r= 2r'^TH^jr -r^TH^jr= (2r'-r)^TH^jr.$$
But $r'^TH^jr$ can be computed in $O(m)$ operations, allowing us to approximate the entries of $M_t(r')$ in $O(tm)$. Similarly, the corresponding approximation to $\alpha_t(r)$ can be computed in $O(t^3)$ operations. Consequently, rather than computing $r''$ in $O(Nt+t^3)$ operations, we can approximate it using $O(mt+t^3)$ operations. This approximation can continue as long as the residuals decrease sufficiently.
\end{remark}

\subsection{Sample Computational Results with CTA} \label{sec4.7}
In this subsection, we present computational results using CTA with a small parameter $t \leq 5$. Given a consistent square system $Ax=b$, we start with an initial residual, $r_0$, and proceed to evaluate $F_1(r_0)$, followed by $F_2 \circ F_1(r_0)$, and so on, up to $F_5 \circ \cdots \circ F_1(r_0)$. Furthermore, we have found it beneficial to use the last residual and cycle back to $F_1$. In other words, given the last residual, denoted by $r_5$, we generate the next residual as $F_4(r_5)$, and so forth. We calculate both the total number of iterations and the time taken to compute a solution within a specified level of precision.

First, using CTA, we solved systems where the matrix $A$ is of two types: symmetric positive definite (PD) and positive semidefinite (PSD), with dimensions $n=100, 500,$ and $1000$. Additionally, we employed MATLAB to solve the same systems using two widely used methods: the Conjugate Gradient (CG) and Generalized Minimal Residual (GMRES) methods.

For GMRES, we utilized the MATLAB function `gmres(A,b,5)' with a restart parameter set to $5$, making it directly comparable with $t \leq 5$ in CTA. The maximum iterations parameter was automatically determined by the GMRES implementation. We did not directly compare CTA with the MINRES method, as MATLAB's use of GMRES apparently already accounts for MINRES when the underlying matrix is symmetric. Furthermore, our computational results with symmetric PD and PSD matrices allowed us to make meaningful comparisons with CG, a method specifically designed for symmetric PD matrices.

Additionally, we solved other types of square linear systems with matrices of dimensions $500$ and $1000$,  employing both CTA and GMRES. When solving these via GMRES,  we utilize the given matrix $A$. However, with CTA we operate with $H=AA^T$. In these cases, we employed incomplete LU factorization as preprocessing.
It's worth noting that the statistics related to preconditioning are not included in the tables. The matrices we used in our experiments were generated  using Lotkin \cite{lotkin}, which consists of non-symmetric and ill-conditioned matrices with small, negative eigenvalues, as well as matrices from Dorr \cite{dorr}, which are diagonally dominant M-matrices that may be ill-conditioned but invertible. For MVMODE matrices, see Day and Dongarra \cite{mvmnep}.  More details can be found in the NEP collection titled MVMODE.

We have provided tables containing runtime and the number of iterations required to achieve a specific level of precision for matrices of dimensions $500 \times 500$ and $1000 \times 1000$. These tables are presented as follows:

\noindent $\bullet$ Table \ref{pd_fr1}: Results for positive definite (PD) matrices of dimensions $500, 1000$.

\noindent $\bullet$     Table \ref{pd_fr2}: Results for positive semidefinite (PSD) matrices of dimensions $500, 1000$.

\noindent $\bullet$    Table \ref{pd_fr3}: Results for three other types of square matrices of dimensions $500 \times  500$.

\noindent $\bullet$     Table \ref{pd_fr4}: Results for the three other types of square matrices of dimensions $1000 \times 1000$.

In these tables, the `quality of solution' is measured by the norm of the residual. Please note that in the last two tables, `NC' stands for `Not Converging.'

For PD and PSD matrices, Figures \ref{FigPD} and \ref{FigPSD} illustrate the number of iterations of CTA, CG and GMRES to achieve a relative residual error within various accuracies up to $10^{-12}$ for matrix dimensions $100$ and $1000$. The relative residual is calculated as $\Vert b- Ax_k \Vert/\Vert b \Vert$.

From these figures, it's evident that CTA performs at least as effectively as CG and GMRES for relative residual errors up to $10^{-6}$ and surpasses these algorithms when higher precision is required. Remarkably, this relative performance remains consistent across different matrix dimensions.

In the case of general matrices where preconditioning was applied, the number of iterations for CTA and GMRES are closely matched, making it less necessary to include such figures. Nevertheless, CTA requires fewer iterations, and its advantages are likely to become more pronounced as matrix dimensions increase.

\begin{table}[!htb]
\centering
\begin{tabular}{|ccccccc|}
\hline
\multicolumn{1}{|l|}{PD Matrices} & \multicolumn{2}{|l|}{Runtime(sec)} &
\multicolumn{2}{|l|}{No. of Iterations} & \multicolumn{2}{|l|}{Quality of Soln. }
\\ \hline
\multicolumn{1}{|c|}{$n$}   & \multicolumn{1}{c|}{500}  & \multicolumn{1}{c|}{1000} & \multicolumn{1}{c|}{500}  & \multicolumn{1}{c|}{1000} & \multicolumn{1}{c|}{500}  & \multicolumn{1}{c|}{1000}  \\ \hline
\multicolumn{1}{|c|}{CTA}  & \multicolumn{1}{c|}{0.33} & \multicolumn{1}{c|}{0.72} & \multicolumn{1}{c|}{460} & \multicolumn{1}{c|}{520} & \multicolumn{1}{c|}{1.0e-10} & \multicolumn{1}{c|}{1.0e-10} \\ \hline
\multicolumn{1}{|c|}{CG}  & \multicolumn{1}{c|}{0.38} & \multicolumn{1}{c|}{0.83} & \multicolumn{1}{c|}{410} & \multicolumn{1}{c|}{603} & \multicolumn{1}{c|}{1.0e-10} & \multicolumn{1}{c|}{1.0e-10} \\ \hline
\multicolumn{1}{|c|}{GMRES} & \multicolumn{1}{c|}{0.45} & \multicolumn{1}{c|}{0.91} & \multicolumn{1}{c|}{429} & \multicolumn{1}{c|}{936}  & \multicolumn{1}{c|}{1.0e-10} & \multicolumn{1}{c|}{1.01e-10} \\ \hline
\end{tabular}
\caption{Comparing  CTA, CG and GMRES for symmetric PD matrices of dimensions $n=500, 1000$.}
\label{pd_fr1}
\end{table}
\begin{table}[!htb]
\centering
\begin{tabular}{|ccccccc|}
\hline
\multicolumn{1}{|l|}{PSD Matrices} & \multicolumn{2}{|l|}{Runtime(sec)} &
\multicolumn{2}{|l|}{No. of Iterations} & \multicolumn{2}{|l|}{Quality of Soln. }
\\ \hline
\multicolumn{1}{|c|}{$n$}   & \multicolumn{1}{c|}{500}  & \multicolumn{1}{c|}{1000} & \multicolumn{1}{c|}{500}  & \multicolumn{1}{c|}{1000} & \multicolumn{1}{c|}{500}  & \multicolumn{1}{c|}{1000}  \\ \hline
\multicolumn{1}{|c|}{CTA}  & \multicolumn{1}{c|}{0.34} & \multicolumn{1}{c|}{0.74} & \multicolumn{1}{c|}{456} & \multicolumn{1}{c|}{501} & \multicolumn{1}{c|}{1.0e-10} & \multicolumn{1}{c|}{1.0e-10} \\ \hline
\multicolumn{1}{|c|}{CG}  & \multicolumn{1}{c|}{0.39} & \multicolumn{1}{c|}{0.90} & \multicolumn{1}{c|}{449} & \multicolumn{1}{c|}{642} & \multicolumn{1}{c|}{1.0e-10} & \multicolumn{1}{c|}{1.0e-10} \\ \hline
\multicolumn{1}{|c|}{GMRES} & \multicolumn{1}{c|}{0.56} & \multicolumn{1}{c|}{1.32} & \multicolumn{1}{c|}{488} & \multicolumn{1}{c|}{962}  & \multicolumn{1}{c|}{1.0e-10} & \multicolumn{1}{c|}{1.03e-10} \\ \hline
\end{tabular}
\caption{Comparing  CTA, CG and GMRES for symmetric PSD matrices of dimensions $n=500, 1000$.}
\label{pd_fr2}
\end{table}
\begin{table}[!htb]
\centering
\begin{tabular}{|ccccccc|}
\hline
\multicolumn{1}{|l|}{Matrix Type} & \multicolumn{2}{|l|}{Runtime(sec)} &
\multicolumn{2}{|l|}{No. of Iterations} & \multicolumn{2}{|l|}{Quality of Soln. }
\\ \hline
\multicolumn{1}{|c|}{$n=500$}   & \multicolumn{1}{c|}{CTA}  & \multicolumn{1}{c|}{GMRES} & \multicolumn{1}{c|}{CTA}  & \multicolumn{1}{c|}{GMRES} & \multicolumn{1}{c|}{CTA}  & \multicolumn{1}{c|}{GMRES}  \\ \hline
\multicolumn{1}{|c|}{Lotkin}  & \multicolumn{1}{c|}{0.63} & \multicolumn{1}{c|}{0.84} & \multicolumn{1}{c|}{34} & \multicolumn{1}{c|}{37} & \multicolumn{1}{c|}{1.0e-10} & \multicolumn{1}{c|}{1.0e-10} \\ \hline
\multicolumn{1}{|c|}{Dorr}  & \multicolumn{1}{c|}{0.64} & \multicolumn{1}{c|}{0.84} & \multicolumn{1}{c|}{34} & \multicolumn{1}{c|}{38} & \multicolumn{1}{c|}{1.0e-8} & \multicolumn{1}{c|}{1.0e-6} \\ \hline
\multicolumn{1}{|c|}{MVMODE} & \multicolumn{1}{c|}{0.64} & \multicolumn{1}{c|}{0.85} & \multicolumn{1}{c|}{35} & \multicolumn{1}{c|}{38}  & \multicolumn{1}{c|}{1.0e-10} & \multicolumn{1}{c|}{NC} \\ \hline
\end{tabular}
\caption{Comparing  CTA, GMRES for square matrices of dimensions $n=500$.}
\label{pd_fr3}
\end{table}
\begin{table}[!htb]
\centering
\begin{tabular}{|ccccccc|}
\hline
\multicolumn{1}{|l|}{Matrix Type} & \multicolumn{2}{|l|}{Runtime(sec)} &
\multicolumn{2}{|l|}{No. of Iterations} & \multicolumn{2}{|l|}{Quality of Soln. }
\\ \hline
\multicolumn{1}{|c|}{$n=1000$}   & \multicolumn{1}{c|}{CTA}  & \multicolumn{1}{c|}{GMRES} & \multicolumn{1}{c|}{CTA}  & \multicolumn{1}{c|}{GMRES} & \multicolumn{1}{c|}{CTA}  & \multicolumn{1}{c|}{GMRES}  \\ \hline
\multicolumn{1}{|c|}{Lotkin}  & \multicolumn{1}{c|}{0.94} & \multicolumn{1}{c|}{1.03} & \multicolumn{1}{c|}{48} & \multicolumn{1}{c|}{53} & \multicolumn{1}{c|}{1.0e-10} & \multicolumn{1}{c|}{1.0e-10} \\ \hline
\multicolumn{1}{|c|}{Dorr}  & \multicolumn{1}{c|}{0.96} & \multicolumn{1}{c|}{1.01} & \multicolumn{1}{c|}{49} & \multicolumn{1}{c|}{54} & \multicolumn{1}{c|}{1.0e-8} & \multicolumn{1}{c|}{1.0e-6} \\ \hline
\multicolumn{1}{|c|}{MVMODE} & \multicolumn{1}{c|}{0.96} & \multicolumn{1}{c|}{1.21} & \multicolumn{1}{c|}{49} & \multicolumn{1}{c|}{54}  & \multicolumn{1}{c|}{1.0e-10} & \multicolumn{1}{c|}{NC} \\ \hline
\end{tabular}
\caption{Comparing  CTA vs. GMRES for square matrices of dimension $n=1000$.}
\label{pd_fr4}
\end{table}
\begin{figure}[!htb]
\centering
\includegraphics[width=3.2 in]{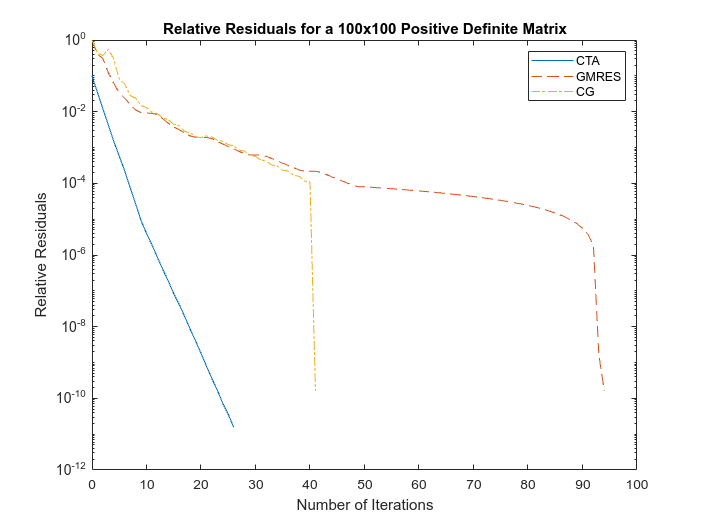}\
\includegraphics[width=3.2 in]{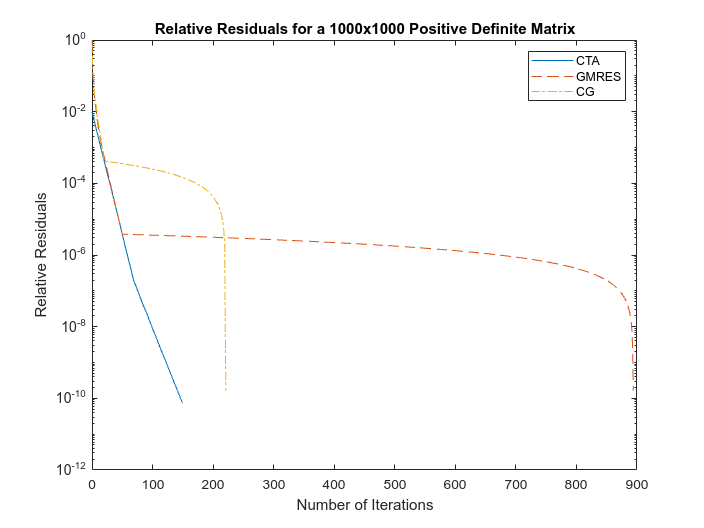}
\caption{Comparison of relative residuals for CTA, CG and GMRES for PD matrices, $n=100, 1000$.}
\label{FigPD}
\end{figure}
\begin{figure}[!htb]
\centering
\includegraphics[width=3.2 in]{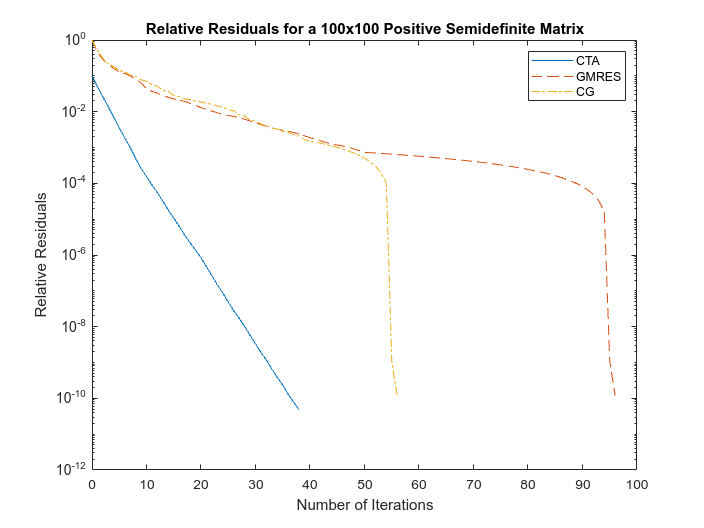}\
\includegraphics[width=3.2 in]{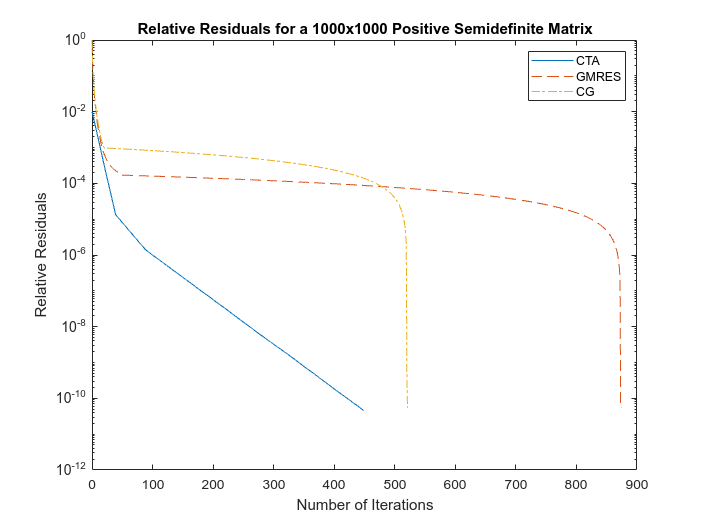}
\caption{Comparison of relative residuals for CTA, CG and GMRES for PSD matrices, $n=100, 1000$.}
\label{FigPSD}
\end{figure}

Clearly, in the cases of PD and PSD matrices, CTA outperforms CG and GMRES. Moreover, when applied to the three other types of matrices, CTA demonstrates superior performance compared to GMRES. These results lead us to believe that CTA is a highly competitive alternative to state-of-the-art linear solvers. For more comprehensive computational results with CTA and TA, as well as comparisons with state-of-the-art algorithms, we refer the reader to \cite{kallau23} and its forthcoming revision. That article also includes links to codes for CTA and TA, allowing practitioners to conduct their own experiments and comparisons if desired.

\section*{Final Remarks}
In this article we have introduced the {\it Triangle Algorithm} (TA) and the class of {\it Centering Triangle Algorithm} (CTA)  for computing exact or approximate solution, as well as exact or approximate minimum-norm solution of a linear system $Ax=b$ or its {\it normal equation} $A^TAx=A^Tb$, where $A$ is an $m \times n$ matrix of arbitrary rank. The algorithms implicitly try to solve $AA^Tw=b$, the {\it normal equation of the second type}, so that the computed approximate or exact solutions to $Ax=b$ are in fact minimum-norm such solutions. However, when $A$ is symmetric PSD, we described special versions of the algorithms with improved complexity bounds.

TA  tests if $b$ lies in a given ellipsoid $E_{A, \rho}= \{Ax: \Vert x \Vert \leq \rho\}$ and when necessary,  systematically adjusts $\rho$ to compute $\varepsilon$-approximate solution to $Ax=b$ or its  normal equation. Next, by allowing each iterate of TA to become the center of a new ellipsoid of appropriate $\rho$, we derived {\it first-order} CTA, whose iterates are conveniently  generated via $F_1(r)=r-(r^THr/r^TH^2r)Hr$, wherein $H=AA^T$. However, when $A$ is symmetric PSD, $H$ can be taken to be $A$ itself.  More generally, we developed the CTA family, $\{F_t, t=1, \dots, m\}$. Letting $\kappa^+ \equiv \kappa^+(H)$ represent the ratio of largest to smallest positive eigenvalues of $H$, we proved when $Ax=b$ is consistent, regardless of invertibility of $H$,
starting with arbitrary $x_0=A^Tw_0$, in $k=O(\kappa^+ t^{-1} \ln \varepsilon^{-1})$ iterations of $F_t$, CTA computes $b_k=Ax_k$ satisfying $\Vert b-b_k \Vert \leq \varepsilon$.  Moreover, regardless of consistency of $Ax=b$,  in $k =O(\kappa^+ t^{-1}\varepsilon^{-1})$ iterations, $\Vert A^Tr_k\Vert= O(\sqrt{\varepsilon})$. However, since the normal equation is always consistent, CTA can be applied to the normal equation to compute in $O(\kappa^+(A^TA) t^{-1} \ln \varepsilon^{-1})$ iterations an  $\varepsilon$-approximate least-squares solution. This suggests given a general system, $Ax=b$, we apply CTA. If $r_k$ does not show sufficient improvement in a few successive iterations, we switch applying CTA to the normal equation, starting with  $A^Tr_k$. Since $\kappa^+(AA^T)=\kappa^+(A^TA)$, it follows that in $O(\kappa^+(AA^+) t^{-1} \ln \varepsilon^{-1})$ iterations of $F_t$, CTA is capable of computing an $\varepsilon$-approximate solution of $Ax=b$ when it is consistent, and when inconsistent, an $\varepsilon$-approximate solution of the normal equation. This makes CTA quite robust.  It's worth noting that in certain cases, such as overdetermined systems, applying CTA directly to the normal equation can be more advantageous. In any case, the cost of each iteration is $O(N)$, where $N$ is the number of nonzero entries of $A$.   When $H=AA^T$, its usage is implicit. Likewise, any usage of $A^TA$ is implicit.

Moreover, we have proven that given any residual $r_0$ having $s$ as the degree of its minimal polynomial with respect to $H$, the CTA family produces exact solution via the {\it point-wise orbit} $\{F_t(r_0)\}_{t=1}^s$, which can be computed in $O(Ns+s^3)$ operations.
Specifically, $Ax=b$ is solvable if and only if $F_{s}(r_0)=0$. Furthermore, exclusively $A^TAx=A^Tb$ is solvable if and only if $F_{s}(r_0) \neq 0$ and $A^T F_s(r_0)=0$. Although this convergence is of theoretical interest, it suggests that in practice, we can generate a truncated version of the point-wise sequence, stopping at small values of $t$. Then, we can repeat the process with the new residual.

Numerous strategies can be devised around CTA. For smaller values of $t$, we can employ a combination of truncated point-wise orbits followed by iterations of $F_t$. CTA members can be used individually or in combination. For example, the composition $F_2 \circ F_1$ can be applied as follows: If $\Vert F_1(r)\Vert$ sufficiently reduces $\Vert r \Vert$, we proceed with the iteration of $F_1$. Otherwise, we compute $F_2(F_1(r))$. This composite approach may be implemented up to larger values of $t$, bounded by a small constant. In an alternate strategy, starting with an initial residual $r_0$, we generate $r_1=F_1(r_0)$, $r_2=F_2(r_1)$, and so on, up to $r_t=F_t(r_{t-1})$. Subsequently, we replace $r_0$ with $r_t$ and repeat this process. When $Ax=b$ is consistent, this ensures that in each cycle, $\Vert r_t \Vert \leq \Vert r_0 \Vert [(\kappa^+-1)/(\kappa^++1)]^{1+2+\cdots+t}$.

As demonstrated here, the worst-case scenario of $F_1$ occurs when $r_0$ is a linear combination of two eigenvectors having distinct positive eigenvalues of $H$. However, this is the best case for $F_2$ since $F_2(r_0)=0$. We have established  iteration complexity bounds for each member of the CTA family. However, except for $t=1$, these bounds are not necessarily optimal. Therefore, in practice, we would expect the iterations of $F_t$ to be much more effective than the bounds would indicate. Deriving the tightest worst-case complexity is indeed an interesting but challenging problem, even for $t=2$. We would expect that when $H$ has few positive eigenvalues or when the minimal polynomial of $r_0$ with respect to $H$ has a small degree, the iterates will be very effective. This should also be the case when the eigenvalues of $H$ are clustered.

To compute the coefficients in $F_t(r)$, we need to solve the $t \times t$ auxiliary linear system. While for small $t$, we do not need to use any special techniques, we have demonstrated that the coefficient matrix of the auxiliary system is both positive semidefinite and positive. Consequently, it can be diagonally scaled into a doubly stochastic matrix, serving as a preconditioning step. Regardless, to solve the auxiliary system, especially for larger values of $t$, we could use CTA itself to approximate its solution. While the approximate solution to $Ax=b$ computed via CTA with $H=AA^T$ is necessarily an approximate minimum-norm solution, this is not necessarily the case when $A$ is symmetric PSD and $H=A$. The same applies when CTA is applied to the normal equation. However, we can use TA to compute approximate minimum-norm solutions, from given approximate solutions. In other words, TA and CTA can be used jointly to compute desired approximate minimum-norm solutions to $Ax=b$ when $H=A$ and such solutions to the normal equation. We also note that according to the complexity bound of CTA, letting $t=\sqrt{m}$, the bound on the number of iterations of $F_t$ becomes $O((\kappa^+/\sqrt{m}) \ln \varepsilon^{-1})$, with each iteration costing $O(\sqrt{m}N+m^{3/2})$ operations. This suggests we can reduce the theoretical bound on the number of iterations while doing more work per iteration. For matrices where $\kappa^+$ is smaller than $\sqrt{m}$ (a fair assumption in many practical situations, especially for large-scale matrices), takeing $t \approx \sqrt{\kappa^+}$, the bound on the number of iterations of CTA becomes $O(\sqrt{\kappa^+} \ln \varepsilon^{-1})$.

The article's theoretical findings introduce the CTA family as a dependable set of iteration functions specifically designed for addressing general linear systems of equations. This is a fundamental issue in both mathematics and computer science. The CTA family's significance lies in its ability to guarantee convergence properties, irrespective of the specific characteristics or properties of the coefficient matrix within the linear system being solved.

Furthermore, these theoretical results also point to new avenues for further research. The assurance of convergence provided by the CTA family prompts exploration into optimizing these algorithms further, examining their adaptability across various applications, or investigating their computational efficiency and convergence rates in diverse scenarios. Within this article, we've scrutinized these algorithms through the lens of infinite precision. Consequently, delving into a formal error analysis to comprehend the effects of finite precision arithmetic stands as a valuable subject for future research. Nevertheless, considering our computational experiments and comparisons with established algorithms, the practical performance of CTA seems notably satisfactory.

The computational experiments conducted in our article involved the utilization of CTA, assessing its performance in comparison to CG and GMRES, resulting in promising outcomes. In a separate publication, referenced as \cite{kallau23}, we present an extensive collection of computational findings concerning CTA, incorporating comparisons with state-of-the-art algorithms.

In addition to solving linear systems, both TA and CTA have other applications. For instance, an application stemming from the point-wise assessment of the CTA family involves computing factors of the minimal polynomial of a specified PSD matrix (refer to Remark \ref{minpolyfactor}). A significant application pertains to solving the linear programming feasibility problem, which is fundamental in optimization. In \cite{kallau23}, we also present computational findings related to LP feasibility problems.

We urge readers, researchers, and practitioners in numerical and scientific computing to evaluate the performance of the proposed algorithms from a computational perspective. They can accomplish this by conducting  experiments with their unique linear systems, utilizing either the implementations of CTA or TA as detailed in \cite{kallau23}, or by creating their own versions of these algorithms. As practitioners consistently conduct their experiments, these algorithms might exhibit their competitiveness among state-of-the-art linear solvers over time.

\section*{Acknowledgements} I thank Chun Lau for carrying out the computational results presented here.

\bibliographystyle{plain}
\bibliography{biblio4}
\end{document}